\pdfoutput=1

    \documentclass[11pt,twoside]{article}

    \usepackage[kerning,spacing]{microtype}
    \microtypecontext{spacing=nonfrench}

    \usepackage{graphicx}
    \usepackage{styleset}
    \usepackage{macros}
    \let\subsubsection\subparagraph

    \title  {Khovanov homology is an unknot-detector}

    \author {P. B. Kronheimer and T. S. Mrowka%
            \thanks{%
            The work of the first author
            was supported by the National Science Foundation through
            NSF grants DMS-0405271 and DMS-0904589. The work of
            the second author was supported by NSF grant DMS-0805841.}} 

    \address {Harvard University, Cambridge MA 02138 \\
              Massachusetts Institute of Technology, Cambridge MA 02139}       

\begin{document}

\maketitle

\begin{abstract}
    We prove that a knot is the unknot if and only if its reduced
    Khovanov cohomology has rank $1$. The proof has two steps. We show first
    that there is a spectral sequence beginning with the reduced
    Khovanov cohomology and abutting to a knot homology
    defined using singular instantons. We then show that the latter homology
    is isomorphic to the instanton Floer homology of
    the sutured knot complement: an invariant that is already known to
    detect the unknot.
\end{abstract}


\section{Introduction}

\subsection{Statement of results}

This paper explores a relationship between the Khovanov cohomology of a
knot, as defined in \cite{Khovanov}, and various homology theories
defined using Yang-Mills instantons, of which the archetype is Floer's
instanton homology of 3-manifolds \cite{Floer-instanton}.  A
consequence of this relationship is a proof that Khovanov cohomology
detects the unknot. (For related results, see \cite{Grigsby-Wehrli,
Hedden-2-cable, Hedden-Watson}).

\begin{theorem}\label{thm:main-theorem}
    A knot in $S^{3}$ is the unknot 
   if and only if its reduced Khovanov cohomology is $\Z$.
\end{theorem}

In \cite{KM-knot-singular}, the authors construct 
a Floer homology for knots
and links in $3$-manifolds using moduli spaces of connections with
singularities in codimension 2. (The locus of the singularity is 
essentially the link $K$, or $\R\times K$ in a cylindrical
$4$-manifold.) Several variations of this construction are already
considered in \cite{KM-knot-singular}, but we will introduce here one
more variation, which we call $\Inat(K)$. 
Our invariant $\Inat(K)$ is an invariant for unoriented links
$K\subset S^{3}$ with a marked point $x\in K$ and a preferred normal
vector $v$ to $K$ at $x$. The purpose of the normal vector is in
making the invariant functorial for link cobordisms: if $S
\subset [0,1]\times S^{3}$ is a link cobordism from $K_{1}$ to
$K_{0}$, not necessarily orientable, but equipped with a path $\gamma$
joining the respective basepoints and a section $v$ of the normal
bundle to $S$ along $\gamma$, then there is an induced map,
\[
     \Inat(K_{1}) \to \Inat(K_{0})
\]
that is well-defined up to an overall sign and satisfies a composition
law. (We will discuss what is needed to resolve the sign ambiguity in
section~\ref{subsec:i-orient}.) The definition is set up so that
$\Inat(K)=\Z$ when $K$ is the unknot. We will refer to this homology
theory as the reduced singular instanton knot homology of $K$. (There
is also an unreduced version which we call $\Isharp(K)$ and which can
be obtained by applying $\Inat$ to the union of $K$ with an extra
unknotted, unlinked component.)  The
definitions can be extended by replacing $S^{3}$ with an arbitrary closed,
oriented $3$-manifold $Y$. The invariants are then functorial for
suitable cobordisms of pairs.

 Our main result concerning $\Inat(K)$ is that it is related to 
reduced Khovanov cohomology by a spectral
sequence. The model for this result is a closely-related theorem due
to Ozsv\'ath and Szab\'o \cite{OS-double-covers} concerning the
Heegaard Floer homology, with $\Z/2$ coefficients, 
of a branched double cover of $S^{3}$. There
is a counterpart to the result of \cite{OS-double-covers} in the
context of Seiberg-Witten gauge theory, due to Bloom \cite{Bloom}.

\begin{proposition}\label{prop:spectral}
    With $\Z$ coefficients, there is a spectral sequence whose
    $E_{2}$ term is the reduced Khovanov cohomology, $\khr(\bar{K})$, of
    the  mirror image knot $\bar K$, and
    which abuts to the reduced singular instanton homology 
    $\Inat(K)$.
\end{proposition}

As an immediate corollary, we have:

\begin{corollary}
\label{cor:lower-bound}
    The rank of the reduced Khovanov cohomology $\khr(K)$ is at least as
    large as the rank of  $\Inat(K)$. \qed
\end{corollary}

To prove Theorem~\ref{thm:main-theorem}, it will therefore suffice to
show that $\Inat(K)$ has rank bigger than $1$ for non-trivial
knots. This will be done by relating $\Inat(K)$ to a knot homology
that was constructed from a different point of view (without singular
instantons) by Floer in \cite{Floer-Durham}. Floer's knot homology was
revisited by the authors in \cite{KM-sutures}, where it appears as an
invariant $\KHI(K)$ of knots in $S^{3}$. (There is a slight difference
between $\KHI(K)$ and Floer's original
version, in that the latter leads to a group with twice the rank). It
is defined using $\SU(2)$ gauge theory on a closed $3$-manifold
obtained from the knot complement. The construction of $\KHI(K)$ in
\cite{KM-sutures} was motivated by Juh\'asz's work on sutured
manifolds in the setting of Heegaard Floer theory
\cite{Juhasz-1,Juhasz-2}: in the context of sutured manifolds,
$\KHI(K)$ can be defined as the instanton Floer homology of the
sutured 3-manifold obtained from the knot complement by placing two
meridional sutures on the torus boundary. It is defined in
\cite{KM-sutures} using complex coefficients for convenience, but one
can just as well use $\Q$ or $\Z[1/2]$. The authors establish in
\cite{KM-sutures} that $\KHI(K)$ has rank larger than $1$ if $K$ is
non-trivial.  The proof of Theorem~\ref{thm:main-theorem} is therefore
completed by the following proposition (whose proof turns out to be a
rather straightforward application of the excision property of
instanton Floer homology).

\begin{proposition}\label{prop:singular-to-sutured}
    With $\Q$ coefficients, there is an isomorphism between the
    singular instanton homology
    $\Inat(K;\Q)$ and the sutured Floer homology of the knot
    complement, $\KHI(K;\Q)$.
\end{proposition}

\begin{remark}
    We will see later in this paper that one can define a version
    of $\KHI(K)$ over $\Z$. The above proposition can then be
    reformulated as an isomorphism over $\Z$ between $\Inat(K)$ and $\KHI(K)$.
\end{remark}

Corollary~\ref{cor:lower-bound} and
Proposition~\ref{prop:singular-to-sutured} yield other lower bounds on
the rank of the Khovanov cohomology. For example, it is shown in
\cite{KM-alexander} that the Alexander polynomial of a knot can be
obtained as the graded Euler characteristic for a certain
decomposition of $\KHI(K)$, so we can deduce:

\begin{corollary}
    The rank of the reduced Khovanov cohomology $\khr(K)$ is bounded
    below by the sum of the absolute values of the coefficients of the
    Alexander polynomial of $K$. \qed
\end{corollary}

For alternating (and more generally, quasi-alternating) knots and
links, it is known that the rank of the reduced Khovanov cohomology
(over $\Q$ or over the field of 2 elements) is equal to the lower
bound which the above corollary provides
\cite{Lee, Manolescu-Ozsvath}. Furthermore, that lower bound is simply
the absolute value of the determinant in this case. 
We therefore deduce also:

\begin{corollary}
    When $K$ quasi-alternating,  the spectral sequence from $\khr(K)$ to
    $\Inat(K)$ has no non-zero differentials after the $E_{1}$ page, 
    over $\Q$ or $\Z/2$. In particular, the total rank of
    $\Inat(K)$ is equal to the absolute value of the determinant of
    $K$. \qed
\end{corollary}

The group $\KHI(K)$ closely resembles the ``hat'' version of Heegaard
knot homology, $\widehat{\HF}(K)$, defined in \cite{OS-knot-hf} and
\cite{Rasmussen-thesis}: one can perhaps think of $\KHI(K)$ as the
``instanton'' counterpart of the ``Heegaard'' group
$\widehat{\HF}(K)$.  The present paper provides a spectral sequence
from $\khr(K)$ to $\KHI(K)$, but at the time of writing it is not
known if there is a similar spectral sequence from $\khr(K)$ to
$\widehat{\HF}(K)$ for classical knots. 
This was a question raised by Rasmussen in
\cite{Rasmussen-knot-homologies}, motivated by observed similarities
between reduced Khovanov cohomology and Heegaard Floer homology.
There are results in the direction of providing such a spectral
sequence in \cite{Manolescu-skein}, but the problem remains open.

\subsection{Outline}

Section~\ref{sec:sing-inst-non} provides the framework for the
definition of the invariant $\Inat(K)$ by discussing instantons on
$4$-manifolds $X$ with codimension-2 singularities along an embedded
surface $\Sigma$. This is material that
derives from the authors' earlier work \cite{KM-gtes-I}, and it was
developed further for arbitrary structure groups in
\cite{KM-knot-singular}. In this paper we work only with the structure
group $\SU(2)$ or $\PSU(2)$, but we extend the previous framework in
two ways. First, in the previous development, the locus $\Sigma$ of the
singularity was always taken to be orientable. This condition can be
dropped, and we will be considering non-orientable surfaces. Second,
the previous expositions always assumed that the bundle which carried
the singular connection had an extension across $\Sigma$ to a bundle
on all of $X$ (even though the connection did not extend across the
singularity). This condition can also be relaxed. A simple example of
such a situation, in dimension 3, arises from a
$2$-component link $K$ in $S^{3}$: the complement of the link has
non-trivial second cohomology and there is therefore a $\PSU(2)$
bundle on the link complement that does not extend across the
link. The second Stiefel-Whitney class of this bundle on
$S^{3}\sminus K$ is dual to an
arc $\omega$ running from one component of $K$ to the other.

Section~\ref{sec:sing-inst-floer} uses the framework from
section~\ref{sec:sing-inst-non} to define an invariant $I^{\omega}(Y,K)$
for suitable links $K$ in $3$-manifolds $Y$. The label $\omega$ is a
choice of representative for the dual of $w_{2}$ for a chosen
$\PSU(2)$ bundle on $Y\sminus K$: it consists of a union of circles
in $S^{3}\sminus K$ and arcs joining components of $K$. The
invariant $\Inat(K)$ for a classical knot or link $K$ arises from this
more general construction as follows. Given $K$ with a framed
basepoint, we form a new link,
\[
           K^{\natural} = K \amalg L
\]
where $L$ is the oriented boundary of a small disk, normal to $K$ at
the basepoint. We take $\omega$ to be an arc joining the basepoint of $K$
to a point on $L$: a radius of the disk in the direction of  the
vector of the framing. We then define
\[
   \Inat(K) = I^{\omega}(S^{3}, K^{\natural}).
\]
This construction and related matters are described in more detail in 
section~\ref{sec:classical}.

Section~\ref{sec:an-excis-argum} deals with Floer's excision
theorem for Floer homology, as it applies in the context of this
paper. In order to work with integer coefficients, some extra work is
needed to deal with orientations and $\PSU(2)$ gauge transformations
that do not lift to $\SU(2)$. The excision property is used to prove
the relationship between $\Inat(K)$ and $\KHI(K)$ asserted in 
Proposition~\ref{prop:singular-to-sutured}, and also to establish a
multiplicative property of $\Inat(K)$ for split links $K$.

Sections~\ref{sec:cubes}, \ref{sec:proof-prop-blocks} and
\ref{sec:unlinks-e-2} are devoted to the proof of
Proposition~\ref{prop:spectral}, concerning the spectral sequence. 
The first part of the proof is to
show that $\Inat(K)$ can be computed from a  ``cube of
resolutions''. 
 This essential idea comes from Ozsv\'ath and Szab\'o's
paper on double-covers \cite{OS-double-covers}, and is closely related
to Floer's surgery long exact sequence for instanton Floer
homology. It can be seen as an extension of a more straightforward
property of $\Inat(K)$, namely that it has a long exact sequence for
the unoriented skein relation: that is, if $K_{2}$, $K_{1}$ and
$K_{0}$ are knots or links differing at one crossing as shown in
Figure~\ref{fig:Tetrahedra-skein}, then the corresponding groups
$\Inat(K_{i})$ form a long exact sequence in which the maps arise from
simple cobordisms between the three links.
 The cube of resolutions provides a spectral sequence
abutting to $\Inat(K)$. Section~\ref{sec:unlinks-e-2} establishes that
the $E_{2}$ term of this spectral sequence is the reduced Khovanov
cohomology of $\bar{K}$.

\begin{remark}
    Although it is often called Khovanov \emph{homology}, Khovanov's
    invariant is a cohomology theory, and we will follow \cite{Khovanov}
    in referring to it as such. The groups we write as $\kh(K)$, or
    in their bigraded version as $\kh^{i,j}(K)$, are the groups
    named  $\mathcal{H}^{i,j}(K)$ in \cite{Khovanov}.
\end{remark}

\subparagraph*{Acknowledgments.} The authors wish to thank Ciprian
Manolescu and Jake Rasmussen for thoughts and conversations that were
important in the development of these ideas. Comments from Ivan Smith
led to an understanding that the version of
Proposition~\ref{prop:unlink-canonical} in an earlier draft was not
stated in a sharp enough form.  Finally, this paper  owes a great deal to
the work of Peter Ozsv\'ath and Zolt\'an Szab\'o, particularly their
paper \cite{OS-double-covers}.

\section{Singular instantons and non-orientable surfaces}
\label{sec:sing-inst-non}

\subsection{Motivation}

In \cite{KM-gtes-I, KM-gtes-II}, the authors considered
$4$-dimensional connections defined on the complement of an embedded
surface and having a singularity along the surface. The basic model is
an $\SU(2)$ connection in the trivial bundle on $\R^{4}\sminus\R^{2}$
given by the connection matrix
\[
        i
        \begin{pmatrix}
            \alpha & 0 \\
             0 & -\alpha
        \end{pmatrix} \,d\theta
\]
where $\alpha$ is some parameter in the interval $(0,1/2)$. Given a
closed, embedded surface $\Sigma$ in an oriented $4$-manifold $X$, one
can study anti-self-dual connections on $X\sminus \Sigma$ whose
behavior near $\Sigma$ is modeled on this example. The moduli spaces
of such connections were defined and studied in \cite{KM-gtes-I}. In
\cite{KM-knot-singular}, a corresponding Floer homology theory was
constructed for knots in $3$-manifolds; but for the Floer homology
theory it was necessary to take $\alpha=1/4$.

We will now take this up again, but with a slightly more general setup
than in the previous paper.  We will continue to take $\alpha=1/4$,
and the local model for the singularities of our connections will be
the same: only the global topology will be more general. First of all,
we will allow our embedded surface $\Sigma$ to be
non-orientable. Second, we will not require that the bundle on the
complement of the surface admits any extension, globally, to a bundle
on the 4-manifold: specifically, we will consider $\PU(2)$ bundles on
$X^{4}\sminus\Sigma^{2}$ whose second Stiefel-Whitney class is
allowed to be non-zero on some torus $\gamma\times S^{1}$, where $\gamma$ is a
closed curve on $\Sigma$ with orientable normal and the $S^{1}$ factor
is the unit normal directions to $\Sigma\subset X$ along this
curve. It turns out that the constructions of \cite{KM-gtes-I} and
\cite{KM-knot-singular} carry over with little difficulty, as long as
we take $\alpha=1/4$ from the beginning.

Our first task will be to carefully describe the models for the sort of
singular connections we will study. A singular $\PU(2)$ connection on
$X\sminus\Sigma$ of the sort
we are concerned with will naturally give rise to $2$-fold covering
space $\pi : \Sigma_{\Delta}\to \Sigma$. To understand why this is
so, consider the simplest local model: a flat $\PU(2)$ connection $A_{1}$ on
$B^{4}\sminus B^{2}$ whose holonomy around the linking circles has
order $2$. The eigenspaces of the holonomy decompose the associated
$\R^{3}$ bundle as $\xi \oplus Q$, where $\xi$ is a trivial rank-1
bundle and $Q$ is a $2$-plane bundle, the $-1$ eigenspace. Suppose we
wish to extend $Q$ from $B^{4}\sminus B^{2}$ to all of $B^{4}$. To
do this, we construct a new connection $A_{0}$ as
\[
      A_{0} = A_{1} - \frac{1}{4} \bi \, d\theta
\]
where $\theta$ is an angular coordinate in the $2$-planes normal to
$B^{2}$ and $\bi$ is a section of the adjoint bundle which annihilates
$\xi$ and has square $-1$ on $Q$. Then $A_{0}$ is a flat connection in
the same bundle, with trivial holonomy, and it determines canonically
an extension of the bundle across $B^{2}$. In this process, there is a
choice of sign: both $d\theta$ and $\bi$ depend on a choice
(orientations of $B^{2}$ and $Q$ respectively). If we change the sign
of $\bi d\theta$ then we obtain a different extension of the bundle. 

In the global setting, when we have a surface $\Sigma \subset X$, we
will have two choices of extension of our $\PU(2)$ bundle at each
point of $\Sigma$. Globally, this will determine a double-cover
(possibly trivial), 
\[
 \pi : \Sigma_{\Delta} \to \Sigma.
\] 
It will be
convenient to think of the two different extensions of the $\PU(2)$
bundle as being defined simultaneously on a non-Hausdorff space
$X_{\Delta}$. This space comes with a projection $\pi: X_{\Delta}
\to X$ whose fibers are a single point of each point of $X\sminus
\Sigma$ and whose restriction to $\pi^{-1}(\Sigma)$ is the
double-covering $\Sigma_{\Delta}$.

\subsection{The topology of singular connections}
\label{subsec:sing-connections}

To set this up with some care, 
we begin with a closed, oriented, Riemannian 4-manifold $X$, a
smoothly embedded surface $\Sigma \subset X$. We identify a tubular
neighborhood 
$\nu$ of $\Sigma$  with the disk bundle of the
normal $2$-plane bundle $N_{\Sigma}\to\Sigma$. This identification
gives a tautological section $s$ of the pull-back of $N_{\Sigma}$ to $\nu$.
The section $s$ is non-zero over $\nu \sminus \sigma$, so on
$\nu\sminus\sigma$ we can consider the section 
\begin{equation}\label{eq:s1}
  s_{1}= s/|s|.
\end{equation}
On the other hand, let us choose any smooth connection in
$N_{\Sigma}\to\Sigma$ and pull it back to the bundle $N_{\Sigma}\to
\nu$. Calling this connection $\nabla$, we can then form the covariant
derivative $\nabla s_{1}$.
 We can identify the adjoint bundle of the $O(2)$ bundle $N_{\Sigma}$ as
$i \R_{o(\Sigma)}$, where $\R_{o(\Sigma)}$ is the real orientation
line bundle of $\Sigma$. So the derivative of $s_{1}$ can be written as
\[
        \nabla s_{1} = i\eta s_{1}
\]
for $\eta$ a $1$-form on $\nu\sminus \Sigma$ with values in
$\R_{o(\Sigma)}$. This $\eta$ is a global angular $1$-form on the
complement of $\Sigma$ in $\nu$.

Fix a 
local system $\Delta$ on $\Sigma$ with
structure group $\pm 1$, or equivalently a double-cover
$\pi: \Sigma_{\Delta}\to \Sigma$. This determines also a double-cover
$\pi:\tilde \nu_{\Delta}\to\nu$. We form a non-Hausdorff space $X_{\Delta}$
as an identification space of $X\sminus \Sigma$ and
$\tilde \nu_{\Delta}$, in which each $x \in \tilde \nu_{\Delta} \sminus
\Sigma_{\Delta}$ is identified with its image under $\pi$ in
$X\sminus\Sigma$. We write $\nu_{\Delta}\subset X_{\Delta}$ for
the (non-Hausdorff) image of the tubular neighborhood
$\tilde\nu_{\Delta}$.

The topological data describing the bundles in which our singular
connections live will be the following. We will have first a $\PU(2)$-bundle
$P_{\Delta}\to X_{\Delta}$. (This means that we have a bundle on
the disjoint union of $\tilde\nu_{\Delta}$ and $X\sminus\Sigma$
together with a bundle isomorphisms between the bundle on
$\tilde\nu_{\Lambda}\sminus\Sigma$ and pull-back of the bundle from
$X\sminus\Sigma$.) In addition, we will have a reduction of the
structure group of $P_{\Delta}|_{\nu_{\Delta}}$ to $O(2)$. We will
write $Q$ for the associated real $2$-plane bundle over
$\nu_{\Delta}$. 

The bundle $Q$ will be required to have a
very particular form.  To describe this, we start with $2$-plane
bundle $\tilde Q \to \Sigma_{\Delta}$ whose orientation bundle is
identified with the orientation bundle of $\Sigma_{\Delta}$:
\[
                 o(\tilde Q)
                 \stackrel{\cong}{\longrightarrow} o(\Sigma_{\Delta}).
\]
We can pull $\tilde Q$ back to $\tilde \nu_{\Sigma}$. In order to
create from this a bundle over the non-Hausdorff quotient
$\nu_{\Delta}$, we must give for each $x \in \tilde\nu_{\Delta}\sminus\Sigma_{\Delta}$ an
identification of the fibers,
\begin{equation}\label{eq:Q-to-Q}
    Q_{\tau(x)} \to Q_{x}   
\end{equation}
where $\tau$ is the covering transformation. Let us write
\begin{equation}\label{eq:hom-minus}
      \Hom^{-}(\tilde Q_{\tau(x)} , \tilde Q_{x})
\end{equation}
for the $2$-plane consisting of linear maps that are scalar multiples of
an orientation-reversing isometry (i.e. the complex-anti-linear maps if we
think of both $Q_{x}$ and $Q_{\tau(x)}$ as oriented). Like $\tilde Q_{x}$ and
$\tilde Q_{\tau(x)}$, the $2$-plane \eqref{eq:hom-minus} has its
orientation bundle canonically identified with $o(\Sigma)$: our
convention for doing this is to fix any vector in $\tilde Q_{\tau(x)}$
and use it to map  $  \Hom^{-}(\tilde Q_{\tau(x)} , \tilde Q_{x})$ to
$\tilde Q_{x}$. We will
give an identification \eqref{eq:Q-to-Q} by specifying an
orientation-preserving bundle isometry of $2$-plane bundles on $\Sigma_{\Delta}$,
\[
             \rho : N_{\Sigma_{\Delta}} \to   \Hom^{-}( \tau^{*}(\tilde Q) ,\tilde Q).
\]
This $\rho$ should satisfy
\[
           \rho(v) \rho(\tau(v)) = 1
\]
for a unit vector $v$ in $N_{\Sigma_{\Delta}}$. The identification
\eqref{eq:Q-to-Q} can then be given by $\rho(s_{1})$, where $s_{1}$ is
as in \eqref{eq:s1}. The existence of such a $\rho$ is a constraint on
$\tilde Q$.

To summarize, we make the following definition.

\begin{definition}\label{def:singular-bundle-data}
Given a pair $(X,\Sigma)$, by \emph{singular bundle data} on
$(X,\Sigma)$ we will mean a choice of the following items:
\begin{enumerate}
\item a double-cover $\Sigma_{\Delta}\to \Sigma$ and an associated
    non-Hausdorff space $X_{\Delta}$;
\item a principal $\PU(2)$-bundle $P_{\Delta}$ on $X_{\Delta}$;
\item a $2$-plane bundle $\tilde Q \to \Sigma_{\Delta}$ whose
    orientation bundle is identified with $o(\Sigma_{\Delta})$;
 \item an orientation-preserving bundle isometry
\[
   \rho : N_{\Sigma_{\Delta}} \to   \Hom^{-}( \tau^{*}(\tilde Q) ,\tilde Q)
\]
\item an identification, on the non-Hausdorff neighborhood
    $\nu_{\Delta}$,
    of the resulting quotient bundle $Q_{\Delta}$
    with an $O(2)$
    reduction of $P_{\Delta}|_{\nu_{\Delta}}$.
\end{enumerate}
\CloseDef
\end{definition}

\begin{remark}
When the conditions in this definition are fulfilled, the double-cover
$\Sigma_{\Delta}$ is in fact determined, up to isomorphism, by the
$\PU(2)$ bundle $P$ on $X\sminus \Sigma$. It is also the case that
not every double-cover of $\Sigma$ can arise as $\Sigma_{\Delta}$. We
shall return to these matters in the next subsection.
\end{remark}

\medskip

Given singular bundle data on $(X,\Sigma)$, we can now write down a model singular
connection. We start with a smooth connection $a_{0}$ in the bundle
$\tilde Q \to \Sigma_{\Delta}$, so chosen that the induced connection
in $ \Hom^{-} (\tau^{*}(\tilde Q) ,\tilde Q)$ coincides with the
connection $\nabla$ in $N_{\Sigma_{\Delta}}$ under the isomorphism
$\rho$. (Otherwise said, $\rho$ is parallel.) By pull-back, this
$a_{0}$ also determines a connection in $\tilde Q$ on
$\tilde\nu_{\Delta}$. On the deleted tubular neighborhood
$\tilde\nu_{\Delta}\sminus\Sigma_{\Delta}$, the bundles
$\tau^{*}(\tilde Q)$ and $\tilde Q$ are being identified by the
isometry $\rho(s_{1})$; but under this identification, the connection
$a_{0}$ is not preserved, because $s_{1}$ is not parallel. So $a_{0}$
does not by itself give rise to a connection over $\nu\sminus\Sigma$
downstairs. The covariant derivative of $s_{1}$ is $i\eta s_{1}$,
where $\eta$ is a $1$-form with values in $\R_{o(\Sigma)}$, or
equivalently in $\R_{o(\tilde Q)}$ on $\tilde\nu_{\Delta}$. We can
therefore form a new connection
\[
       \tilde a_{1} = a_{0} + \frac{ i }{2}\eta
\]
as a connection in $\tilde Q \to \tilde \nu_{\Delta}
\sminus\Sigma_{\Delta}$. With respect to this new connection, the
isometry $\rho(s_{1})$ is covariant-constant, so $\tilde a_{1}$
descends to a connection $a_{1}$ on the resulting bundle $Q \to
\nu\sminus\Sigma$ downstairs.

Since $Q$ is a reduction of the $\PU(2)$ bundle $P$ on $\nu
\sminus\Sigma$, the $O(2)$ connection $a_{1}$ gives us a $\PU(2)$
connection in $P$ there. Let us write $A_{1}$ for this $\PU(2)$
connection. We may extend $A_{1}$ in any way wish to a connection in
$P$ over all of $X\sminus \Sigma$.  If we pick a point $x$ in
$\Sigma_{\Delta}$, then a standard neighborhood of $x$ in
$X_{\Delta}$ is a $B^{4}$ meeting $\Sigma_{\Delta}$ in a standard
$B^{2}$. In such a neighborhood, the connection $A_{1}$ on
$B^{4}\sminus B^{2}$ can be written as
\[
                   A_{1} = A_{0} + \frac{1}{4}
                   \begin{pmatrix}
                       i & 0 \\ 
                       0 & -i
                   \end{pmatrix} \eta,
\]
where $A_{0}$ is a smooth $\PU(2)$ connection with reduction to
$O(2)$. Here are notation identifies the Lie algebra of $\PU(2)$ with
that of $\SU(2)$, and the element $i$ in $\mathrm{Lie}(O(2))$
corresponds to the element
\[
\frac{1}{2}
                   \begin{pmatrix}
                       i & 0 \\ 
                       0 & -i
                   \end{pmatrix}
\]
in $\su(2)$. Note that in this local description, the connection
$A_{0}$ depends in a significant way on our choice of $x \in
\Sigma_{\Delta}$, not just on the image of $x$ in $\Sigma$. Two
different points $x$ and $x'$ with the same image in $\Sigma$ give
rise to connections $A_{0}$ and $A_{0}'$ which differ by a term
\[
\frac{1}{2}
                   \begin{pmatrix}
                       i & 0 \\ 
                       0 & -i
                   \end{pmatrix} \eta.
\]

We will refer to any connection $A_{1}$ arising in this way as a
\emph{model singular connection}. Such an $A_{1}$ depends on a choice
of singular bundle data (Definition~\ref{def:singular-bundle-data}), a
choice of $a_{0}$ making $\rho$ parallel, and a choice of extension of
the resulting connection from the tubular neighborhood to all of
$X\sminus\Sigma$. The latter two choices are selections from certain
affine spaces of connections, so it is the singular bundle data that
is important here.

\subsection{Topological classification of singular bundle data}
\label{subsec:topol-class-sing}

When classifying bundles over $X_{\Delta}$ up to isomorphism, it is
helpful in the calculations to replace this non-Hausdorff space by a
Hausdorff space with the same weak homotopy type. We can construct
such a space, $X^{h}_{\Delta}$, as an identification space of the
disk bundle $\tilde \nu_{\Delta}$ and the complement $X\sminus
\mathrm{int}(\nu)$, glued together along $\partial \nu$ using the
$2$-to-$1$ map $\partial(\tilde  \nu_{\Delta}) \to \partial
\nu$. There is a map
\[
    \pi:  X^{h}_{\Delta} \to X
\]
which is $2$-to-$1$ over points of $\mathrm{int}(\nu)$ and $1$-to-$1$
elsewhere. The inverse image of $\nu$ under the map $\pi$ is a
$2$-sphere bundle
\[
          S^{2} \hookrightarrow D \to \Sigma.
\]
In the case that $\Delta$ is trivial, this $2$-sphere bundle $D$ is
the double of the tubular neighborhood $\nu$, and a choice of
trivialization of $\Delta$ determines an orientation of $D$. When $\Delta$ is
non-trivial, $D$ is not orientable: its orientation bundle is
$\Delta$. There is also an involution
\[
      t: X^{h}_{\Delta} \to X^{h}_{\Delta}
\]
with $\pi \circ t = \pi$ whose restriction to each $S^{2}$ in $D$ is
an  orientation-reversing map. 

In $H_{4}(X^{h}_{\Delta};\Q)$ there is a unique class $[X_{\Delta}]$ which is
invariant under the involution $t$ and has $\pi_{*}[X_{\Delta}] =
[X]$. This class is not integral: in terms of a triangulation of
$X_{\Delta}^{h}$, the class can be described as the sum of the
$4$-simplices belonging to $X^{h}_{\Delta}\sminus \tilde\nu_{\Delta}$,
plus half the sum of the simplices belonging to $D$, with all
simplices obtaining their orientation from the orientation of $X$.
In the case that the double-cover
$\Sigma_{\Delta} \to \Sigma$ is trivial, there is a different
fundamental class to consider. In this case, $X_{\Delta}$
is the union of two copies of $X$, identified on the complement of
$\Sigma$. A choice of trivialization of $\Delta$ picks out one of
these two copies: call it $X_{+} \subset X_{\Delta}$. The
fundamental class $[X_{+}]$ is an integral class in
$H_{4}(X_{\Delta})$: it can be expressed as
\[
  [X_{+}] = ([X_{\Delta}] + [D]/2)
\]
where $D$ is oriented using the trivialization of $\Delta$.
 A
full description of $H_{4}$ with $\Z$ coefficients is given by the
following, whose proof is omitted.

\begin{lemma}\label{lem:H4-XLambda}
    The group $H_{4}(X^{h}_{\Delta}; \Z)$ has rank $1 + s$, where $s$ is
    the number of components of $\Sigma$ on which $\Delta$ is
    trivial. If $\Delta$ is trivial and trivialized, 
   then  free generators are provided by (a) the fundamental
    classes of the components of the $2$-sphere bundle
    $D$, and (b) the integer class $[X_{+}]$ determined by the trivialization. 
     If $\Delta$ is
    non-trivial, then free generators are provided by (a) the
    fundamental classes of the orientable components of $D$, and (b)
    the integer class $2[X_{\Delta}]$. \qed
\end{lemma}

The top cohomology of $X_{\Delta}$ contains $2$-torsion if $\Delta$
is non-trivial:

\begin{lemma}\label{lem:coH4-XLambda}
     If $a$ is the number of 
    components of $\Sigma$ on which $\Delta$ is non-trivial, then the
    torsion subgroup of $H^{4}(X^{h}_{\Delta};\Z)$ is isomorphic
    to $(\Z/2)^{a-1}$ if $a\ge 2$ and is zero otherwise. To describe
    generators, let $ x_{1},\dots,  x_{a}$ be points in the
    different non-orientable components of $D$ 
    which map under $\pi$ to points 
     in $\mathrm{int}(\nu)$. Let $\xi_{i}$ be the image in
    $H^{4}(X^{h}_{\Delta})$ of a generator of $H^{4}(X^{h}_{\Delta},
    X^{h}_{\Delta}\sminus  x_{i})\cong \Z$, oriented so that
    $\langle\xi_{i}, [X_{\Delta}]\rangle=1$. Then generators for the torsion
    subgroup are the elements $\xi_{i}-\xi_{i+1}$, for
    $i=1,\dots,a-1$.
\end{lemma}

\begin{proof}
    This is also straightforward. The element $\xi_{i} - \xi_{i+1}$ is
    non-zero because it has non-zero pairing with the $\Z/2$
    fundamental class of the $i$'th component of $D$. On the other
    hand, $2\xi_{i}= 2\xi_{i+1}$, because both of these classes are
    equal to the pull-back by $\pi$ of the generator of $H^{4}(X;\Z)$.
\end{proof}

Because we wish to classify $\SO(3)$ bundles, we are also interested
in $H^{2}$ with $\Z/2$ coefficients.

\begin{lemma}\label{lem:H2-XLambda}
 The group $H^{2}(X^{h}_{\Delta};\Z/2)$ lies in an exact
 sequence
\[
     0 \to H^{2}(X;\Z/2) \stackrel{\pi^{*}}{\longrightarrow}
          H^{2}(X^{h}_{\Delta};\Z/2) 
            \stackrel{e}{\longrightarrow} (\Z/2)^{N}  \to H_{1}(X)           
\]
where $N$ is the number of components of $\Sigma$. The map
$e$ is the restriction map to 
\[
   \bigoplus_{i} H^{2}(S_{i}^{2};\Z/2)
\]
for a collection of fibers $S_{i}^{2}\subset D$, one from each component
of $D$. If $\lambda_{i}\in H_{1}(\Sigma)$ denotes the Poincar\'e dual
of $w_{1}(\Delta|_{\Sigma_{i}})$, then 
the last map is given by
\[
             (\epsilon_{1},\dots,\epsilon_{N}) \mapsto  
                  \sum_{i} \epsilon_{i} \lambda_{i}  \in H_{1}(X;\Z/2).
\]
\end{lemma}

\begin{proof}
With $\Z/2$ coefficients  understood, we have the following
commutative diagram, in which the vertical arrows from the bottom row are given by
$\pi^{*}$, the rows are exact sequences coming from Mayer-Vietoris,
and the middle column is a short exact sequence:
     \begin{align*}
    \xymatrix{
                           &          &   (\Z/2)^{N} & & \\
             \cdots \ar[r] &    H^{2}(X^{h}_{\Delta}) \ar[r]  &
                    H^{2}(X\sminus\Sigma) \oplus H^{2}(D) \ar[r] \ar[u] &
                    H^{2}(\partial \nu) \ar[r] & \cdots \\
                   \cdots \ar[r] &    H^{2}(X) \ar[r] \ar[u] &
                    H^{2}(X\sminus\Sigma) \oplus H^{2}(\nu) \ar[r] \ar[u]&
                    H^{2}(\partial \nu) \ar[r] \ar[u] & \cdots
	   }
           \end{align*}
The lemma follows from an examination of the diagram.
\end{proof}

Let us recall from \cite{Dold-Whitney} that $\SO(3)$ bundles $P$ on a
4-dimensional simplicial complex $Z$ can be classified as follows. First,
$P$ has a Stiefel-Whitney class $w_{2}(P)$, which can take on any
value in $H^{2}(Z;\Z/2)$. Second, the isomorphism classes of bundles
with a given $w_{2}(P)=w$ are acted on transitively by $H^{4}(Z;\Z)$. 
In the basic case of a class in
$H^{4}$ represented by the characteristic function of a single
oriented $4$-simplex, this action can be described as altering the
bundle on the interior of the simplex by forming a connect sum with an
$\SO(3)$ bundle $Q \to S^{4}$ with $p_{1}(Q)[S^{4}] = -4$ (i.e. the
$\SO(3)$ bundle associated to an $\SU(2)$ bundle with
$c_{2}=1$). Acting on a bundle $P$ by a class $z \in H^{4}(Z;\Z)$
alters $p_{1}(P)$ by $-4z$.  The action is may not be effective if the
cohomology of $Z$ has $2$-torsion: according to \cite{Dold-Whitney},
the kernel of the action is the subgroup
\[
                      \mathcal{T}^{4}(Z; w) \subset H^{4}(Z;\Z)
\]
given by
\[
       \mathcal{T}^{4}(Z;w) = \{ \, \beta(x)\cupprod \beta(x) + \beta(x
       \cupprod w_{2}(P)) \mid x \in H^{1}(Z;\Z/2) \, \}   ,
\]
where $\beta$ is the Bockstein homomorphism $H^{i}(Z;\Z/2) \to
H^{i+1}(Z;\Z)$. There are two corollaries to note concerning the class
$p_{1}(P)$ here. First, if $H^{4}(Z;\Z)$ contains classes $z$ with
$4z=0$ which do \emph{not} belong to $\mathcal{T}^{4}(Z;w)$, then
an $\SO(3)$ bundle $P \to Z$ with $w_{2}(P)=w$ is not determined up
to isomorphism by its Pontryagin class. Second, $p_{1}(P)$ is
determined by $w=w_{2}(P)$ to within a coset of the subgroup
consisting of multiples of $4$; on in other words, the image of
$\bar p_{1}(P)$  of $p_{1}(P)$ in $H^{4}(Z;\Z/4)$ is determined by $w_{2}(P)$. According
to \cite{Dold-Whitney} again, the determination is
\[
          \bar p_{1}(P) = \mathcal{P}( w_{2}(P)),
\]
where $\mathcal{P}$ is the Pontryagin square, $H^{2}(Z;\Z/2) \to
H^{4}(Z;\Z/4)$.

Now let us apply this discussion to $X_{\Delta}^{h}$ and the bundles
$P_{\Delta}$ arising from singular bundle data as in
Definition~\ref{def:singular-bundle-data}. The conditions of
Definition~\ref{def:singular-bundle-data} imply that
$w_{2}(P_{\Delta})$ is non-zero on every $2$-sphere fiber in
$D$. A first step in classifying such bundles $P_{\Delta}$ is to
classify the possible classes $w_{2}$ satisfying this condition. 
Referring to Lemma~\ref{lem:H2-XLambda}, we  obtain:

\begin{proposition}
    Let $\lambda \subset \Sigma$ be a $1$-cycle with $\Z/2$
    coefficients dual to
    $w_{1}(\Delta)$. A necessary and sufficient condition that there
    should exist a bundle $P_{\Delta} \to X_{\Delta}$ with $w_{2}$
    non-zero on every $2$-sphere fiber of $D$ is that $\lambda$
    represent the zero class in $H_{1}(X;\Z/2)$. When this condition
    holds the possible values for $w_{2}$ lie in a single coset of
    $H^{2}(X;\Z/2)$ in $H^{2}(X_{\Delta};\Z/2)$.  \qed
\end{proposition}

Let us fix $w_{2}$ and
consider the action of $H^{4}(X_{\Delta};\Z)$ on the isomorphism
classes of bundles $P_{\Delta}$. We orient all the
$4$-simplices of $X^{h}_{\Delta}$ using the map $\pi$ to $X$. We also
choose trivializations of $\Delta$ on all the components of $\Sigma$
on which it is trivial. We can then act by the class in $H^{4}$
represented by the characteristic function of a single oriented
$4$-simplex $\sigma$. We have the following cases, according to where
$\sigma$ lies.
\begin{enumerate}
\item If $\sigma$ is contained in $X\sminus
    \nu \subset X^{h}_{\Delta}$, then we refer to this operation as
    \emph{adding an instanton}.
\item If $\sigma$ is contained in $D$, then we have the following
    subcases:
    \begin{enumerate}
    \item if the component of $D$ is orientable, so that it is the
        double of $\nu$, and if $\sigma$ belongs to the distinguished
        copy of $\nu$ in $D$ picked out by our trivialization of
        $\Delta$, then we refer this operation as \emph{adding an
          anti-monopole} on the given component;
    \item if the component of $D$ is again orientable, but $\sigma$
        lies in the other copy of $\nu$, then we refer to the action
        of this class as \emph{adding a monopole};
     \item if the component of $D$ is not orientable then the
         characteristic function of any $4$-simplex in $D$ is
         cohomologous to any other, and we refer to this operation as
         adding a monopole.
    \end{enumerate}
\end{enumerate}

We have the following dependencies among these operations, stemming
from the fact that the corresponding classes in
$H^{4}(X^{h}_{\Delta};\Z)$ are equal:

\begin{enumerate}
\item
   \label{item:mono-plus-anti}
 adding a monopole and an anti-monopole to the same orientable
    component  is the same as adding an instanton;
\item 
  \label{item:mono-plus-mono}
  adding two monopoles to the same non-orientable component
    is the same as adding an instanton.
\end{enumerate}

Further dependence among these operations arises from the fact that
the action of the subgroup $\mathcal{T}^{4}(X_{\Delta};w_{2})$ is
trivial. The definition of $\mathcal{T}^{4}$ involves
$H^{1}(X_{\Delta}^{h};\Z/2)$, and the latter group is isomorphic to
$H^{1}(X;\Z/2)$ via $\pi^{*}$. Since $H^{4}(X;\Z)$ has no $2$-torsion,
the classes $\beta(x)\cupprod\beta(x)$ are zero. Calculation of the
term $\beta(x \cupprod w_{2})$ leads to the following interpretation:

\begin{enumerate}
\setcounter{enumi}{2}
\item
  \label{item:Dold-Whitney-extra}
  For any class $x$ in $H^{1}(X;\Z/2)$, let $n$ be the
    (necessarily even) number of components of $\Sigma$ on which
    $w_{1}(\Delta)\cupprod(x|_{\Sigma})$ is non-zero. Then the effect
    of adding in $n$ monopoles, one on each of these components, is
    the same as adding $n/2$ instantons.
\end{enumerate}

Our description so far gives a complete classification of $\SO(3)$ (or
$\PSU(2)$) bundles $P_{\Delta}\to X_{\Delta}$ having non-zero
$w_{2}$ on the $2$-sphere fibers of $D$. Classifying such bundles
turns out to be equivalent to classifying the (a priori more
elaborate) objects described as singular bundle data in
Definition~\ref{def:singular-bundle-data}.

\begin{proposition}
    For fixed $\Delta$, the forgetful map from the set of isomorphism
    classes of singular bundle data to the set of isomorphism classes
    of $\SO(3)$ bundles $P_{\Delta}$ on $X_{\Delta}$ is a bijection
    onto the isomorphism classes of bundles $P_{\Delta}$ with
    $w_{2}(P_{\Delta})$ odd on the $2$-sphere fibers in $D$.
\end{proposition}

\begin{proof}
    Let $P_{\Delta}\to X_{\Delta}$ be given. We must show that the
    restriction of $P_{\Delta}$ to the non-Hausdorff neighborhood
    $\nu_{\Delta}$ (or equivalently, on $X^{h}_{\Delta}$, the
    restriction of $P_{\Delta}$ to the $2$-sphere bundle $P$) admits
    a reduction to an $O(2)$ bundle $Q_{\Delta}$ of the special sort
    described in the definition. We must also show that this reduction
    is unique up to homotopy.
    
     Consider the reduction of $P_{\Delta}$ to the $2$-sphere bundle
     $D$. The space $D$ is the double of the tubular neighborhood, and
     as a sphere-bundle it therefore comes with a
     $2$-valued section. Fix a point $x$ in $\Sigma$ and consider the
     fiber $D_{x}$ over $x$, written as the union of $2$ disks $D^{+}$
     and $D^{-}$ whose centers are points $x^{+}$ and $x^{-}$ given by
     this $2$-valued section at the point $x$. The bundle
     $P_{\Delta}$ on $D_{x}$ can be described canonically up to
     homotopy as arising from a clutching function on the equatorial
     circle:
     \[
                  \gamma : (D^{+}\cap D^{-}) \to \Hom(
                  P_{\Delta}|_{x^{-}},   P_{\Delta}|_{x^{+}} ).
      \]
     The target space here is an isometric copy of $\SO(3)$. The loop
     $\gamma$ belongs to the non-trivial homotopy class by our
     assumption about $w_{2}$. The space
     of loops $\mathrm{Map}(S^{1}, \SO(3))$ in the non-zero homotopy
     class contains inside it the simple closed geodesics; and the
     inclusion of the space of these geodesics is an isomorphism on $\pi_{1}$ and
     $\pi_{2}$ and surjective map on $\pi_{3}$, as follows from a
     standard application of the Morse theory for geodesics. Since
     $\Sigma$ is $2$-dimensional, the classification of bundles
     $P_{\Delta}$ is therefore the same as the classification of
     bundles with the additional data of being constructed by
     clutching functions that are geodesics on each circle fiber
     $D^{+}\cap D^{-}$.  On the other hand, describing $P_{\Delta}$
     on $D$ by such clutching functions is equivalent to giving a
     reduction of $P_{\Delta}$ to an $O(2)$ bundle $Q_{\Delta}$
     arising in the way described in Definition~\ref{def:singular-bundle-data}.
\end{proof}

\subsection{Instanton and monopole numbers}
\label{subsec:inst-monop-numb}

Consider again for a moment the case that $\Delta$ is trivial and
trivialized, so that we have a standard copy $X_{+}$ of $X$ inside
$X_{\Delta}$. Inside $X_{+}$ is a preferred copy, $\Sigma_{+}$, of
the surface $\Sigma$. The orientation bundle of $Q_{\Delta}$ and the
orientation bundle of $\Sigma_{+}$ are canonically identified along
$\Sigma_{+}$. In this situation we can define two characteristic
numbers,
\[
\begin{aligned}
    k &= -\frac{1}{4} \langle p_{1}(P_{\Delta}) , [X_{+}]\rangle \\
    l &= - \frac{1}{2} \langle e(Q_{\Delta}), [\Sigma_{+}]\rangle.
\end{aligned}
\]
(Here the euler class $e(Q_{\Delta})$ is regarded as taking values in
the second cohomology of $\Sigma_{+}$ with coefficients twisted by the
orientation bundle.) We call these the \emph{instanton number} and the
\emph{monopole number} respectively: in the case that $\Sigma$ is
orientable, these definitions coincide with those from the authors
earlier papers \cite{KM-gtes-I, KM-knot-singular}. Since the $O(2)$ bundle
$Q_{\Delta}$ has degree $-1$ on the fibers of the $2$-sphere bundle
$D$, a short calculation allows us to express $l$ also in terms of the
Pontryagin class of the bundle $P_{\Delta}$: 
\begin{equation}\label{eq:l-formula}
        l = \frac{1}{4} p_{1}(P_{\Delta})[D] + \frac{1}{4} \Sigma\cdot\Sigma.
\end{equation}
(The term $\Sigma\cdot\Sigma$ is the ``self-intersection'' number of
$\Sigma$. Recall that this is a well-defined integer, even when
$\Sigma$ is non-orientable, as long as $X$ is oriented.)
When $\Delta$ is non-trivial, we cannot define either $k$ or $l$ in
this way. We can always evaluate the Pontryagin class on
$[X_{\Delta}]$ however.

Recall that for an $\SU(2)$ bundle on a closed $4$-manifold $X$, the
characteristic number $c_{2}(P)[X]$ can be computed by the Chern-Weil formula
\[
    \frac{1}{8 \pi^2 } \int_{X} \tr (F_{A}\wedge F_{A})
\]
where $A$ is any $\SU(2)$ connection and the trace is the usual trace
on $2$-by-$2$ complex matrices. For a $\PSU(2)$ bundle, the same
formula computes $-(1/4)p_{1}(P)[X]$. (Here we must identify the Lie
algebra of $\PSU(2)$ with that of $\SU(2)$ and define the trace form
accordingly.) 

Consider now a model singular connection $A$, as defined in
section~\ref{subsec:sing-connections}, corresponding to singular
bundle data $P_{\Delta} \to X_{\Delta}$. We wish to interpret the
Chern-Weil integral in terms of the Pontryagin class of
$P_{\Delta}$. We have:

\begin{proposition}
    For a model singular connection $A$ corresponding to singular
    bundle data $P_{\Delta}\to X_{\Delta}$, we have
\[
 \frac{1}{8 \pi^2 } \int_{X\sminus\Sigma} \tr (F_{A}\wedge F_{A})
    = -\frac{1}{4} p_{1}(P_{\Delta})[X_{\Delta}] + \frac{1}{16}
    \Sigma\cdot \Sigma,
\]
where $[X_{\Delta}] \in H_{4}(X_{\Delta};\Q)$ is again the
fundamental class that is invariant under the involution $t$ and has
 $\pi_{*}[X_{\Delta}]=[X]$.
\end{proposition}

We shall write $\kappa$ for this Chern-Weil integral:
\begin{equation}\label{eq:action}
          \kappa(A) = \frac{1}{8 \pi^2 } \int_{X\sminus\Sigma} \tr (F_{A}\wedge F_{A}).
\end{equation}

\begin{proof}[Proof of the proposition]
A formula for $\kappa$ in terms of characteristic classes was proved
in \cite{KM-gtes-I} under the additional conditions that $\Sigma$ was
orientable and $\Delta$ was trivial. 
In that case, after choosing a
trivialization of $\Delta$, the formula from \cite{KM-gtes-I} was
expressed in terms of the instanton and monopole numbers, $k$ and $l$,
as
\[
 \kappa(A) = k + \frac{1}{2} l  - \frac{1}{16} \Sigma\cdot\Sigma. 
\]
(In \cite{KM-gtes-I} the formula was written more generally for a
singular connection with a holonomy parameter $\alpha$. The formula
above is the special case $\alpha=1/4$.) Using the expression
\eqref{eq:l-formula} for the monopole number, this formula becomes
\[
    \kappa(A) = -\frac{1}{4}p_{1}(P_{\Delta})[X_{+}]  
                     +    \frac{1}{8} p_{1}(P_{\Delta})[D] 
                   + \frac{1}{16} \Sigma\cdot\Sigma,
\]
which coincides with the formula in the proposition, because
$[X_{\Delta}] = [X_{+}] - (1/2)[D]$. Thus the formula in the
proposition coincides with the formula from \cite{KM-gtes-I} in this
special case. The proof in \cite{KM-gtes-I} is essentially a local
calculation, so the result in the general case is the same.
\end{proof}

\subsection{The determinant-1 gauge group}

If $P\to X$ is an $\SO(3)$ bundle then there is a bundle
$G(P)\to X$ with fiber the group $\SU(2)$, associated to $P$ via the
adjoint action of $\SO(3)$ on $\SU(2)$. We refer to the sections of
$G(P)\to X$ as \emph{determinant-1 gauge transformation}, and we
write $\G(P)$ for the space of all such sections, the
\emph{determinant-1 gauge group}.  The gauge group $\G(P)$ acts on
the bundle $P$ by automorphisms, but the map
\[
     \G(P) \to \Aut(P)
\]
is not an isomorphism: its kernel is the two-element group $\{\pm 1\}$,
and its cokernel can be identified with $H^{1}(X;\Z/2)$.  

Suppose we are now given singular bundle data over the non-Hausdorff
space $X_{\Delta}$, represented in particular by an $\SO(3)$ bundle
$P_{\Delta}\to X_{\Delta}$.  We can consider the group of
determinant-1 gauge transformations,
$\G(P_{\Delta})$, in this context. Up to this point, we have not been
specific, but let us now consider simply \emph{continuous} sections of
$G(P_{\Delta})$ and denote the correspond group as $\Gtop$.

To understand $\Gtop$, consider a 4-dimensional ball neighborhood $U$ of a point
in $\Sigma_{\Delta}$. A section $g$ of the restriction of $G(P_{\Delta})$ to $U$ is
simply a map $U\to \SU(2)$, in an appropriate trivialization. Let $U'$
be the image of $U$ under the involution $t$ on $X_{\Delta}$, so that
$U\cap U' = U \sminus \Sigma_{\Delta}$. The section $g$ on $U$
determines a section $g'$ of the same bundle on $U'\sminus
\Sigma_{\Delta}$. In order for $g$ to extend to a
section of $G(P_{\Delta})$ on $U\cup U'$, it is necessary that $g'$
extend across $U' \cap \Sigma_{\Delta}$. In a trivialization, $g'$ is
an $\SU(2)$-valued function on $B^{4}\sminus B^{2}$ obtained by
applying a discontinuous gauge transformation to the function $g :
B^{4}\to \SU(2)$. It has the local form
\[
            g'(x) = \mathrm{ad}(u(\theta)) g(x)
\]
where $\theta$ is an angular coordinate in the normal plane to
$B^{2}\subset B^{4}$ and $u$ is the one-parameter subgroup of $\SU(2)$
that respects the reduction to an $O(2)$ bundle $Q_{\Delta}$ on
$U$. In order for $g'$ to extend continuously over $B^{2}$, it is
necessary and sufficient that $g(x)$ commutes with the one-parameter
subgroup $u(\theta)$ when $x\in B^{2}$: that is, $g(x)$ for $x\in U
\cap \Sigma_{\Delta}$ should
itself lie in the $S^{1}$ subgroup that preserves the subbundle
$Q_{\Delta}$ as well as its orientation.

To summarize, the bundle of groups $G(P_{\Delta}) \to X_{\Delta}$ has
a distinguished subbundle over $\Sigma_{\Delta}$,
\[
            H_{\Delta} \subset G(P_{\Delta}) \to \Sigma_{\Delta}
\]
whose fiber is the
group $S^{1}$; and the continuous sections of $G(P_{\Delta})$ take
values in this subbundle along $\Sigma_{\Delta}$. The local model is
an $\SU(2)$-valued function on $B^{4}$ constrained to take values in
$S^{1}$ on $B^{2}\subset B^{4}$.

The bundle $H_{\Delta}\to \Sigma_{\Delta}$ is naturally pulled back
from $\Sigma$. Indeed, we can describe the situation in slightly
different terms, without mentioning $\Delta$.  We have an $\SO(3)$
bundle $P \to X\sminus \Sigma$ and a reduction of $P$ to an $O(2)$
bundle $Q$ on $\nu\sminus \Sigma$. The bundle $G(P)\to
X\sminus\Sigma$ has a distinguished subbundle $H$ over
$\nu\sminus\Sigma$, namely the bundle whose fiber is the group
$S^{1}\subset \SU(2)$ which preserves $Q$ and its orientation. This
subbundle has structure group $\pm 1$ and is associated to the
orientation bundle of $Q$. This local system with structure group $\pm
1$ on $\nu\sminus\Sigma$ is pulled back from $\Sigma$ itself, so $H$
extends canonically over $\Sigma$. Thus, although the bundle $G(P)$ on
$X\sminus\Sigma$ does not extend, its subbundle $H$ does. There is a
topological space over $\mathbf{G}\to X$ obtained as an identification space of
$G(P)$ over $X\sminus\Sigma$ and $H$ over $\nu$:
\[
     \mathbf{G}=      (H \cup G(P))/\sim.
\]
The fibers of $\mathbf{G}$ over $X$ are copies of $S^{1}$ over
$\Sigma$ and copies of $\SU(2)$ over $X\sminus\Sigma$. The group $\Gtop$ is
the space of continuous sections of $\mathbf{G} \to X$.

We now wish to understand the component group, $\pi_{0}(\Gtop)$. To
begin, note that we have a restriction map
\[
         \Gtop \to \calH
\]
where $\calH$ is the space of sections of $H\to \Sigma$. 

\begin{lemma}\label{lem:pi0H}
    The group of components $\pi_{0}(\calH)$ is isomorphic to
    $H_{1}(\Sigma; \Z_{\Delta})$, where $\Z_{\Delta}$ is the local
    system with fiber $\Z$ associated to the double-cover
    $\Delta$. The map $\pi_{0}(\Gtop) \to \pi_{0}(\calH)$ is
    surjective. 
\end{lemma}

\begin{proof}
    Over $\Sigma$ there is a short exact sequence of sheaves
    (essentially the real exponential exact sequence, twisted by the
    orientation bundle of $Q$):
\[
              0 \to \Z_{Q} \to C^{0}(\R_{Q}) \to C^{0}(H) \to 0.
\]
From the resulting long exact sequence, one obtains an isomorphism
between $\pi_{0}(\calH)$ and $H^{1}(\Sigma; \Z_{Q})$, which is
isomorphic to $H_{1}(\Sigma;\Z_{\Delta})$ by Poincar\'e
duality. (Recall that the difference between $\Delta$ and the
orientation bundle of $Q$ is the orientation bundle of $\Sigma$.)
Geometrically, this isomorphism is realized by taking a section of $H$
in a given homotopy class, perturbing it to be transverse to the
constant section $-1$, and then taking the inverse image of $-1$. This
gives a smooth $1$-manifold in $\Sigma$ whose normal bundle is
identified with the orientation bundle of $H$, and whose tangent
bundle is therefore identified with the orientation bundle of
$\Delta$. This $1$-manifold, $C$, represents the element of $H_{1}(\Sigma;
\Z_{\Delta})$ corresponding to the given element of $\pi_{0}(\calH)$.

To prove surjectivity, we consider a class in $H_{1}(\Sigma;
\Z_{\Delta})$ represented by a $1$-manifold $C$ in $\Sigma$ whose
normal bundle is identified with $\R_{Q}$, and we seek to extend
the corresponding section $h$ of $H\to\Sigma$ to a section $g$ of
$\mathbf{G}\to X$. We can take $h$ to be supported in a $2$-dimensional
tubular neighborhood $V_{2}$ of $C$, and we seek a $g$ that is supported in a
$4$-dimensional tubular neighborhood $V_{4}$ of $C$. Let
$V'_{4}\subset V_{4}$ be smaller $4$-dimensional tubular
neighborhood. The section $h$
determines, by extension, a section $h'$ of $\mathbf{G}$ on
$V'_{4}$, and we need to show that $h'|_{\partial V'_{4}}$ is homotopic to the
section $1$. The fiber of $\partial V'_{4}$ over a point $x \in C$ is a
$2$-sphere, and on this $2$-sphere $h'$ is equal to $-1$ on an
equatorial circle $E$ (the circle fiber of the bundle $\partial\nu\to
\Sigma$ over $x$). To specify a standard homotopy from $h'$ to $1$ on
this $2$-sphere it
is sufficient to specify a non-vanishing section of $Q$ over
$E$. To specify a homotopy on the whole of $\partial V'_{4}$ we
therefore seek a non-vanishing section of the $2$-plane bundle $Q$ on
the circle bundle $T= \partial\nu\_{C} \to C$. This $T$ is a union of
tori or a Klein bottles, and on each component the orientation bundle of
$Q$ is identified with the orientation bundle of $T$. The obstruction
to there being a section is therefore a collection of integers $k$,
one for each component. Passing the
$\Delta$-double cover, we find our bundle $Q$ extending from the
circle bundle to the disk bundle, as the $2$-plane bundle
$Q_{\Delta}\to \nu_{\Delta}|_{C}$. The integer obstructions $k$
therefore satisfy $2k=0$. So $k=0$ and the homotopy exists.
\end{proof}

Next we look at the kernel of the restriction map $\Gtop \to\calH$,
which we denote temporarily by $\mathcal{K}$, in the exact sequence
\[
       1 \to \mathcal{K} \to \Gtop \to \calH.
\]

\begin{lemma}\label{lem:pi0-K}
    The group of components, $\pi_{0}(\mathcal{K})$, admits a
    surjective map
\[
          \pi_{0}(\mathcal{K}) \to H_{1}(X\sminus\Sigma;\Z)
\]
   whose kernel is either trivial or $\Z/2$. The latter occurs
   precisely when $w_{2}(P) = w_{2}(X\sminus\Sigma)$ in
   $H^{2}(X\sminus\Sigma;\Z/2)$.
\end{lemma}

\begin{proof}
    This is standard \cite{Akbulut-Mrowka-Ruan}. 
    A representative $g$ for an element
    of $\pi_{0}(\mathcal{K})$ is a section of $\mathbf{G}$ that is $1$ on
    $\Sigma$. The corresponding element of
    $H_{1}(X\setminus\Sigma;\Z)$ is represented by
    $\tilde{g}^{-1}(-1)$, where $\tilde{g} \simeq g$ is a section
    transverse to $-1$. The kernel is generated by a gauge
    transformation that is supported in a $4$-ball and represents the
    non-trivial element of $\pi_{4}(\SU(2))$. This element survives in
    $\pi_{0}(\mathcal{K})$ precisely when the condition on $w_{2}$ holds.
\end{proof}

\subsection{Analysis of singular connections}
\label{subsec:orbi-analysis}

Having discussed the topology of singular connections, we quickly
review some of the analytic constructions of \cite{KM-knot-singular}
which in turn use the work in \cite{KM-gtes-I}.  Fix a closed pair
$(X,\Sigma)$ and singular bundle
data $\bP$ on $(X,\Sigma)$, and construct a model singular connection
$A_{1}$ on $P\to X\sminus \Sigma$. (See section
\ref{subsec:sing-connections}.)  We wish to define a Banach space of
connections modeled on $A_{1}$. In \cite{KM-gtes-I} two approaches to
this problem were used, side by side. The first approach used spaces
of connections modeled on $L^{p}_{1}$, while the second approach used
stronger norms. The second approach required us to equip $X$ with a
metric with an orbifold singularity along $\Sigma$ rather than a
smooth metric.

It is the second of the two approaches that is most convenient in the
present context. Because we are only concerned with the case that the
holonomy parameter $\alpha$ is $1/4$, we can somewhat simplify the
treatment: in  \cite{KM-gtes-I}, the authors to used metrics $g^{\nu}$
on $X$ with cone angle $2\pi/\nu$ along $\Sigma$, with $\nu$ a
(possibly large) natural number. In the present context we can simply
take $\nu=2$, equipping $X$ with a metric with cone angle $\pi$ along
$\Sigma$.

Equipped with such a metric, $X$ can be regarded as an orbifold with
point-groups $\Z/2$ at all points of $\Sigma$. We will write
$\orbiX$ for $X$ when regarded as an orbifold in this way, and we
write $\orbig$ for an orbifold Riemannian metric. The holonomy of
$A_{1}$ on small loops linking $\Sigma$ in $X\setminus\Sigma$ is
asymptotically of order $2$; so in local branched double-covers of
neighborhoods of points of $\Sigma$, the holonomy is asymptotically
trivial. We can therefore take it that $P$ extends to an orbifold
bundle $\orbiP\to \orbiX$ and $A_{1}$ extends to a smooth
orbifold connection $\orbiA_{1}$ in this orbifold
$\SO(3)$-bundle. (Note that if we wished to locally extend an $\SU(2)$
bundle rather than an $\SO(3)$ bundle in this context, we should have
required a $4$-fold branched cover and we would have been led to use a
cone angle of $\pi/2$, which was the approach in \cite{KM-gtes-I}.)

Once we have the metric $\orbig$ and our model connection
$\orbiA_{1}$, we can define Sobolev spaces using the covariant
derivatives $\nabla^k_{\orbiA_{1}}$ on the bundles $\Lambda^p(T^*\orbiX)\otimes
\g_{\orbiP}$, where the Levi-Civita connection is used in
$T^*\orbiX$. The Sobolev space
$\orbiL^2_k(\orbiX ;\Lambda^p\otimes \g_{\orbiP})$ is the
completion of space of smooth orbifold sections with respect
to the norm
\begin{equation*}
\|a\|^2_{\orbiL^2_{k,\orbiA_{1}}}=\sum_{i=0}^k\int_{X\sminus \Sigma} |
\nabla_{\orbiA_{1}}a|^2 d \mathrm{vol}_{\orbig}.
\end{equation*}
We fix $k\ge 3$ and we consider a  space of
connections on $P\to X\sminus \Sigma$ defined as
\begin{equation}\label{eq:cA-def}
            \cA_k(X,\Sigma,\bP) = \{\,  A_{1} + a \mid
                                                a \in \orbiL^{2}_{k}(\orbiX) \,\}.
\end{equation}
As in \cite[Section 3]{KM-gtes-I}, the definition of this space of
connections can be reformulated to make clear that it depends only on
the singular bundle data $\bP$, and does not
otherwise depend on $A_{1}$. The reader can look there for a full
discussion.

This space of connections admits an action  by the
gauge group
\[
                \G_{k+1}(X,\Sigma,\bP)
\]
which is the completion in the $\orbiL^{2}_{k+1}$ topology of the
group $\G(\orbiP)$ of smooth, determinant-1 gauge transformations of
the orbifold bundle.  The fact that this is a Banach Lie group acting
smoothly on $\cA_k(X,\Sigma,\bP)$ is a consequence of multiplication
theorems just as in \cite{KM-gtes-I}.  Note that the center $\pm 1$ in
$SU(2)$ acts trivially on $\cA_k$ via constant gauge
transformations. Following the usual gauge theory nomenclature we call
a connection whose stabilizer is exactly $\pm 1$ \emph{irreducible}
and otherwise we call it \emph{reducible}. The homotopy-type of the
gauge group $\G_{k+1}(X,\Sigma,\bP)$ coincides with that of $\Gtop$,
the group of continuous, determinant-1 gauge transformations
considered earlier.

Here is the Fredholm package for the present situation.  Let $A \in
\cA_k(X,\Sigma,\bP)$ be a singular connection on $(X,\Sigma)$
equipped with the metric $\orbig$, and let $d_{A}^{+}$ be the
linearized anti-self-duality operator acting on $\g_{P}$-valued
$1$-forms, defined using the metric $\orbig$. Let $\mathcal{D}$ be the
operator
\begin{equation}\label{eq:D-4d}
                \mathcal{D}= 
                d^{+}_{A} \oplus -d^{*}_{A}
\end{equation}
acting on the spaces
\begin{equation*}
                \orbiL^{2}_{k}(\orbiX;\g_{\orbiP}\otimes
                \Lambda^1)
                \to
                \orbiL^{2}_{k-1}(\orbiX;\g_{\orbiP}\otimes
                (\Lambda^{+} \oplus \Lambda^{0}))   .            
\end{equation*}
Then in the orbifold setting $\mathcal{D}$ is a Fredholm
operator. (See for example \cite{KawaskiindVmfld} and compare with
\cite[Proposition 4.17]{KM-gtes-I}.)

We now wish to define a moduli space of anti-self-dual connections as
\[
            M(X,\Sigma,\bP) =
                    \{ \, A \in \cA_k \mid F^{+}_{A} = 0 \, \}
                    \bigm/ \G_{k+1}.
\]
Following \cite{KM-gtes-I}, there
is a Kuranishi model for the neighborhood of a  connection $[A]$
in $M(X,\Sigma,\bP)$ described by a Fredholm complex.  The Kuranishi theory then tells us,
in particular, that if $A$ is irreducible and the
operator $d^{+}_{A}$
is surjective, then a neighborhood of $[A]$ in $M(X,\Sigma,\bP)$ is a
smooth manifold,
and its dimension is equal to  the index of $\mathcal{D}$.

\subsection{Examples of moduli spaces} 
\label{subsec:RP2-examples}

The quotient of $\bar{\CP}^{2}$ by the action
of complex conjugation can be identified with $S^{4}$, containing a
copy of $\RP^{2}$ as branch locus. The self-intersection number of
this $\RP^{2}$ in $S^{4}$ is $+2$. The pair $(S^{4},\RP^{2})$ obtains
an orbifold metric from the standard Riemannian metric on
$\bar\CP^{2}$. We shall describe the corresponding 
moduli spaces $M(S^{4}, \RP^{2}, \bP)$ for various choices of singular
bundle data $\bP$.  

 On $\bar{\CP}^{2}$ with the
standard Riemannian metric, there is a unique anti-self-dual $\SO(3)$
connection with $\kappa=1/4$. This connection $A_{\bar\CP^{2}}$ is
reducible and has non-zero $w_{2}$: it splits as $\R\oplus L$, where
$L$ is an oriented $2$-plane bundle with $e(L)[\CP^{1}] = 1$. (See
\cite{Donaldson-Kronheimer} for example.) We can view $L$ as the
tautological line bundle on $\bar\CP^{2}$, and as such we see that the
action of complex conjugation on $\bar\CP^{2}$ lifts to an involution
on $L$ that is orientation-reversing on the fibers. This involution
preserves the connection. Extending the involution to act as $-1$ on
the $\R$ summand, we obtain an involution on the $\SO(3)$ bundle,
preserving the connection. The quotient by this involution is an
anti-self-dual connection $A$ on $S^{4}\setminus\RP^{2}$ for the
orbifold metric. It has $\kappa=1/8$. This orbifold connection
corresponds to singular bundle data $\bP$ on $(S^{4}, \RP^{2})$ with
$\Delta$ trivial. The connection is irreducible, and it is regular
(because $d^{+}$ is surjective when coupled to $A_{\bar\CP^{2}}$
upstairs).

This anti-self-connection is unique, in the following strong sense,
amongst solutions with $\Delta$ trivial and $\kappa=1/8$. To explain
this, suppose that we have $[A] \in M(S^{4}, \RP^{2}, \bP)$ and $[A']
\in M(S^{4}, \RP^{2}, \bP')$ are two solutions with $\kappa=1/8$, and
that trivializations of the corresponding local systems $\Delta$ and
$\Delta'$ are given. When lifted to
$\bar\CP^{2}$, both solutions must give the same $\SO(3)$ connection
$A_{\bar\CP^{2}}$ up to gauge transformation, in the bundle $\R\oplus
L$. There are two different ways to lift the involution to $L$ on
$\bar\CP^{2}$, differing in overall sign, but these two involutions on
$L$ are intertwined by multiplication by $i$ on $L$. It follows that,
as $\SO(3)$-bundles with connection on $S^{4}\setminus \RP^{2}$, the
pairs $(P,A)$ and $(P',A')$ are isomorphic. Such an isomorphism of
bundles with connection extends canonically to an isomorphism $\psi$ from
$\bP$ to $\bP'$. A priori, $\psi$ may not preserve the given
trivializations of $\Delta$ and $\Delta'$. However, the connection $A$
on $P$ has structure group which reduces to $O(2)$, so it has a $\Z/2$
stabilizer in the $\SO(3)$ gauge group on $S^{4}\setminus\RP^{2}$. The
non-trivial element of this stabilizer is an automorphism of $P$ that
extends to an automorphism of $P_{\Delta}$ covering the non-trivial
involution $t$ on $X_{\Delta}$. So $\psi$ can always be chosen to
preserve the chosen trivializations.

When $\Delta$ is trivialized, singular bundle data $\bP$ is classified
by the evaluation of $p_{1}(P_{\Delta})$ on $[D]$ and on
$[X_{+}]$ in the notation of subsection~\ref{subsec:inst-monop-numb},
or equivalently by its instanton and monopole numbers $k$ and $l$. The
uniqueness of $[A]$ means that the corresponding singular bundle data
$\bP$ must be invariant under the involution $t$, which in turn means
that $p_{1}(P_{\Delta})[D]$ must be zero. Using the formulae for $k$,
$l$ and $\kappa$, we see that $\bP$ has $k=0$ and $l=1/2$; or
equivalently, $p_{1}(P_{\Delta})=0$. We summarize this discussion
with a proposition.

\begin{proposition}
\label{prop:RP2-example-prop}
    Let $(S^{4},\RP^{2})$ be as above, so that the branched
    double-cover is $\bar\CP^{2}$. Fix a trivial and trivialized
    double-cover $\Delta$ of $\RP^{2}$, and let $S^{4}_{\Delta}$ be
    the corresponding space. Then, amongst singular bundle data with
     $\kappa=1/8$, there is exactly one $P_{\Delta} \to
     S^{4}_{\Delta}$ with a non-empty moduli space, namely the one
     with $p_{1}(P_{\Delta})=0$ in
     $H^{4}(S^{4}_{\Delta})=\Z\oplus\Z$. The corresponding moduli space
     is a single point, corresponding to an irreducible, regular
     solution. \qed
\end{proposition}

On the same pair $(S^{4},\RP^{2})$, there is also a solution in a
moduli space corresponding to singular bundle data $\bP$ with $\Delta$
non-trivial. This solution can be described in a similar manner to the
previous one, but starting with \emph{trivial} $\SO(3)$ bundle on
$\bar{\CP}^{2}$, acted on by complex conjugation, lifted as an
involution on the bundle as an element of order $2$ in $\SO(3)$. The
resulting solution on $(S^{4},\RP^{2})$ has $\kappa=0$ and $\Delta$
non-trivial. Knowing that $\kappa=0$ is enough to pin down
$P_{\Delta}\to S^{4}_{\Delta}$ up to isomorphism in this case, because
$H^{4}(S^{4}_{\Delta})$ is now $\Z$. This solution $[B]$ is
reducible. It is regular, for similar reasons as arise in the previous
case. The index of $\mathcal{D}$ in this case is therefore $-1$.

\medskip One can also consider the quotient of $\CP^{2}$, rather then
$\bar\CP^{2}$, with respect to the same involution, which leads to a
pair $(S^{4},\RP^{2}_{-})$ with $\RP^{2}_{-}\cdot
\RP^{2}_{0}=-2$. There is an isolated anti-self-dual $\PSU(2)$
connection on $\CP^{2}$, arising from the $U(2)$ connection given by
the Levi-Civita derivative in $T\CP^{2}$. This gives rise to a
solution on $(S^{4},\RP^{2}_{-})$ with $\Delta$ trivial and
$\kappa=3/8$. As solutions with $\Delta$ trivialized (rather than just
trivial), this solution gives rise to two different solutions, with
$(k,l)=(0,1/2)$ and $(k,l)=(1,-3/2)$, as the reader can verify. The
operator $\mathcal{D}$ is again invertible for these two
solutions. The flat connection $[B]$ on $\CP^{2}\setminus\RP^{2}$
provides another solution, with $\kappa=0$ and $\Delta$
non-trivial. This solution is reducible and the operator $d^{+}_{B}$
now has $2$-dimensional cokernel, so that $\mathcal{D}$ has index $-3$.

\subsection{The dimension formula} 

We next compute the index of the operator $\mathcal{D}$ for a
connection $A$ in $\cA_{k}(X,\Sigma,\bP)$. The index of $\mathcal{D}$
will coincide with the dimension of the moduli space $M(X,\Sigma,\bP)$
in the neighborhood of any irreducible, regular solution.

\begin{lemma}\label{lem:index}
The index of $\mathcal{D}$ is given by
    \[
    8\kappa(A)-\frac{3}{2}\bigl(\chi(X)+\sigma(X)\bigr)+\chi(\Sigma) +    
 \frac{1}{2} (\Sigma \cdot \Sigma)
    \]
where $\kappa$ is again the action \eqref{eq:action}.
\end{lemma}

\begin{proof}
   For the case of orientable surfaces with $\Delta$ trivial, this
   formula reduces to the formula proved in 
  \cite{KM-gtes-I}, where it appears as
   \begin{equation}\label{eq:old-k-l-dim}
    8k  + 4l -\frac{3}{2}\bigl(\chi(X)+\sigma(X)\bigr)+\chi(\Sigma).
   \end{equation}
   Just as in \cite{KM-gtes-I}, the general case can be proved by
   repeatedly applying excision and the homotopy invariance of the index,
   to reduce the problem to a few model cases. In addition to the
   model cases from the proof in \cite{KM-gtes-I}, it is now necessary
   to treat one model case in which $w_{1}(\Delta)^{2}$ is non-zero on
   $\Sigma$. Such an example is
   provided by the flat connection $[B]$ on $(S^{4},\RP^{2})$ from the
   previous subsection. In this case, $\kappa$ is $0$ and the formula
   in the lemma above predicts that the index of $\mathcal{D}$ should
   be $-3$. This is indeed the index of $\mathcal{D}$ in this case, as
   we have already seen.
\end{proof}

\subsection{Orientability of moduli spaces}

We continue to consider the moduli space $M(X,\Sigma,\bP)$ associated
to a closed pair $(X,\Sigma)$ equipped with an
orbifold metric and singular bundle data $\bP$. The irreducible,
regular solutions form a subset of $M(X,\Sigma,\bP)$ that is a smooth
manifold of the dimension given by Lemma~\ref{lem:index}, and our next
objective is to show that this manifold is orientable. As usual, the
orientability of the moduli space is better expressed as the
triviality of the real determinant line of the family of operators
$\mathcal{D}$ over the space $\bonf^{*}_{k}(X,\Sigma,\bP)$ of all
irreducible connections modulo the determinant-1 gauge group.

\begin{proposition}
\label{prop:orientable}
    The real line bundle $\det \mathcal{D}$ of the family of
    operators $\mathcal{D}$ over $\bonf^{*}_{k}(X,\Sigma,\bP)$ is trivial.
\end{proposition}

\begin{proof}
    The proof follows that of the corresponding result in
    \cite{KM-gtes-II}, which in turn is based on
    \cite{Donaldson-orientations}.  We must show that the determinant
    line is orientable along all closed loops in
    $\bonf^{*}_{k}(X,\Sigma,\bP)$.
     The fundamental group of
    $\bonf^{*}_{k}(X,\Sigma,\bP)$ is isomorphic to the group of
    components of $\G_{k+1}(X,\Sigma,\bP)/\{\pm 1\}$. The group
    $\pi_{0}(\G_{k+1}(X,\Sigma,\bP)$ is the same as $\pi_{0}(\Gtop)$,
    for which explicit generators can be extracted from
    Lemmas~\ref{lem:pi0H} and \ref{lem:pi0-K}. These generators
    correspond to loops in $X\sminus\Sigma$, loops in $\Sigma$ along
    which $\Delta$ is trivial, and a possible additional $\Z/2$. Of
    these, the only type of generator that is new in the present paper
    is a loop  $\Sigma$ along which $\Delta$ is trivial but $\Sigma$
    is non-orientable. (The orientable case is essentially dealt with
    in \cite{KM-gtes-II}.)  
 
    So let $\gamma$ be such a loop in $\Sigma$ along which $\Delta$ is
    trivialized. There is an element
    $g$ in $\pi_{0}(\G_{k+1})$ which maps to the class $[\gamma]$ in
    $H_{1}(\Sigma;\Z_{\Delta})$ under the map in
    Lemma~\ref{lem:pi0H}. We wish to describe loop $\gamma^{*}$ in
    $\bonf^{*}_{k}$ that represents a corresponding element in
    $\pi_{1}(\bonf^{*}_{k})$. As in \cite[Appendix 1(i)]{KM-gtes-II}, we construct
    $\gamma^{*}$ by gluing in a monopole at and ``dragging it around
    $\gamma$''. To do this, we first let $\bP'$ be singular bundle
    data such that $\bP$ is obtained from $\bP'$ by adding a
    monopole. We fix a connection $A'$ in $\bP'$. We let $J$ be a
    standard solution on $(S^{4}, S^{2})$ with monopole number $1$ and
    instanton number $0$, carried by singular bundle data  with
    $\Delta$ trivialized. For each $x$ in $\gamma$, we form a
    connected sum of pairs,
    \[
                  (X,\Sigma) \#_{x} (S^{4}, S^{2})
     \]
    carrying a connection $A(x) = A' \# J$. Even though $\Sigma$ is
    non-orientable along $\Delta$, a closed loop in
    $\bonf_{k}(X,\Sigma,\bP)$ can be constructed this way, because the
    solution on $(S^{4}, S^{2})$ admits a symmetry which reverses the
    orientation of $S^{2}$ while preserving the trivialization of
    $\Delta$. This is the required loop $\gamma^{*}$. If we write
    $\delta_{A(x)}$ and $\delta_{A'}$ for the determinant lines of
    $\mathcal{D}$ for the two connections, then (much as in
    \cite{KM-gtes-II}) we have the relation
\[
          \delta_{A(x)} = \delta_{A'} \otimes \det(T_{x}\Sigma \oplus
          \R \oplus \R_{Q})        
\]
   where the first $\R$ can be interpreted as the scale parameter in the
   gluing and the factor $\R_{Q}$ is the real orientation bundle of
   $Q$, which arises as the tangent space to the $S^{1}$ gluing
   parameter. Since the product of the orientation bundles of $\Sigma$
   and $Q$ is trivial along $\gamma$ (being the orientation bundle of
   $\Delta$), we see that $\delta$ is trivial along the loop $\gamma^{*}$.    
\end{proof}

The factor $\det(T_{x}\Sigma \oplus \R \oplus \R_{Q})$ above can be
interpreted as the orientation line on the moduli space of framed
singular instantons on the orbifold $(S^{4}, S^{2})$ (or in other
words, the moduli space of finite-action solutions on $(\R^{4},
\R^{2})$ modulo gauge transformations that are asymptotic to $1$ at
infinity. This is a complex space, and therefore has a preferred
orientation, for any choice of instanton and monopole charges. In the
situation that arises in the proof of the previous proposition, there
is therefore a preferred way to orient $\delta_{A(x)}$, given an
orientation of $\delta_{A'}$. To make use of this, we consider the
following setup. Let $\bP$ be singular bundle data, given on
$(X,\Sigma)$, and let us consider the set of
pairs $(\bP' , g')$, where $\bP'$ is another choice of singular bundle
data and 
\[
    g' : \bP'|_{(X_{\Delta}\setminus \bx')} 
    \to \bP|_{(X_{\Delta}\setminus \bx')}
\]
is an isomorphism defined on the complement of a finite set $\bx'$. We
say that $(\bP' , g')$ is isomorphic to 
$(\bP'', g'')$ if there is an isomorphism of singular bundle data, 
$h : \bP' \to \bP''$, such that
the composite $g'' \circ (g')^{-1}$ can be lifted to a
determinant-1 gauge transformation on its domain of definition,
$X_{\Delta}\setminus (\bx' \cup \bx'')$.

\begin{definition}\label{def:P-marked}
    We refer to such a pair $(\bP', g')$ as a $\bP$-marked bundle.
   \CloseDef
\end{definition}

The classification of the isomorphism classes of $\bP$-marked bundles
on $(X,\Sigma)$ can be deduced from the material of
section~\ref{subsec:topol-class-sing}. Every $\bP$-marked bundle can
be obtained from $\bP$ by ``adding instantons and
monopoles''. Furthermore, as in \ref{item:mono-plus-anti} and
\ref{item:mono-plus-mono} on page \pageref{item:mono-plus-anti},
adding a monopole and an anti-monopole to the same orientable
component of $\Sigma$ is equivalent to adding an instanton, while
adding two monopoles to the same non-orientable component is also the
same as adding an instanton. There are no other relations: in
particular, because of the determinant-1 condition in our definition
of the equivalence relation, there is no counterpart here of the
relation \ref{item:Dold-Whitney-extra} from page
\pageref{item:Dold-Whitney-extra}. We now have, as in \cite{KM-gtes-I}:

\begin{proposition}\label{prop:orientations-propogate}
    Using the complex orientations of the framed moduli spaces on
    $(S^{4}, S^{2})$, an orientation for the determinant line of $\calD$ over
    $\bonf(X,\Sigma,\bP)$ determines an orientation of the determinant
    line also over $\bonf(X,\Sigma,\bP')$ for all $\bP$-marked bundles
    $(\bP', g')$. These orientations are compatible with equivalence
    of $\bP$-marked bundles, in that if $h:\bP'\to \bP''$ is an
    equivalence (as above), then the induced map on the corresponding
    determinant line is orientation-preserving. \qed
\end{proposition}

\begin{remark}
    The reason for the slightly complex setup in the above proposition
    is that we do not know in complete generality whether gauge
    transformations that do not belong to the determinant-1 gauge
    group give rise to orientation-preserving maps on the moduli
    spaces $M(X,\Sigma,\bP)$. We will later obtain some partial
    results in this direction, in section~\ref{subsec:almost-complex};
    but those results apply, as they stand, only to the case that
    $\Sigma$ is orientable.
\end{remark}

\section{Singular instantons and Floer homology for knots}
\label{sec:sing-inst-floer}

In this section we review how to adapt the gauge theory for singular
connections on pairs $(X,\Sigma)$ to the three-dimensional case of a
link $K$ in a $3$-manifold $Y$, as well as the case of a $4$-manifold
with cylindrical ends. This is all adapted from
\cite{KM-knot-singular}: as in section~\ref{sec:sing-inst-non}, the
new ingredient is that we are no longer assuming that our $\SO(3)$
bundle $P$ on $Y\sminus K$ extends over $K$.

\subsection{Singular connections in the 3-dimensional case}

We fix a closed, oriented, connected
three manifold $Y$ containing a knot or link $K$. The definition of
\emph{singular bundle data} from
Definition~\ref{def:singular-bundle-data} adapts in a straightforward
way to this $3$-dimensional case: such data consists of a double-cover
$K_{\Delta} \to K$, an $\SO(3)$ bundle $P_{\Delta}$ on the
corresponding non-Hausdorff space $Y_{\Delta}$, and a reduction of
structure group to $O(2)$ in a neighborhood of $K_{\Delta}\subset
Y_{\Delta}$, taking a standard form locally along $K_{\Delta}$. Note
that $K$ is always orientable, but has not been oriented. A choice of
orientation for $K$ will fix an isomorphism between the local system
$\Delta$ and the orientation bundle of the $O(2)$ reduction $Q$ in the
neighborhood of $K$. For a given $\Delta$, we can also form the
Hausdorff space $Y^{h}_{\Delta}$, which contains a copy of $2$-sphere
bundle over $K$ (non-orientable if $\Delta$ is non-trivial). The $\SO(3)$-bundle
$P_{\Delta}$ has $w_{2}$ non-zero on the $S^{2}$ fibers. Such
$w_{2}'s$ form an affine copy of $H^{2}(Y;\Z/2)$ in
$H^{2}(Y_{\Delta};\Z/2)$, 
and they classify singular bundle data for the given $\Delta$.

Equipping $Y$ with
an orbifold structure along $K$ with cone angle $\pi/2$ and a compatible
orbifold  Riemannian metric $\orbig$, we can regard singular bundle
data $\bP$ as determining an orbifold bundle $\orbiP \to \orbiY$, as
in the $4$-dimensional case. Thus we construct Sobolev spaces
\[
 \orbiL^2_{k}(\orbiY;\g_{\orbiP}\otimes \Lambda^q)
\]
as before, leading to spaces of $\SO(3)$ connections $\cA_k(Y,K,\bP)$
and determinant-1 gauge transformations
$\cG_{k+1}(Y,K,\bP)$.  The space of connections  $\cA_k(Y,K,\bP)$ is an affine space, and on the 
tangent space 
\[
T_{B}\cA_k(Y,K,\bP)=\orbiL^2_{k}(\orbiY; \g_{\orbiP}\otimes T^{*}\orbiY )
\] 
we define an $L^{2}$ inner product (independent of $B$) by
\begin{equation}\label{eq:inner-prod}
            \langle b, b' \rangle_{L^{2}} = \int_{Y}
             -\tr  (*b \wedge  b'),
\end{equation}
Here $\tr$ denotes the Killing form on $\su(2)$ and 
the Hodge star is the one defined by the
singular metric $\orbig$. We have the \emph{Chern-Simons functional} on
$\cA_k(Y,K,\bP)$
characterized by
\[
           ( \grad \CS)_{B} = *F_{B}.
\]
Critical points of $\CS$ are the flat connections in $\cA_{k}(Y,K,\bP)$.
We denote the set of gauge equivalence classes
of critical points by $\crit=\crit(Y,K,\bP)$.

\subsection{The components of the gauge group on a three manifold.}

We can analyze the component group $\pi_{0}(\G_{k+1}(Y,K,\bP))$ much as
we did in the $4$-dimensional case. The group $\G_{k+1}(Y,K,\bP)$ is
homotopy equivalent to the $\Gtop$, the continuous gauge
transformations of $P_{\Delta}$ respecting the reduction, and 
we have a short exact sequence,
\[
        0 \to \Z \to \pi_{0}(\G_{k+1}(Y,K,\bP)) \to H_{0}(K;
        \Z_{\Delta}) \to 0.
\]
(See Lemma~\ref{lem:pi0H}.) 
The $\Z$ in the kernel has  a generator $g_{1}$ represented by a
gauge transformation supported in a ball disjoint from $K$. The group
$H_{0}(K;\Z_{\Delta})$ arises as $\pi_{0}(\calH)$, where $\calH$ is
the group of sections of the bundle $H_{\Delta}\to K$ with fiber
$S^{1}$. The group $H_{0}(K;\Z_{\Delta})$ is a direct sum of one copy
of $\Z$ for each component of $K$ on which $\Delta$ is trivial and one
copy of $\Z/2$ for each component on which $\Delta$ is non-trivial.

The above sequence is not split in general. We can express the
component group
$\pi_{0}(\G_{k+1}(Y,K,\bP))$ as having generators $g_{1}$ (in the
kernel) and one generator $h_{i}$ for each component $K_{i}$ of $K$,
subject to the relations
\[
           2 h _{i} = g_{1}
\]
whenever $\Delta|_{K_{i}}$ is non-trivial.

The Chern-Simons functional is invariant under the identity component
of the gauge group. Under the generators $g_{1}$ and $h_{i}$ it
behaves as follows:
\[
\begin{aligned}
    \CS(g_{1}(A)) &= \CS(A) - 4\pi^{2} \\
    \CS(h_{i}(A)) &= \CS(A) - 2\pi^{2}.
\end{aligned}
\]

\subsection{Reducible connections and the non-integral condition}

For the construction of Floer homology it is important to understand
whether there are reducible connections in $\cA(Y,K,\bP)$, or at least
whether the critical points of $\CS$ are irreducible. 

This is
discussed in \cite{KM-knot-singular} for the case that $P$ extends to
$Y$ (or equivalently, the case that $\Delta$ is trivial). See in
particular the ``non-integral condition'' of 
\cite[Definition 3.28]{KM-knot-singular}. For this paper, 
we make the following adaptation of the definition, whose
consequences are summarized in the following proposition.

\begin{definition}\label{def:non-int}
Let singular bundle data $\bP$ on $(Y,K)$ be given. We say that an embedded closed
oriented surface $\Sigma$ is a \emph{non-integral surface} if either
\begin{itemize}
\item $\Sigma$ is disjoint from $K$ and $w_{2}(P)$ is non-zero on $\Sigma$; or
\item $\Sigma$ is transverse to $K$ and $K\cdot\Sigma$ is odd.
\end{itemize}
We say that $\bP$ satisfies the \emph{non-integral condition} if there
is a non-integral surface $\Sigma$ in $Y$.
\CloseDef
\end{definition}

\begin{proposition}\label{prop:non-integral}
    If singular bundle data $\bP$ on $(Y,K)$ satisfies the
    non-integral condition, then the Chern-Simons functional on
    $\conf_{k}(Y,K,\bP)$ has no reducible critical
    points. Furthermore, if $\Delta$
    is non-trivial on any component of $K$, 
    then a stronger conclusion holds: the configuration space 
    $\conf_{k}(Y,K,\bP)$ contains no reducible
    connections at all.
\end{proposition}

\begin{proof}
    For the first part, the point is that there are already no
    reducible flat connections in the restriction of $\bP$ to
    $\Sigma$: since $\Delta$ becomes trivial when restricting to the
    surface, there is nothing new in this statement beyond the
    familiar case where $\bP$ extends across $K$.

    If  $\Delta$ is non-trivial on a component $K'$ of
$K$,  then  if $T'$ is a torus which is the boundary of a small
tubular neighborhood of $K'$ we have
\[
\langle w_{2}(P),[T']\rangle \ne 0.
\]
Since $T'$ is disjoint from $K$, we see that $T'$ is a non-integral
surface, and so the non-integral condition is automatically satisfied.
The asymptotic
holonomy group of a connection in $\conf(Y,K,\bP)$ 
in the neighborhood of $K'$ lies in
$O(2)$ and contains both (a) an element of $\SO(2)\subset O(2)$ of
order $2$ (the meridional holonomy), and (b) an element of $O(2)
\setminus \SO(2)$, namely the holonomy along a longitude. Such a
connection therefore cannot be reducible, irrespective of whether it
is flat or not.
\end{proof}

\subsection{Perturbations}
\label{subsec:perturbations}

We now introduce standard perturbations of the Chern-Simons functional
to achieve suitable transversality properties for the both the set of
critical points and the moduli spaces of trajectories for the formal
gradient flow. Section 3.2 of \cite{KM-knot-singular} explains how to
do this, following work of Taubes~\cite{Taubes-Casson} and Donaldson;
and the approach described there needs almost no modification in the
present context. 

The basic function used in constructing perturbations 
is obtained as follows. Choose a lift of the bundle $P \to Y\setminus
K$ to a $U(2)$ bundle $\tilde{P}\to Y\setminus K$, and fix a
connection $\theta$ on $\det\tilde{P}$. Each $B$ in
$\conf_{k}(Y,K,\bP)$ then gives rise to a connection $\tilde{B}$ in
$P$ inducing the connection $\theta$ in $\det\tilde{P}$. Take an immersion
$q:S^1\times D^2$ to $Y\sminus K$, and choose a base point 
$p\in S^1$. For each $x \in D^2$
the holonomy of a connection $\tilde{B}$ about $S^1\times x$ starting at $(p,x)$
gives an element $\mathrm{Hol}_x(\tilde{B})\in U(2)$. Taking a class function 
$h:U(2) \to \R$, we obtain  a gauge-invariant function 
\[
              H_{x} : \cA(Y,K,\bP)\to \R
\]
as the composite
$H_{x}(B) = h \circ\mathrm{Hol}_x(\tilde B)$.
For analytic purposes it is useful to mollify this function by introducing
\[
            f_q(B) =\int_{D^2} H_x(B)\mu
\]
where $\mu$ is choice of volume form on $D^2$ which we take
to be supported in the interior of $D^2$ and have integral $1$.

More generally taking a collection of such immersions ${\mathbf{q}}=(q_1,\ldots,q_l)$
so that they all agree on $p\times D^2$, the holonomy about the
$l$ loops determined by $x\in D^2$ gives a map $\mathrm{Hol}_x:\cA_{k} \to U(2)^l$.
Now taking an conjugation invariant function $h:U(2)^l \to \R$
we obtain again a gauge invariant function 
$H_{x}= h\circ \mathrm{Hol}_x\colon\cA_{k} \to \R$. Mollifying this we obtain
\[
\begin{gathered}
    f_{\mathbf{q}} : \cA_{k}(Y,K,\bP) \to \R \\
    f_{\mathbf{q}}(B) =\int_{D^2} H_x(B)\mu.
\end{gathered}
\]
These are smooth gauge invariant functions on $\cA$
and are called \emph{cylinder functions}. 

Our typical perturbation $f$ will be an infinite linear combination of
such cylinder functions. In \cite{KM-knot-singular} it is explained how to
construct an infinite collection $\mathbf{q}^{i}$ of immersions and a separable
Banach space $\Pert$ of sequences $\pert=\{\pert_{i}\}$, with norm
\[
         \| \pert \|_{\Pert} = \sum_{i}C_{i} |\pert_{i}|
\]       
such that for each $\pert\in \Pert$, the sum
\[
        f_{\pi} = \sum_{i} \pert_{i} f_{\mathbf{q}^{i}}
\]
is convergent and defines a smooth, bounded function $f_{\pi}$ on
$\cA_{k}$. Furthermore, the formal $L^{2}$ gradient of $f_{\pi}$
defines a smooth vector field on the Banach space $\cA_{k}$, which we
denote by $V_{\pert}$. The analytic properties that we require for
$V_{\pert}$ (all of which can be achieved by a suitable choice of
$\Pert$) are summarized in \cite[Proposition 3.7]{KM-knot-singular}.

We refer to a function $f_{\pert}$ of this sort as a \emph{holonomy
  perturbation}. Given such a function, we then consider the perturbed
Chern-Simons functional $\CS+f_{\pert}$, The set of gauge
equivalence classes of critical points for the perturbed functional 
is denoted 
\[
\crit_{\pert}(Y,K,P,\rho) \subset \bonf_{k}(Y,K,\bP),
\]
or simply as $\crit_{\pert}$.
 Regarding
the utility of these holonomy perturbations, $\Pert$
we have \cite[Proposition 3.12]{KM-knot-singular}:

\begin{proposition}\label{prop:residual-crit}
    There is a residual subset of the Banach space $\Pert$ such that for
    all $\pert$ in this subset, all the irreducible critical points of the perturbed
    functional $\CS + f_{\pert}$ in $\cA^{*}(Y,K,P,\rho)$ are
    non-degenerate in the directions transverse to the gauge orbits. \qed
\end{proposition}

This says nothing yet about the reducible critical points. We will be
working eventually with configurations $(Y,K,\bP)$ satisfying the
non-integral condition of Definition~\ref{def:non-int}. In case the
bundle does not extend there can be no reducible connections at all
(by the second part of Proposition~\ref{prop:non-integral}), but in
the case where $P$ does extend, there may be reducible critical points
if the perturbation is large. However, for small perturbations, there
are none:

\begin{lemma}[Lemma 3.11 of \cite{KM-knot-singular}]
\label{lem:pert}
     Suppose $(Y,K,\bP)$ satisfies the non-integral condition.
    Then there exists $\epsilon>0$ such that for all $\pert$ with $\| \pert \|_{\Pert} \le
    \epsilon$, the critical points of $\CS + f_{\pert}$ in
    $\conf_{k}(Y,K,\bP)$ are
    all irreducible. \qed
\end{lemma}

This lemma is corollary of the compactness properties of the perturbed
critical set: in particular, the fact that the projection
\[ \crit_{\bullet} \to \Pert \] from the parametrized critical-point set
$\crit_{\pert}\subset \Pert\times \bonf_{k}(Y,K,\bP)$ is a
proper map. This properness also ensures that the $\crit_{\pert}$ is
finite whenever all the critical points of the perturbed functional
are non-degenerate.

\subsection{Trajectories for the perturbed gradient flow}

Given a holonomy perturbation $f_{\pert}$ for the Chern-Simons
functional on the space $\cA_{k}(Y,K,\bP)$, we get a perturbation
of the anti-self-duality equations on the cylinder. The ``cylinder''
here is the product pair,
\[
    (Z,S) = \R \times (Y,K)
\]
viewed as a $4$-manifold with an embedded surface. Since $Y$ is
equipped with a singular Riemannian metric to become an orbifold
$\orbiY$, so also $Z$ obtains a product orbifold structure: it becomes
an orbifold $\orbiZ$ with singularity along $S$, just as in our discussion
from section~\ref{subsec:orbi-analysis} for a closed pair $(X,\Sigma)$.

To write down the perturbed equations, let the $4$-dimensional
connection be expressed as
\[
A=B+c dt,
\]
with $B$ a $t$-dependent (orbifold) connection on $(Y,K)$ and $c$ a
$t$-dependent section of $\g_{\orbiP}$. We write
\[
\hat V_{\pert}(A) = P_{+}(dt\wedge V_{\pert}(B))
\]
where $P_{+}$ the projection onto the self-dual 2-forms, and
$V_{\pert}(B)$ is viewed as a $\g_{\orbiP}$-valued 1-form on $\R\times
\orbiY$ which evaluates to zero on multiples of $d/dt$. The
$4$-dimensional self-duality equations on $\R\times Y$, perturbed by
the holonomy perturbation $V_{\pert}$, are the equations
\begin{equation}\label{eq:ASDpert}
   F^{+}_{A}+\hat V_{\pert}(A)=0.
\end{equation}
These equations are invariant under the $4$-dimensional gauge group.
Solutions to the downward gradient flow equations for the perturbed
Chern-Simons functional correspond to solutions of these equations
which are in temporal gauge (i.e. have $c=0$ in the above
decomposition of $A$).  The detailed mapping properties of $\hat
V_{\pert}$ and its differential are given in Proposition 3.15 of
\cite{KM-knot-singular}.

Let $\pert$ be chosen so that all critical points in $\crit_{\pert}$
are irreducible and non-degenerate. 
Let $B_{1}$ and $B_{0}$ be critical points in $\cA_{k}(Y,K,\bP)$, and
let $\beta_{1}, \beta_{0}\in \crit_{\pert}$ be their gauge-equivalence
classes. Let $A_{o}$ be a connection on $\R\times Y$ which agrees with
the pull-back of $B_{1}$ and $B_{0}$ for large negative and large
positive $t$ respectively. The connection $A_{o}$ determines a path
$\gamma : \R\to \bonf_{k}(Y,K,\bP)$, from $\beta_{1}$ to $\beta_{0}$,
which is constant outside a compact set. The relative homotopy class
\[
      z \in \pi_{1}(\bonf_{k} , \beta_{1},\beta_{0})
\]
of the path $\gamma$ depends on the choice of $B_{1}$ and $B_{0}$
within their gauge orbit. Given $A_{o}$, we can construct a space of
connections
\[
       \cA_{k,\gamma}(Z,S,\bP; B_{1}, B_{0})
\]
as the affine space
\[
\{\, A \mid A - A_{o}\in \orbiL^{2}_{k,A_{o}}(\orbiZ ;
T^{*}\orbiZ\otimes \g_{\orbiP})\,\}.
\]
There is a corresponding gauge group, the space of sections of the bundle
$G(P)\to Z\sminus S$ defined by
\[
                \G_{k+1}(Z,S,\bP) = \{\,
                g  \mid
                \text{ $\nabla_{A_{o}}g,\dots,
                \nabla^{k}_{A_{o}}g \in
               \check L^{2}(Z\sminus S)$}\,\}.
\]
We have the quotient space
\[
\cB_{k,z}(Z,S,\bP;\beta_{1},\beta_{0})=
  \cA_{k,\gamma}(Z,S,\bP;B_{1},B_{0})\big/\G_{k+1}(Z,S,\bP)
\]
Here $z$ again denotes the homotopy class of $\gamma$ in
$\pi_{1}(\cB_{k}(Y,K,\bP);\beta_{1},\beta_{0})$.

We can now construct the moduli space of solutions to the perturbed
anti-self-duality equations (the $\pert$-ASD connections) as a
subspace of the above space of connections modulo gauge:
\[
    M_{z}(\beta_{1},\beta_{0})=\bigl\{ \,[A] \in
    \bonf_{k,z}(Z,S,\bP;\beta_{1},\beta_{0})\bigm|
         F^{+}_{A} + \hat V_{\pert}(A)=0\,\bigr\}.
\]
This space is homeomorphic to the space of trajectories of the formal
downward gradient-flow equations for $\CS+f_{\pert}$ running from
$\beta_{1}$ to $\beta_{0}$ in the relative homotopy class $z$.
Taking the union over all $z$, we write
\[
M(\beta_{1},\beta_{0})= \bigcup_{z} M_{z}(\beta_{1},\beta_{0}).
\]
The action of $\R$ by translations on $\R\times Y$ induces
an action on $M(\beta_{1},\beta_{0})$.  The action is free
except in the case of $M(\beta_{1},\beta_{1})$ and constant
trajectory.
The quotient of the space of non-constant solutions by this action of $\R$ is
denoted $\Mu(\beta_{1},\beta_{0})$ and typical elements are denoted
$[\breve A]$.

\medskip

The linearization of the $\pert$-ASD condition at a connection $A$ in
$\conf_{k,\gamma}(B_{1},B_{0})$ 
is the map
\[
   d^{+}_{A}+D\hat V_{\pert} : \orbiL^{2}_{k,A}(\orbiZ;\Lambda^{1}\otimes\g_{\orbiP})
            \to
            \orbiL^{2}_{k-1}(\orbiZ;\Lambda^{+}\otimes\g_{\orbiP}).
\]
When $B_{1}$ and $B_{1}$ are irreducible and non-degenerate, we have a
good Fredholm theory for this linearization together with gauge
fixing. For $A$ as above, we write (as in \eqref{eq:D-4d}, but now
with the perturbation)
\[
        \calD_{A} = (d^{+}_{A} + D\hat{V}_{\pert}) \oplus -d^{*}_{A}
\]
which we view as an operator
\[
\orbiL^{2}_{k,A}(\orbiZ;\Lambda^{1}\otimes\g_{\orbiP}) \to
\orbiL^{2}_{k-1}(\orbiZ;(\Lambda^{+}
\oplus\Lambda^{0})\otimes\g_{\orbiP})
\]
Viewed this was, $\calD_{A}$ is a Fredholm operator.  When $\calD_{A}$ is 
surjective we say that $A$ is a \emph{regular} solution, and in this case
 $M_{z}(\beta_{1},\beta_{0})$ is a smooth 
manifold near the gauge equivalence class of $[A]$, of dimension equal
to the index of $\calD_{A}$.
The index of the operator, which can be interpreted as a spectral flow,
will be denoted by
\[
  \gr_{z}(\beta_{1},\beta_{0}).
\]
We will need the moduli spaces of fixed relative grading
\[
M(\beta_{1},\beta_{0})_{d}= \bigcup_{z\,\mid\,\gr_{z}=d}
M_{z}(\beta_{1},\beta_{0})
\]

For the four-dimensional equations we have the following
transversality result.

\begin{proposition}[\protect{\cite[Proposition 3.18]{KM-knot-singular}}]
\label{prop:4d-transversality}
Suppose that $\pert_{0}$ is a perturbation such that  all the
critical points in $\Crit_{\pert_{0}}$ are non-degenerate and have
stabilizer $\pm 1$.   
Then there exists $\pert\in \Pert$ such that:
    \begin{enumerate}
        \item $f_{\pert}=f_{\pert_{0}}$ in a neighborhood of all the critical
        points of $\CS+f_{\pert_{0}}$;
        \item the set of critical points for these two perturbations
        are the same, so that $\Crit_{\pert} = \Crit_{\pert_{0}}$;
        \item for all critical points $\beta_{1}$ and $\beta_{0}$ in
        $\Crit_{\pert}$ and all paths $z$, the moduli spaces
        $M_{z}(\beta_{1},\beta_{0})$ for the perturbation $\pert$
        are regular. \qed
        \end{enumerate}
\end{proposition}

From this point on, we will always suppose that our perturbation has
been chosen in this way. For reference, we state:

\begin{hypothesis}
\label{hyp:morse-smale}
    We assume that $Y$ is a connected, oriented, closed $3$-manifold,
    that $K$ is a link in $Y$ (possibly empty), and that $\bP$ is
    singular bundle data satisfying the non-integral condition. 
    We suppose that an
    orbifold metric $\orbig$ and perturbation $\pert \in \Pert$ are
    chosen so that $\crit_{\pert}$ consists only of non-degenerate,
    irreducible critical points, and all the moduli spaces
    $M_{z}(\beta_{1},\beta_{0})$ are regular.
\end{hypothesis}

\subsection{Orientations and Floer homology}

The Fredholm operators $\calD_{A}$ form a family over the space
$\bonf_{k,z}(\beta_{1}, \beta_{0})$ whose determinant line
$\det(\calD_{A})$ is orientable. This follows from the corresponding
result for the closed pair $(X,\Sigma) = S^{1}\times (Y,K)$
(Proposition~\ref{prop:orientable}) by an application of excision. It
follows that the (regular) moduli space $M_{z}(\beta_{1},\beta_{0})$
is an orientable manifold. Moreover, given two different paths $z$ and
$z'$ between the same critical points, if we choose an orientation for
the determinant line over $\bonf_{k,z}(\beta_{1},\beta_{0})$, then it
canonically orients the determinant line over
$\bonf_{k,z'}(\beta_{1},\beta_{0})$ also: this follows from the
corresponding result for closed manifolds
(Proposition~\ref{prop:orientations-propogate}), 
because any two paths are related by the addition
of instantons and monopoles. We may therefore define
\[
         \Lambda(\beta_{1},\beta_{0})
\]
as the two-element set of orientations of $\det(\calD_{A})$ over
$\bonf_{k,z}(\beta_{1},\beta_{0})$, with the understanding that this
is independent of $z$.

If $\beta_{1}$ and $\beta_{0}$ are arbitrary connections in
$\bonf_{k}(Y,K,\bP)$ rather than critical points, then we can still
define $\Lambda(\beta_{1},\beta_{0})$ in essentially the same way. The
only point to take care of is that, if the Hessian of the perturbed
functional is singular at either $\beta_{1}$ or $\beta_{0}$, then the
corresponding operator $\calD_{A}$ is not Fredholm on the usual
Sobolev spaces. As in \cite{KM-knot-singular}, we adopt the convention
that $\calD_{A}$ is considered as a Fredholm operator acting on the
weighted Sobolev spaces
\begin{equation}\label{eq:weighted}
         e^{-\epsilon t} \orbiL^{2}_{k,A_{o}}         
\end{equation}
for a small positive weight $\epsilon$; and we then define
$\Lambda(\beta_{1},\beta_{0})$ for any $\beta_{i}$ using this
convention. 

In particular, we can choose any basepoint $\theta$ in
$\bonf_{k}(Y,K,\bP)$ and \emph{define}
\begin{equation}\label{eq:Lambda-beta}
         \Lambda(\beta) = \Lambda(\theta,\beta)
\end{equation}
for any critical point $\beta$. Without further input, there is no a
priori way to rid this definition of its dependence on $\theta$. In
the case considered in \cite{KM-knot-singular}, when $K$ was oriented
and $\Delta$ was trivial, we had a preferred choice of $\theta$
arising from a reducible connection. When $\Delta$ is non-trivial
however, the space $\bonf_{k}(Y,K,\bP)$ contains no reducibles, so an
arbitrary choice of $\theta$ is involved. 

Having defined $\Lambda(\beta)$ in this way, we have canonical identifications
\[
\Lambda(\beta_{1},\beta_{0}) = \Lambda(\beta_{1})\Lambda(\beta_{0}),
\]
where the product on the right is the usual product of 2-element sets
(defined, for example, as the set of bijections from
$\Lambda(\beta_{1})$ to $\Lambda(\beta_{2})$). What this implies is
that a choice of orientation for a component of the moduli space
$M_{z}(\beta_{1},\beta_{0})$ (or equivalently, a choice of
trivializations of the determinant on
$\bonf_{k,z}(\beta_{1},\beta_{0})$) determines an identification
$\Lambda(\beta_{1})\to\Lambda(\beta_{0})$. In particular, each
one-dimensional connected component  \[ [\breve A] \subset
M(\beta_{1},\beta_{0})_{1}, \] being
just a copy of $\R$ canonically oriented by the action of
translations, determines an isomorphism $\Lambda(\beta_{1}) \to
\Lambda(\beta_{0})$. As in \cite{KM-knot-singular,KM-book}, we denote by
$\Z\Lambda(\beta)$ the infinite cyclic group whose two generators are
the two elements of $\Lambda(\beta)$, and we denote by
\[
        \epsilon[\breve A] : \Z\Lambda(\beta_{1}) \to \Z\Lambda(\beta_{0})
\]
the resulting isomorphism of groups.

\medskip
We now have everything we need to define Floer homology groups. Let
$(Y,K)$ be an unoriented link in a closed, oriented, connected
$3$-manifold $Y$, and let $\bP$ be singular bundle data satisfying the
non-integral condition, Definition~\ref{def:non-int}. Let a metric
$\orbig$ and perturbation $\pert$ be chosen satisfying
Hypothesis~\ref{hyp:morse-smale}. Finally, let a basepoint $\theta$ in
$\bonf_{k}(Y,K,\bP)$ be chosen.
Then we define the chain complex $(C_{*}(Y,K,\bP),\partial)$ of
free abelian groups by setting
\begin{equation}\label{eq:chain-complex}
            C_{*}(Y,K,\bP) = \bigoplus_{\beta\in \Crit_{\pert}}
            \Z\Lambda(\beta),
\end{equation}
and
\begin{equation}\label{eq:boundary-map-def}
    \partial = \sum_{(\beta_{1},\beta_{0},z)} 
    \sum_{[\breve{A}]\subset M(\beta_{1},\beta_{0})_{1}}
    \epsilon[\breve{A}]
\end{equation}
where the first sum runs over all triples with
$\gr_{z}(\beta_{1},\beta_{0})=1$.  That this is a finite sum follows
from the compactness theorem, Corollary 3.25 of
\cite{KM-knot-singular}. (As emphasized in \cite{KM-knot-singular}, it
is the compactness result here depends crucially  on that fact 
that our choice of holonomy for our singular
connections satisfies a ``monotone'' condition: that is, the formula
for the dimension of moduli spaces in Lemma~\ref{lem:index} involves
the topology of the bundle $\bP$ only though the action $\kappa$.

\begin{definition}
\label{def:basic-I}
    For $(Y,K)$ as above, with singular bundle data $\bP$ satisfying
    the non-integral condition, Definition~\ref{def:non-int}, and
    choice of $\orbig$, $\pert$ and $\theta$ as above, we define the
    instanton Floer homology group
    \[
           I(Y,K,\bP)
    \]
   to be the homology of the complex $(C_{*}(Y,K,\bP),\partial)$.
  \CloseDef
\end{definition}

As usual, we have presented the definition of $I(Y,K,\bP)$ as
depending on some auxiliary choices. The standard type of cobordism
argument (using the material from the following subsection) shows that
$I(Y,K,\bP)$ is independent of the choice of $\orbig$, $\pert$ and
$\theta$. There is a slight difference from the usual presentation of
(for example) \cite{KM-knot-singular} however, which stems from our
lack of a canonical choice of basepoint $\theta$. The result of this
is that, if $(\orbig,\pert,\theta)$ and $(\orbig', \pert',\theta')$
are two choices for the auxiliary data, then the isomorphism between
the corresponding homology groups $I(Y,K,\bP)$ and $I'(Y,K,\bP)$ is
well-defined only up to an overall choice of sign.

\subsection{Cobordisms and manifolds with cylindrical ends}

Let $(W,S)$ be a cobordism of pairs, from $(Y_{1},K_{1})$ to $(Y_{0},
K_{0})$. We assume that $W$ is connected and oriented, but $S$ need
not be orientable. Let $\bP$ be singular bundle data on $(W,S)$, and
let $\bP_{i}$ be its restriction to $(Y_{i}, K_{i})$.  We shall recall
the standard constructions whereby $(W,S,\bP)$ induces a map on the
Floer homology groups,
\[
       I(W,S,\bP) : I(Y_{1}, K_{1}, \bP_{1}) \to  I(Y_{0}, K_{0},
       \bP_{0}),
\]
which is well-defined up to an overall sign.

To set this up, we assume  that auxiliary data is given,
\begin{equation}\label{eq:aux-data}
\begin{aligned}
    \aux_{1} &= (\orbig_{1}, \pert_{1}, \theta_{1}) \\
     \aux_{0} &= (\orbig_{0}, \pert_{0}, \theta_{0}),
\end{aligned}
\end{equation}
on $(Y_{1}, K_{1},\bP_{1})$ and $(Y_{0}, K_{0},\bP_{0})$ respectively,
and that our standing assumptions hold
(Hypothesis~\ref{hyp:morse-smale}). We equip the interior of $W$ with
an orbifold metric $\orbig$ (with orbifold singularity along $S$)
having two cylindrical ends
\[
\begin{aligned}
    (-\infty, 0\,] &\times \orbiY_{1} \\
    [\, 0,\infty) &\times \orbiY_{0} 
\end{aligned}
\]
where the metric $\orbig$ is given by
\[
\begin{gathered}
    dt^{2} + \orbig_{1}\\  dt^{2} + \orbig_{0}
\end{gathered}
\]
respectively.
To do this, we consider an open collar neighborhood $[\,0,1) \times
\orbiY_{1}$ of $\orbiY_{1}$, with coordinate $r$ for the first
factor, and we set $t =  \ln(r)$. Symmetrically, we take an open
collar $(-1,0\,]\times \orbiY_{0}$ at the other end, and set $t = -
\ln(-r)$ there.

Now we make the following choices (an adaptation, for the case of
manifolds with boundary, of the definition of $\bP$-marked bundle from
Definition~\ref{def:P-marked}. We choose singular bundle data
$\bP'$ on $(W,S)$ differing from $\bP$ by the addition of instantons
and monopoles, so that $\bP$ and $\bP'$ are identified outside a
finite set by a preferred map $g' : \bP'\to \bP$. 
We choose gauge representatives $B'_{1}$ and $B'_{0}$ for
$\beta_{1}$ and $\beta_{0}$, as connections in $\bP_{1}$ and $\bP_{0}$
respectively. And we choose an orbifold connection $A_{o}$  in $\bP'$
that coincides with the pull-back of $B'_{1}$ and $B'_{0}$ in the collar
neighborhoods of the two ends. We can then construct as usual an
affine space of connections
\[
       \cA_{k}(\bP',A_{o})
\]
consisting of all $A$ with 
\[
       A-A_{o} \in \orbiL_{k,A_{o}}(\orbiW, \Lambda^{1}\otimes \g_{\bP'})
\]
where the Sobolev space is defined using the cylindrical-end metric $\orbig$.
As in the case of closed manifolds (Definition~\ref{def:P-marked}), 
we say that choices $(\bP', B'_{1}, B'_{0})$ and $(\bP'', B''_{1},
B''_{0})$ are isomorphic if there is a determinant-1 gauge
transformation $\bP' \to \bP''$ pulling back $B''_{i}$ to
$B'_{i}$. (The notion of ``determinant 1'' has meaning here, because
both $\bP'$ and $\bP''$ are identified with $\bP$ outside a finite set.)
We denote by $z$ a typical isomorphism-classes of choices:
\[
         z = [\bP', B'_{1}, B'_{0}].
\]
The quotient of $\cA_{k}(\bP', A_{o})$ by the determinant-1 gauge
group $\G_{k+1}(\bP',A_{o})$ depends only on $z$, and we write
\[
       \bonf_{z}(W,S,\bP; \beta_{1}, \beta_{0})
           =  \cA_{k}(\bP',A_{o}) \big/ \G_{k+1}(\bP',A_{o})
\]

\begin{remark}
    The discussion here is not quite standard. If $S$ has no closed
    components, then every isomorphism class $z$ has a representative
    $[\bP, B_{0}, B_{1}]$ in which the singular bundle data is
    $\bP$. If $S$ has closed components however, then we must allow
    $\bP' \ne \bP$ in order to allow differing monopole charges on the
    closed components. In the case of a cylinder $(W,S) = I \times
    (Y,K)$, the set of $z$'s coincides with the previous space of
    homotopy-classes of paths from $\beta_{1}$ to $\beta_{0}$.
\end{remark}

 The perturbed version of
the ASD equations that we shall use is defined as follow $\pert_{i}$
be the chosen holonomy perturbations on the $Y_{i}$. We consider
perturbation of the $4$-dimensional equations on $\orbiW$
\begin{equation}\label{eq:fourdpertU}
F^{+}_{A} + \hat V(A)=0
\end{equation} 
where $\hat V$ is a holonomy-perturbation 
supported on the cylindrical ends. To define this on the cylindrical
end $[\,0,\infty) \times Y_{0}$, we take the given $\pert_{0}$ from
the auxiliary data $\aux_{0}$ and an additional term $\pert'_{0}$. We
then set
\begin{equation}\label{eq:fourdpertdef}
       \hat{V}(A) =  \phi(t)\hat{V}_{\pert_{0}}(A) +
       \psi(t)\hat{V}_{\pert'_{0}}(A),
\end{equation}
where $\phi(t)$ is a cut-off function equal to $1$ on $[\,1,\infty)$
and equal to $0$ near $t=0$, while $\psi(t)$ is a bump-function
supported in $[0,1]$. The perturbation is defined similarly on the
other end, using $\pert_{1}$ and an additional $\pert'_{1}$.
This sort of perturbation is used in
\cite[Section~24]{KM-book} and again in \cite{KM-knot-singular}. 
We write
\[
              M_{z}(W,S,\bP; \beta_{1}, \beta_{0}) \subset 
             \bonf_{k,z}(W,S,\bP; \beta_{1}, \beta_{0})
\]
for the moduli 
space of solutions to the perturbed anti-self-duality equations. Note
that $\pert'_{i}$ contributes a perturbation that is compactly
supported in the cobordism. We refer to these additional terms as
\emph{secondary perturbations.}

Just as in the cylindrical case, we can combine the linearization of
the left hand side of \eqref{eq:fourdpertU} with gauge fixing to
obtain a Fredholm operator $\calD_{A}$.  We say that the moduli space
is regular if $\calD_{A}$ is surjective at all solutions. Regular
moduli spaces are smooth manifolds, of dimension equal to the index of
$\calD_{A}$, and we write as
\[
        \gr_{z}(W,S,\bP ; \beta_{1},\beta_{0}).
\]
We again set
\[
              M(W,S,\bP; \beta_{1}, \beta_{0})_{d}
     = \bigcup_{\gr_{z}=d}  M_{z}(W,S,\bP; \beta_{1}, \beta_{0}).
\]
The arguments use to prove Proposition 24.4.7 of \cite{KM-book}
can be used to prove the following genericity result.

\begin{proposition}
    Let $\pert_{1}$, $\pert_{0}$ be perturbations satisfying
    Hypothesis~\ref{hyp:morse-smale}.  Let $(W,S,\bP)$ be given,
    together with a cylindrical-end metric $\orbig$ as above. 
    Then there are secondary
    perturbations  $\pert'_{1}$, $\pert'_{0}$
    so that for all $\beta_{1}$, $\beta_{0}$ and $z$, all the moduli
    spaces $M_{z}(W,S,\bP;\beta_{1},\beta_{0})$ of solutions to the
    perturbed equations  \eqref{eq:fourdpertU}
    are regular. \qed
\end{proposition}

\subsection{Maps from cobordisms}
\label{subsec:maps-from-cob}

The moduli spaces $M_{z}(W,S,\bP;\beta_{1},\beta_{0})$ are orientable,
because the determinant line $\det(\calD_{A})$ over
$\bonf_{z}(W,S,\bP;\beta_{1},\beta_{0})$ is trivial. Furthermore, when
we 
orient $\det(\calD_{A})$ for one particular $[A]$ in
$\bonf_{z}(W,S,\bP;\beta_{1},\beta_{1})$, then this canonically
determines an orientation for the moduli spaces
$M_{z'}(W,S,\bP;\beta_{1},\beta_{0})$ for all other $z'$. 

To see what is involved in specifying an orientation for the moduli
spaces, let us recall first that we have chosen basepoints
$\theta_{i}$ in $\bonf(Y_{i},K_{i};\bP_{i})$ for $i=0,1$, as part of
the auxiliary data $\aux_{i}$. We have defined $\Lambda(\beta_{i})$ to
be $\Lambda(\theta_{i},\beta_{i})$. Choose gauge representatives
$\Theta_{i}$ for the connections $\theta_{i}$, an let
$A_{\theta}$ be a connection on $(W,S,\bP)$ which is equal to the
pull-back of $\Theta_{i}$ on the two cylindrical ends. The operator
$\calD_{A_{\theta}}$ is Fredholm on the weighted Sobolev spaces
\eqref{eq:weighted} for small $\epsilon$, and we define
\[
           \Lambda(W,S,\bP)
\]
to be the two-element set of orientations of
$\det(\calD_{A_{\theta}})$. (This set is dependent on $\theta_{1}$ and
$\theta_{0}$, though our notation hides this.) In this context, we
make the a definition:

\begin{definition}\label{def:I-orientation}
   Let $(Y_{1}, K_{1},\bP_{1})$ and 
   $(Y_{0}, K_{0},\bP_{0})$ be manifolds with singular bundle data,
   and let  $(W,S,\bP)$ be a  cobordism from the first to the second.
   Let auxiliary data $\aux_{1}$,
   $\aux_{0}$ be given on the two ends, as above. We define an
   \emph{$I$-orientation} of $(W,S,\bP)$ to be a choice of element
   from $\Lambda(W,S,\bP)$, or equivalently, an orientation of the
   determinant line $\det(\calD_{A_{\theta}})$.  \CloseDef
\end{definition}

The definition of $\Lambda(W,S,\bP)$ 
is constructed so that the two-element set of
orientations of the determinant line over $\bonf_{z}(W,S,\bP;
\beta_{1},\beta_{0})$ is isomorphic to the product
\[
     \Lambda(\beta_{1})\Lambda(\beta_{0}) \Lambda(W,S,\bP).
\]
Thus, once an $I$-orientation of $(W,S,\bP)$ is given, a choice of
orientation for a component of any moduli space
$M_{z}(W,S,\bP;\beta_{1},\beta_{0})$ determines an isomorphism
$\Z\Lambda(\beta_{1})\to\Z\Lambda(\beta_{0})$. In particular each
point $[A]$ in a zero-dimensional moduli space $M_{0}(W,S,\bP ;
\beta_{1},\beta_{0})$ determines such an isomorphism,
\[
     \epsilon([A]) : \Z\Lambda(\beta_{1}) \to  \Z\Lambda(\beta_{0}).
\]

In this way, given a choice of an $I$-orientation of $(W,S,\bP)$, 
we obtain a homomorphism
\begin{equation}\label{eq:chain-map-def}
    m = \sum_{\beta_{1},\beta_{0}} 
    \sum_{[A]\in M(W,S,\bP; \beta_{1},\beta_{0})_{0}}
    \epsilon([A]),
\end{equation}
which by the usual arguments is a chain map
\[
          m: C_{*}(Y_{1},K_{1},\bP_{1}) \to  C_{*}(Y_{0},K_{0},\bP_{0}).
\]
The induced map in homology depends only on $(W,S,\bP)$, the
original auxiliary data $\aux_{0}$ and $\aux_{1}$ on the two ends, and
the $I$-orientation, not
on the choice of the secondary perturbations $\pert'_{i}$ or on the
choice of Riemannian metric $\orbig$ on the interior of
$W$. Furthermore, if $(W,S,\bP)$ is expressed as the union of two
cobordisms $(W',S',\bP')$ and $(W'', S'', \bP'')$, then the chain map
$m$ is chain-homotopic to the composite,
\[
        m \simeq m'' \circ m'.
\]
 By the
standard approach, taking $(W,S,\bP)$ to be a cylinder, we deduce that
$I(Y_{0},K_{0},\bP_{0})$ is independent of the choice of auxiliary
data $\aux_{0}$. More precisely, if $\aux_{0}$ and $\aux_{1}$ are two
choices of auxiliary data for $(Y,K,\bP)$, then there is a canonical
pair of isomorphisms $\{\,m_{*}, -m_{*}\,\}$ differing only in sign,
\[
         \pm m_{*}: I(Y,K,\bP)_{\aux_{1}} \to I(Y,K,\bP)_{\aux_{0}}.
\]
Thus $I(Y,K,\bP)$ is a topological invariant of $(Y,K,\bP)$. Note that
we have no a priori way of choosing $I$-orientations to resolve the
signs in the last formulae: the dependence on the auxiliary data
$\aux_{i}$ means that cylindrical cobordisms do not have canonical
$I$-orientations.

\subsection{Families of metrics and compactness}
\label{subsec:families}

The proof from \cite{KM-knot-singular} that the chain-homotopy class
of the map $m$ above is independent of the choice of $\pert_{i}'$ and
the metric on the interior of $W$ follows standard lines and exploits
a parametrized moduli space, over a family of Riemannian metrics and
perturbations. For our later applications, we will need to consider
parameterized moduli spaces where the metric is allowed to vary in
certain more general, controlled non-compact families.

Let $(Y_{1},K_{1}, \bP_{1})$ and $(Y_{0}, K_{0}, \bP_{0})$ be given
and let $\aux_{1}$ and $\aux_{0}$ be auxiliary data as in
\eqref{eq:aux-data}, so that the transversality conditions of
Proposition \ref{prop:4d-transversality} hold.  Let $(W,S,\bP)$ be a
cobordism between these, equipped with singular bundle data $\bP$
restricting to the $\bP_{i}$ at the ends.  By a family of metrics
parametrized by a smooth manifold $G$ we mean a smooth orbifold
section of $\mathrm{Sym}^2(T^*\orbiW)\times G \to \orbiW \times G$
which restricts to each $W\times \{\orbig\}$ as a Riemannian metric
denoted $\orbig$.  These metrics will always have an orbifold
singularity along $S$. Choose a family of secondary perturbations
$\pert'_{i,\orbig}$, supported in the collars of the two boundary
components as before, but now dependent on the parameter $\orbig\in
G$. We have then a corresponding perturbing term $\hat V_{g}$ for the
anti-self-duality equations on $\orbiW$.  For any pair of critical
points $(\beta_{1}, \beta_{0}) \in \crit_{\pert_{1}}(Y_{1}) \times
\crit_{\pert_{2}}(Y_{2})$, we can now form a moduli space
\[
     M_{z,G}(W,S,\bP;\beta_{1},\beta_{0}) \subset 
                 \cB_{z}(W,S,\bP;\beta_{1},\beta_{0}) \times G
\]
of pairs $([A],g)$ 
and where $[A]$
solves the equation
\begin{equation}\label{eq:fourdpertG}
  F^{+_{g}}_{A}+ \hat V_{g}(A)=0.
\end{equation} 
Here $+_{g}$ denote the projection onto $g$-self-dual two forms.  A
solution $([A],g)$ is called to Equation~\eqref{eq:fourdpertG} is
called regular if the differential of the map
\[
(A,g)\mapsto F^{+_{g}}_{A}+ \hat V_{g}(A)
\]
is surjective.
The arguments use to prove Propostion 24.4.10 of \cite{KM-book}
can be used to prove the following genericity result.

\begin{proposition}
    Let $\pert_{i}$ be perturbations satisfying the conclusions of
    Proposition \ref{prop:4d-transversality} and let $G$ be a family
    of metrics as above.  There is a family of secondary perturbations
    $\pi'_{i,g}$, for $i=0,1$, parameterized by $g\in G$ so that for all
    $\beta_{i}\in \Crit_{\pi_{i}}$ and all paths $z$, the moduli space
\[
     M_{z,G}(W,S,\bP;\beta_{1},\beta_{0})
\]
consists of regular solutions.  \qed
\end{proposition}

In the situation of the proposition, the moduli space
$M_{z,G}(W,S,\bP;\beta_{1}, \beta_{0})$ is smooth of dimension
$\gr_{z}(\beta_{1},\beta_{0}) + \dim G$.  We use
$M_{G}(W,S,\bP;\beta_{1}, \beta_{0})_{d}$ to denote its
$d$-dimensional components. To orient the moduli space,
we orient both the determinant line bundle $\det(\calD)$ on
$\bonf_{z}(W,S,\bP;\beta_{1},\beta_{0})$ and the parametrizing
manifold $G$, using a fiber-first convention.

Consider now the case that $G$ is a compact, oriented manifold with
oriented boundary $\partial G$. We omit $(W,S,\bP)$ from our notation
for brevity, and denote the moduli space by
$M_{G}(\beta_{1},\beta_{0})$. Let us suppose (as we may) that the
secondary perturbations are chosen so that both
$M_{G}(\beta_{1},\beta_{0})$ and $M_{\partial G}(\beta_{1},\beta_{0})$
are regular. The first moduli space will then be a (non-compact)
manifold with boundary, for we have
\[
\partial M_{G}(\beta_{1},\beta_{0})_{d} = M_{\partial
  G}(\beta_{1},\beta_{0})_{d-1}.
\]
But this will not be an equality of \emph{oriented} manifolds. Our
fiber-first convention for orienting $M_{G}$ and the standard
outward-normal-first convention for the boundary orientations interact
here to give
\begin{equation}
    \label{eq:boundary-M}
      \partial M_{G}(\beta_{1},\beta_{0})_{d} =(-1)^{d-\dim G} 
     M_{\partial G}(\beta_{1},\beta_{0})_{d-1}.
\end{equation}

When an orientation of $G$ and an element of $\Lambda(W,S,\bP)$ are
chosen, the count of the solutions $[A]$ in
$M_{G}(\beta_{1},\beta_{0})_{0}$ defines a group homomorphism
\[
      m_{G} :  C(Y_{1},K_{1},\bP_{1}) \to C(Y_{0}, K_{0}, \bP_{0}).
\]
Similarly $M_{\partial G}(\beta_{1},\beta_{0})_{0}$ defines a group
homomorphism $m_{\partial G}$. These two are related by the following
chain-homotopy formula:
\[
  m_{\partial G} +  (-1)^{\dim G} m_{G}\circ \partial =  \partial \circ m_{G}.
\]
 (See \cite[Proof of Proposition 25.3.8]{KM-book}, though there is a
 sign error in \cite{KM-book} at this point.)
The proof of this formula is to count the endpoints in the
$1$-dimensional moduli spaces $M_{G}(\beta,\alpha)_{1}$ on $(W,S,\bP)$.

\begin{remarks}
    Although we will not have use for greater generality here, it is a
    straightforward matter here to extend this construction by
    allowing $G$ to parametrize not just a family of metrics on a
    fixed cobordism $(W,S)$, but a smooth family of cobordisms (a
    fibration over $G$) with fixed trivializations of the family at
    the two ends $Y_{1}$ and $Y_{0}$. There is an obvious notion of an
    $I$-orientation for such a family (whose existence needs to be a
    hypothesis), and one then has a formula just like the one above
    for a compact family $G$ with boundary. One can also consider the
    case that $G$ is a simplex in a simplicial complex $\Delta$. In
    that case, given a choice of perturbations making the moduli
    spaces transverse over every simplex, and a coherent choice of
    $I$-orientations of the fibers, the above formula can be
    interpreted as saying that we have a chain map
\[
             \underbar{m} : C(\Delta) \otimes C(Y_{1},K_{1},\bP_{1})
             \to  C(Y_{0}, K_{0}, \bP_{0})
\]
   where $C(\Delta)$ is the simplicial chain complex (and the usual
   convention for the signs of the differential on a product complex
   apply). If the local system defined by the $I$-orientations of the
   fibers is non-trivial, then there is a similar chain map, but we
   must then use the chain complex $C(\Delta;\xi)$ with the
   appropriate local coefficients $\xi$.
\end{remarks}

Next we wish to generalize some of the stretching arguments that are
used (for example) in proving the composition law for the chain-maps
$m$ induced by cobordisms.  For this purpose we introduce the notion
of a broken Riemannian metric on a cobordism $(W,S)$ from
$(Y_{1},K_{1})$ to $(Y_{0},K_{0})$.  A \emph{cut} of $(W,S)$ is an orientable
codimension-1 submanifold $Y_{c} \subset \mathrm{int}( W )$ so that the intersection
$Y_{c}\cap S= K_{c}$ is transverse.  A cylindrical-end metric $\orbig$
on $(W,S)$ is \emph{broken along a cut} $ Y_{c}$ if it is a
complete Riemannian metric on $({\mathrm{int}} W)\sminus Y_{c}$ and there is
a normal coordinate collar neighborhood $(-\epsilon,\epsilon)\times Y_{c}$
with normal coordinate $r_{c}\in (-\epsilon,\epsilon)$ so that
\[
g= (d r_{c} / r_{c})^{2}+\orbig_{Y_{c}}
\] 
where $\orbig_{Y_{c}}$ is a metric on $Y_{c}$ with orbifold singularity along
$K_{c}$.  Note that $Y_{c}$ may be disconnected, may have parallel components, and
may have components parallel to the boundary.  The manifold 
$({\mathrm{int}} W)\sminus Y_{c}$ equipped with this metric has two
extra cylindrical ends (each perhaps with several components), namely
the ends
\begin{equation}\label{eq:extra-ends}
\begin{aligned}
    (-\epsilon, 0) &\times Y_{c} \\
    (0, \epsilon) &\times Y_{c}.
\end{aligned}
\end{equation}
We will assume that each component of $(Y_{c}, K_{c},\bP|_{Y_{c}})$
satisfies the non-integral condition.

Given perturbations $\pi_{i}$
for the $(Y_{i},K_{i})$ (for $i=0,1$)  and $\pi_{c}$ for $(Y_{c},K_{c})$ and secondary
perturbations $\pi_{i}'$ and $\pi_{c,+}'$ and  $\pi_{c,-}'$, 
we can write down perturbed
ASD-equations on ${\mathrm{int}}(W)\sminus Y_{c}$ as in
\eqref{eq:fourdpertU}. The perturbation $\hat V$ is defined as in
\eqref{eq:fourdpertdef}, using the perturbation $\pi_{c}$ together with the secondary
perturbations $\pi'_{c,-}$, $\pi'_{c,+}$ on the two ends
\eqref{eq:extra-ends} respectively.

Given a pair of configurations $\beta_{i}$ in $\bonf(Y_{i},K_{i},
\bP_{i})$, for $i=1,0$, and a cut $Y_{c}$, a \emph{cut path} from
$\beta_{1}$ to $\beta_{0}$ along $(W,S,\bP)$ is a continuous connection
$A$ in a  $\bP' \to (W,S)$ in singular bundle data $\bP'$ equivalent
to $\bP$, such that $A$ is smooth on
${\mathrm{int}}(W)\sminus Y_{c}$ and its restriction to
$(Y_{i},K_{i})$ belongs to the gauge equivalence class $\beta_{i}.$ A
\emph{cut trajectory} from $\beta_{1}$ to $\beta_{0}$ along $(W,S,\bP)$
is a cut path from $\beta_{1}$ to $\beta_{0}$ along $(W,S,\bP)$ whose
restriction to ${\mathrm{int}}(W)\sminus Y_{c}$ is a solution to
the perturbed ASD equation \eqref{eq:fourdpertU}.

Given a cut $Y_{c}$ of $(W,S)$ we can construct a family of
Riemannian metrics on $\mathrm{int}(W)$ which degenerates to a broken
Riemannian metric. To do this, we start with  a Riemannian metric
$\orbig_{o}$ on $W$ (with orbifold singularity along $S$) which
contains a collar neighborhood of $Y_{c}$ on which the metric is a
product
\[
d r^{2} + \orbig_{Y_{c}}
\]
where $r\in [-1,1]$ denotes the signed distance from $Y_{c}$.
Let 
\[
        f_{s} : \R \to \R
\]
be a family of functions parametrized by $s\in [0, \infty]$ that
smooths out the function which is given by
\[
\frac{1+1/s^{2}}{r^{2}+1/s^{2}}
\]
for $r \in [-1,1]$ and $1$ otherwise.  Note that the above expression
is $1$ on the boundary $r=\pm 1$ as well as when $s=0$.  Also note
that $\lim_{s\mapsto \infty}f_{s}(r)=1/r^{2}$.  For each component of
\[ Y_{c}=\coprod _{i=1}^{N}Y_{c}^{i}\] we introduce a parameter $s_{i}$
and modify the metric by
\[
f_{s_{i}}(r)dr^{2} +\orbig_{Y_{c}}.
\]
When $s_{i}=\infty$ the metric is broken along the $Y^{i}_{c}$.  
In the way we get a family of Riemannian metrics
parametrized by $[0,\infty)^{N}$ which compactifies naturually
to a family of broken metrics parameterized by 
$[0,\infty]^{N}.$ In the neighborhood of $\{\infty\}^{N}$, each metric
is broken along some subset of the components of $Y_{c}$. We can
further elaborate this  construction slightly by allowing the original
metric also to vary in a family $G_{1}$, while remaining unchanged in the collar
neighborhood of $Y_{c}$, so that we have a family of metrics
parametrized by
\begin{equation}\label{eq:corners}
           [0,\infty]^{N} \times G_{1}
\end{equation}
for some $G_{1}$. Given the cut $Y_{c}$, we say that a family of
singular metrics on $(W,S)$ is a ``model family'' for the cut $Y_{c}$ if
it is a family of this form.

To describe suitable perturbations for the equations over such a model
family of singular metrics, we choose again a generic perturbation
$\pert_{c}$ for $Y_{c}$ and write its component belonging to
$Y_{c}^{i}$ as $\pert_{c}^{i}$. When the coordinate $s_{i}\in
[0,\infty]$ is large, the metric contains a cylindrical region
isometric to a product
\[
          [-T_{i} ,T_{i} ] \times Y_{c}^{i}
\]
where $T_{i} \to \infty$ as $s_{i} \to\infty$. We require that for
large $s_{i}$, the perturbation on this cylinder has the form
\[
    \hat{V}_{\pert_{c}^{i}}  +
       \psi_{-}(t)\hat{V}_{(\pert^{i}_{c,-})'} + 
  \psi_{+}(t)\hat{V}_{(\pert^{i}_{c,+})'}
\]
where the functions $\psi_{+}$ and $\psi_{-}$ are bump-functions
supported near $t=-T_{i}$ and $t=T_{i}$ respectively. Here
$(\pert^{i}_{c,\pm})'$ are secondary perturbations which are allowed to
vary with the extra parameters $G_{1}$.

We can now consider a general family of metrics which degenerates like
the model family at the boundary. Thus we consider a manifold with
corners, $G$, parametrizing a family of broken Riemannian metrics $\orbig$ on
$W$ (with orbifold singularity along $S$), and we ask that in the
neighborhood of each point  of every codimension-$n$ facet, there
should be a cut $Y_{c}$ with exactly $n$ components, so that in the
neighborhood of this point the family is equal to a neighborhood of
\[ \{\infty\}^{n}\times G_{1}\] in some model family for the cut
$Y_{c}$. (Note that the cut $Y_{c}$ will vary.)  Whenever we talk of a
``family of broken metrics'', we shall mean that the family has this
model structure at the boundary. When considering perturbations, we
shall need to have fixed perturbations $\pert_{c}^{i}$ for every
component of every cut, and also secondary perturbations which have
the form described above, in the neighborhood of each point of the
boundary. We suppose that these are chosen so that the parametrized
moduli spaces over all strata of the boundary are regular.

Suppose now that we are given a family of broken metrics of this sort,
parametrized by a compact manifold-with-corners $G$. Over the interior
$\mathrm{int}(G)$, we have a smooth family of metrics and hence a
parametrized moduli space 
\[
M_{z,\mathrm{int}(G)}(\beta_{1},\beta_{0})
\] on the
cobordism $(W,S,\bP)$.  This moduli space has a natural
completion involving cut paths (in the above sense) on $(W,S)$,
as well as broken trajectories on the cylinders over $(Y^{i}_{c},
K^{i}_{c})$ for each component $Y^{i}_{c}\subset Y_{c}$. We denote the
completion by
\[
      \Mbk_{z,G}(\beta_{1},\beta_{0}) \supset M_{z,\mathrm{int}(G)}(\beta_{1},\beta_{0}).
\]

The completion is a space stratified by manifolds, whose top
stratum is $M_{z,\mathrm{int}(G)}(\beta_{1},\beta_{0})$. The
codimension-1 strata are of three sorts. 

\begin{enumerate}
\item \label{item:boundary-a}
 First, there are the cut paths on $(W,S,\bP)$ from $\beta_{1}$
    to $\beta_{0}$, cut along some \emph{connected} $Y_{c}$. These
    form a codimension-1 stratum of the compactification lying over a
    codimension-1 face of $G$ (the case $N=1$ in \eqref{eq:corners}).
\item  \label{item:boundary-b}
Second there are the strata corresponding to a trajectory sliding
    off the incoming end of the cobordism, having the form
    \[
                    \Mu_{z_{1}}(\beta_{1},\alpha_{1}) \times M_{z-z_{1},G}(\alpha_{1},\beta_{0})
     \]
     where the first factor is a moduli space of trajectories on
     $Y_{1}$ and $\alpha_{1}$ is a critical point.
\item  \label{item:boundary-c}
Third there is the symmetrical case of a trajectory sliding off
    the outgoing end of the cobordism:
    \[
                   M_{z-z_{0},G}(\beta_{1},\alpha_{0}) \times 
                 \Mu_{z_{0}}(\alpha_{0},\beta_{0}).
     \]
\end{enumerate}
In the neighborhood of a point in any one of these codimension-1
strata, the compactification $\Mbk_{z,G}(\beta_{1},\alpha_{0})$ has
the structure of a $C^{0}$ manifold with boundary. 

The completion $\Mbk_{z,G}(\beta_{1},\beta_{0})$ is not in general
compact, because of bubbling off of instantons and monopoles. However
it will be compact when it has dimension less than $4$.

An $I$-orientation for $(W,S,\bP)$, together with a choice of element
from $\Lambda(\beta_{1})$ and $\Lambda(\beta_{0})$, gives rise not only
to an orientation of the moduli space
$M_{z}(W,S\bP;\beta_{1},\beta_{0})$ for a fixed smooth metric
$\orbig$, but also to an orientation of the moduli spaces of cut
paths, for any cut $Y_{c}$, by a straightforward generalization of the
composition law for $I$-orientations. Thus, given such an
$I$-orientation and an orientation of $G$, we obtain oriented moduli
spaces over both $G$ and over the codimension-1 strata of $\partial G$. 

Just as in the case of a smooth family of metrics parametrized by a
compact manifold with boundary, a family of broken Riemannian metrics
parametrized by an oriented manifold-with-corners $G$, together with
an $I$-orientation of $(W,S,\bP)$, gives rise to a chain-homotopy
formula,
\[
  m_{\partial G} +  (-1)^{\dim G} m_{G}\circ \partial =  \partial \circ m_{G}.
\]
Again, the proof is obtained by considering the endpoints of
$1$-dimensional moduli spaces over $G$:
the three terms correspond to the three different types of
codimension-1 strata. The map $m_{G}$ is defined as above by counting
solutions in zero-dimensional moduli spaces over $G$; and
$m_{\partial G}$ is a sum of similar terms, one for each codimension-1
face of $G$ (with the outward-normal first convention). 

The case of most interest to us is when each face of $G$ corresponds
to a cut $Y_{c}$ whose single connected component separates $Y_{1}$
from $Y_{0}$, so separating $W$ into two cobordisms: $W'$ from $Y_{1}$
to $Y_{c}$, and $W''$ from $Y_{c}$ to $Y_{0}$. Let us suppose that
$G$ has $l$ codimension-1 faces, $G_{1}, \dots, G_{l}$, all of this
form. Each face corresponds to a cut which expresses $W$ as a union,
\[
           W = W'_{j} \cup W''_{j}, \quad j=1,\dots, l.
\]
 In the neighborhood of
such a point of $G_{j}$, $G$ has the structure
\[
           (0,\infty] \times G_{j}.
\]
Let us suppose also that $G_{j}$ has the form of a product,
\[ G_{j} = G'_{j} \times G''_{j},\] 
where the two factors parametrize families of
metrics on $W'_{j}$ and $W''_{j}$. Let us also equip $W'_{j}$ and
$W''_{j}$ with $I$-orientations so that the composite
$I$-orientation is that of $W$. We can write the chain-homotopy
formula now as
\[
\Bigl( \sum_{j }m_{G_{j}} \Bigr) + (-1)^{\dim G} m_{G}\circ \partial
= \partial \circ m_{G},
\]
where we are still orienting $G_{j}$ as part of the boundary of $G$.
On the other hand, we can interpret $m_{G_{j}}$ (which counts cut
paths on $W$) as the composite of the two maps obtained from the
cobordisms $W'_{j}$ and $W''_{j}$. When we do so, there is an
additional sign, because of our convention that puts the $G$ factors
last: we have
\[
       m_{G_{j}} =(-1)^{\dim G_{j}''\dim G_{j}'} m''_{j} \circ m'_{j}
\]
where $m'_{j}$ and $m''_{j}$ are the maps induced by the two
cobordisms with their respective families of metrics. Thus the
chain-homotopy formula can be written
\begin{equation}
    \label{eq:chain-homotopy-faces}
    \Bigl( \sum_{j }(-1)^{\dim G_{j}''\dim G_{j}'} m''_{j} \circ m'_{j} \Bigr) + (-1)^{\dim G} m_{G}\circ \partial
    = \partial \circ m_{G}.  
\end{equation}

\section{Topological constructions}
\label{sec:classical}

In this section, we shall consider how to ``package'' the
constructions of section~\ref{sec:sing-inst-floer}, so as to have a
functor on a cobordism category whose definition can be phrased in
terms of more familiar topological notions.

\subsection{Categories of 3-manifolds, bundles and cobordisms}

When viewing Floer homology as a ``functor'' from a category of
3-manifolds and cobordisms to the category of groups, one needs to
take care in the definition of a ``cobordism''. Thus, by a cobordism
from a 3-manifold $Y_{1}$ to $Y_{0}$ we should mean a
$4$-manifold $W$ whose boundary $\partial W$ comes equipped with a
diffeomorphism $\phi : \partial W \to Y_{1} \cup Y_{0}$. If the
manifolds $Y_{i}$ are oriented (as our $3$-manifolds always are), then
$(W,\phi)$ is an \emph{oriented cobordism} if $\phi$ is
orientation-reversing over $Y_{1}$ and orientation-preserving over
$Y_{0}$. Two oriented cobordisms $(W,\phi)$ and $(W',\phi')$ between
the same $3$-manifolds are isomorphic if there is an
orientation-preserving diffeomorphism $W\to W'$ intertwining $\phi$
with $\phi'$. Closed, oriented $3$-manifolds form the objects of a
category in which the morphisms are isomorphism classes or oriented
cobordisms. We can elaborate this idea a little, to include
$\SO(3)$-bundles over our $3$-manifolds, defining a category $\buncob$
as follows.

An object of the category  $\buncob$
consists of the following data:
\begin{itemize}
\item a closed, connected, oriented $3$-manifold $Y$;
\item an $\SO(3)$-bundle $P\to Y$ satisfying the \emph{non-integral
    condition} (so that $w_{2}(P)$ has odd evaluation on some oriented surface).
\end{itemize}
A morphism from $(Y_{1}, P_{1})$ to $(Y_{0},  P_{0})$
is an \emph{equivalence class} of data of the following sort:
\begin{itemize}
\item an oriented cobordism $W$ from $Y_{1}$ to
    $Y_{0}$;
\item a bundle $P$ on $W$;
\item an isomorphism of the bundle  $P|_{\partial W}$ with
    $P_{1} \cup P_{0}$, covering the given diffeomorphism of the
    underlying manifolds.
\end{itemize}
Here we say that $(W,P)$ and $(W',P')$ are \emph{equivalent} (and are the
same morphism in this category) 
if there is a diffeomorphism
\[
  \psi: W \to W'
\]
intertwining the given diffeomorphism of $\partial W$ and $\partial
W'$ with $-Y_{0} \cup Y_{1}$, together a bundle isomorphism on the
complement of a finite set $x$,
\[
            \Psi : P|_{W\setminus x} \to P'_{W'\setminus \psi(x)},
\]
covering the map $\psi$ and intertwining the given bundle maps at the
boundary.

The category $\buncob$ is essentially
the category for which Floer constructed his instanton homology
groups: we can regard instanton homology as a \emph{projective
  functor} $I$ from $\buncob$ to groups. Here, the word
``projective'' means that we regard the morphism $I(W,P)$
corresponding to a cobordism as being defined only up to an overall
sign. Alternatively said, the target category is the category
$\pgroup$ of
abelian groups in which the morphisms are taken to be unordered
pairs $\{h,-h\}$
of group homomorphisms, so that $I$ is a functor
\[
     I : \buncob \to \pgroup
\]

\medskip
On a $3$-manifold $Y$ a bundle $P$ is determined by its
Stiefel-Whitney class; but knowing the bundle only up to
isomorphism is not sufficient to specify an actual object in 
$\buncob$, nor can we make any useful category whose objects are such
isomorphism classes of bundles. However, if instead of just specifying
the Stiefel-Whitney class we specify a particular geometric
representative for this class, then we will again have a sensible
category with which to work, as we now explain.

Consider first the following situation. Let $\omega$ be a finite, $2$-dimensional
simplicial complex and let $V$ be a $2$-disk bundle over $\omega$ (not
necessarily orientable). Let $T\subset V$ be the circle bundle and
$V/T$ the Thom space. There is a distinguished class $\tau$ in
$H^{2}(V,T; \Z/2)$, the Thom class. If $(P,\phi)$ is an $\SO(3)$
bundle on $V$ with a trivialization $\phi$ on $S$, then it has a
well-defined relative Stiefel-Whitney class in $H^{2}(V,T;\Z/2)$. We
can therefore seek a pair $(P,\phi)$ whose relative Stiefel-Whitney class
is the Thom class. By the usual arguments of obstruction theory
applied to the cells of $V/T$, one sees that such a $(P,\phi)$ exists,
and that its isomorphism class is unique up to the action of adding
instantons to the $4$-cells of $V$.  Furthermore, if $(P',\phi')$ is
already given over some subcomplex $\omega' \subset \omega$, then it can be
extended over all of $\omega$. 
If  $\omega$ is a closed surface, perhaps non-orientable, in a
$4$-manifold $X$, then we can apply the observation of this
paragraph to argue (as in the proof of the theorem of Dold and Whitney
\cite{Dold-Whitney}) that there exists bundle $P \to X$ with a
trivialization outside a tubular neighborhood of $\omega$, whose
relative Stiefel-Whitney class is dual to $\omega$.  

This leads us to consider a category in which an object is a closed,
connected, oriented 
$3$-manifold $Y$ together with an embedded, unoriented $1$-manifold
$\omega\subset Y$ (thought of as a dual representative for $w_{2}$ of
some bundle). A morphism in this category, from $(Y_{0},\omega_{0})$
to $(Y_{1},\omega_{1})$ is simply a cobordism of pairs $(W,\omega)$,
with $W$ an oriented cobordism, and $\omega$ unoriented. Morphisms are
composed in the obvious way. We call this category $\buncobw$. 

It is tempting to suppose that there is a functor from $\buncobw$ to
$\buncob$, which should assign to an object
$(Y,\omega)$ a pair $(Y,P)$ where $P$ is a bundle with $w_{2}(P)$ dual
to $\omega$. As it stands, however, $P$ is unique only up to
isomorphism. To remedy this, we analyze the choice involved. 
Given an object $(Y,\omega)$ in
$\buncobw$, let us choose a tubular neighborhood $V$ for $\omega$ and choose
a bundle $(P_{0},\phi_{0})$, with a trivialization $\phi_{0}$ outside
$V$, whose relative $w_{2}$ is the Thom class. 
If $(P_{1},\phi_{1})$
is another such choice, then the objects $(Y,P_{1})$ and $(Y,P_{0})$
are connected by a uniquely-determined morphism in
$\buncob$. Indeed, in the cylindrical cobordism $[0,1]\times Y$,
we have a product embedded surface $[0,1]\times \omega$, and there is an
$\SO(3)$ bundle $P$ with trivialization outside the tubular
neighborhood of $[0,1]\times \omega$, which extends the given data at
the two ends of the cobordism. The resulting bundle is unique up to
the addition of instantons, so it gives a well-defined morphism -- in
fact an isomorphism -- in
$\buncob$. 

In this way, from an object $(Y,\omega)$  in $\buncobw$, we obtain not an
object in $\buncob$, but a commutative diagram in
$\buncob$, consisting of all possible bundles $P_{i}$ arising
this way and the canonical morphisms between them. Applying the
projective functor $I$, we nevertheless obtain from this a
well-defined projective functor,
\[
          I : \buncobw \to \pgroup
\]
We shall write $I^{\omega}(Y)$ for the instanton homology of
$(Y,\omega)$  in this context.

\subsection{Introducing links: singular bundles}

The discussion from the previous subsection can be carried over to the
more general situation in which a we have a knot or link $K \subset
Y$, rather than a $3$-manifold $Y$ alone.  We can define a category
$\buncoblink$ in which objects are pairs $(Y,K)$ carrying singular
bundle data $\bP$.  Thus $Y$ is again a closed, oriented, connected
$3$-manifold, $K$ is an unoriented link in $Y$, and we are given singular
bundle data $\bP$ in the form of a double-cover $K_{\Delta}\to K$, an
$\SO(3)$ bundle $P_{\Delta}\to K_{\Delta}$ and an $O(2)$ reduction in
the neighborhood of $K_{\Delta}$. The singular bundle data is required
to satisfy the non-integral condition from
Definition~\ref{def:non-int}. 
A morphism in this category, from
$(Y_{1},K_{1},\bP_{1})$ to $(Y_{0},K_{0},\bP_{0})$ is a cobordism of
pairs, $(W,S)$, with $W$ an oriented cobordism from $Y_{1}$ to
$Y_{0}$ and $S$ and unoriented cobordism from $K_{1}$ to $K_{0}$,
equipped with singular bundle data $\bP$ and an identification of all
this data with the given data at the two ends of the cobordism. Two
such cobordisms with singular bundle data are regarded as being the
same morphism if they are equivalent, in the same sense that we
described for $\buncob$. Just
as in the case of $\buncob$, singular instanton homology defines a
functor,
\[
                   I : \buncoblink \to \pgroup.
\]
This is essentially the content of section~\ref{subsec:maps-from-cob}.

In addition to $\buncoblink$, we would like to have a version
$\buncoblinkw$ an analogous to $\buncobw$ above, in which we replace
the bundle data $\bP$ with a codimension-2 locus $\omega$ representing
the dual of $w_{2}$. To begin again with the closed case, 
suppose that we are given a $4$-manifold $X$, an embedded surface $\Sigma$ and
a surface-with-boundary, $(\omega,\partial\omega) \subset
(X,\Sigma)$. We require that $\omega \cap \Sigma$ is a
collection of circles and points: the circles are $\partial \omega$,
along which $\omega$ should meet $\Sigma$ normally; and the points are
transverse intersections of $\omega$ and $\Sigma$ in $X$. 
From the circles $\partial\omega\subset \Sigma$
we construct a double-cover $\Sigma_{\Delta}$ by starting with a
trivialized double-cover of $\Sigma\sminus\partial\omega$, say
\[
        (\Sigma\sminus\partial\omega) \times \{1,-1\},
\]
and then
identifying across the cut with an interchange of the two sheets. Thus
$w_{1}(\Delta)$ is dual to $[\partial\omega]$ in
$H^{1}(\Sigma;\Z/2)$. 

Because $\Delta$ is trivialized, by construction, on
$\Sigma\sminus\partial\omega$, we have  distinguished copies of
$\Sigma\sminus\partial\omega$ inside the non-Hausdorff space
$X_{\Delta}$. Let 
\[
     \Sigma_{-} \subset X_{\Delta}
\]
be the closure of the copy $(\Sigma\sminus\partial\omega) \times
\{-1\}$ in $X_{\Delta}$. This is a 
surface whose boundary is two copies of $\partial\omega$. Let
$\omega_{1}$ denote the two-dimensional complex
\[
         \omega_{1} = \pi^{-1}(\omega) \cup \Sigma_{-}
\]
in $X_{\Delta}$. Figure~\ref{fig:omega-h} shows the inverse image
$\omega^{h}_{1}$ of $\omega_{1}$ in the Hausdorff space
$X^{h}_{\Delta}$, in a schematic lower-dimensional picture.
\begin{figure}
    \begin{center}
        \includegraphics[height=3 in]{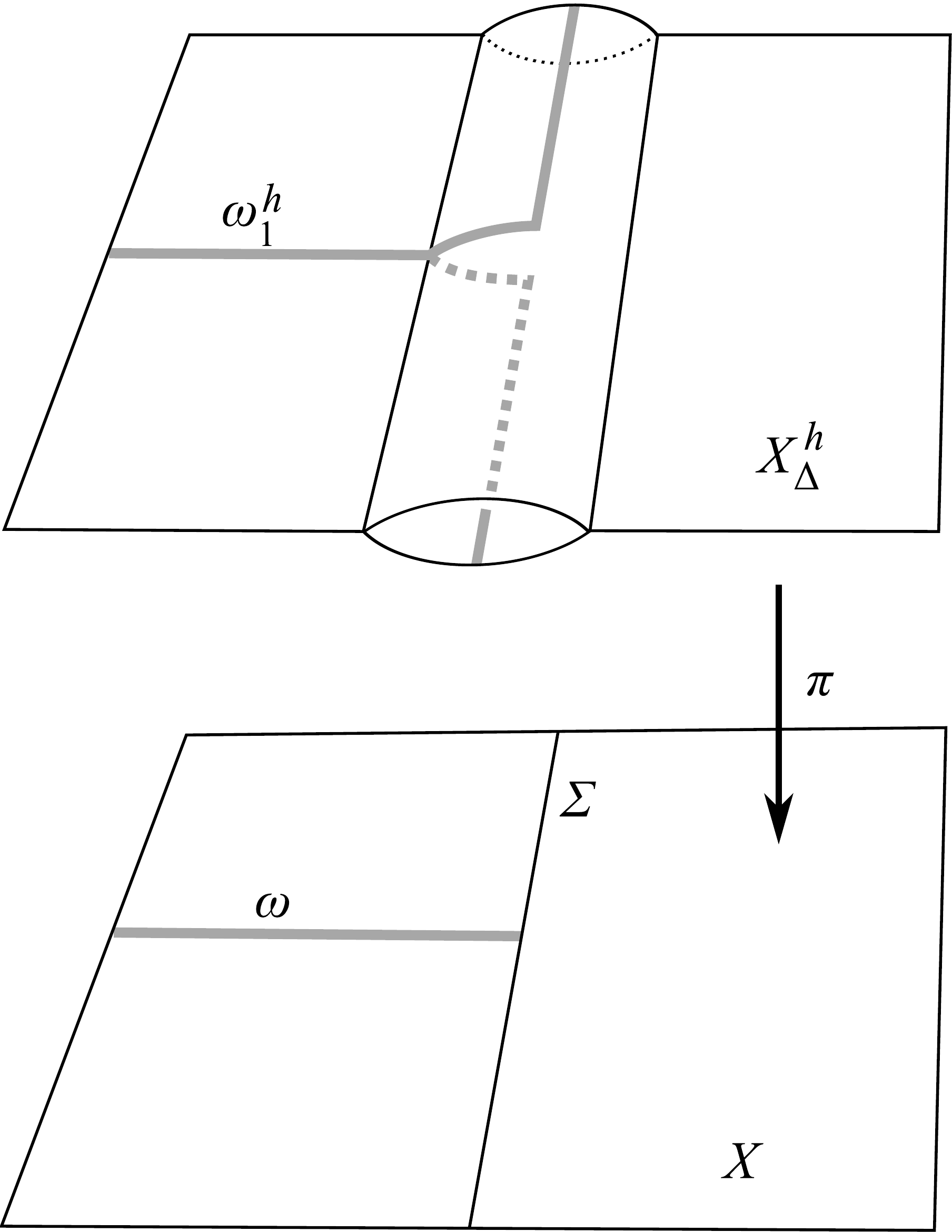}
    \end{center}
    \caption{\label{fig:omega-h}
    The codimension-2 subcomplex $\omega^{h}_{1}$, illustrated in a
    lower-dimensional picture.}
\end{figure}
A regular
neighborhood of $\omega^{h}_{1}$ in the complex $X^{h}_{\Delta}$ is a
disk bundle over $\omega^{h}_{1}$, so there is a well-defined dual
class, as in our previous discussion. There is therefore a bundle
$P_{\Delta}$ on $X_{\Delta}^{h}$, with a trivialization outside a
regular neighborhood of $\omega_{1}^{h}$, whose $w_{2}$ is the Thom
class. In this way, the original surface $\omega$ determines a bundle 
$P_{\Delta} \to X_{\Delta}$, uniquely up to the addition
of instantons and monopoles. This $P_{\Delta}$ in turn determines
singular bundle data $\bP$ on $X_{\Delta}$, up to isomorphism and the
addition of instantons and monopoles.

 We now set up the category $\buncoblinkw$ in
which objects are triples $(Y,K,\omega)$, where:
\begin{itemize}
\item $Y$ is a closed, oriented, connected $3$-manifold;
\item $K$ is an unoriented link in $Y$;
\item $\omega$ is an embedded $1$-manifold with $\omega\cap K
    = \partial\omega$, meeting $K$ normally at its endpoints.
\end{itemize}
The triples $(Y,K,\omega)$ are required to satisfy the non-integral
condition, Definition~\ref{def:non-int}.
The morphisms from $(Y_{1},K_{1},\omega_{1})$ to $(Y_{0}, K_{0},
\omega_{0})$ are isomorphism classes of triples $(W,S,\omega)$,
where
\begin{itemize}
\item $(W,S)$ is a cobordism of pairs, with $W$ an oriented
    cobordism;
\item $\omega\subset W$ is a $2$-manifold with corners, whose boundary is the
    union of $\omega_{1}$, $\omega_{0}$, and some arcs in $S$,
    along which $\omega$ normal to $S$. The intersection
    $\omega\cap S$ is also allowed to contain finitely many
    points where the intersection is transverse.
\end{itemize}
Just as with $\buncobw$, an object $(Y,K,\omega)$ in $\buncoblinkw$ gives
rise to a commutative diagram of objects $(Y,K,\bP)$ in
$\buncoblink$. Singular instanton homology therefore defines a
functor,
\[
                   I : \buncoblinkw \to \pgroup
\]
as in the previous arguments. We denote this by
$I^{\omega}(Y,K)$. Thus a morphism from $(Y_{1}, K_{1}, \omega_{1})$ to 
$(Y_{0}, K_{0}, \omega_{0})$ represented by a cobordism $(W,S,
\omega)$ gives rise to a group homomorphism (well-defined up to an
overall sign),
\[
              I^{\omega}(W,S) : I^{\omega_{1}}(Y_{1}, K_{1}) 
                        \to I^{\omega_{0}}(Y_{0}, K_{0}).
\]
Most often, the particular choice of $\omega$ is clear from the
context, and we will often write simply a generic ``$\omega$'' in
place of the specific $\omega_{1}$, $\omega_{0}$, etc. 

\subsection{Constructions for classical knots}
\label{subsec:classical}

If $K$ is a classical knot or link in $S^{3}$, or more generally a
link in a homology sphere $Y$, then the triple $(Y,K,\omega)$
satisfies the non-integral condition if and only if some component of
$K$ contains an odd number of endpoints of $\omega$. In particular, we
cannot apply the functor $I$ to such a triple when $\omega$ is empty,
so we do not directly obtain an invariant of classical knots and links
without some additional decoration.

As described in the introduction however, we can use a simple
construction to obtain an invariant
of a link with a given \emph{basepoint}. More precisely, we consider a
link $K$ in an arbitrary $3$-manifold $Y$, together with a basepoint
$x\in K$ and a given normal vector $v$ to $K$ at $x$. Given this data,
we let $L$ be a circle at the boundary of a standard disk centered at
$x$ in the tubular neighborhood of $K$, and we let $\omega$ be a
radius of the disk: a standard arc joining $x\in K$ to a point in $L$, with
tangent vector $v$ at $x$. We write $K^{\natural}$ for the new link
\[
       K^{\natural} = K \amalg L
\]
and we define
\[
             I^{\natural}(Y,K,x,v) = I^{\omega}(Y, K^{\natural}).
\]
See Figure~\ref{fig:KNatural}.
\begin{figure}
    \begin{center}
        \includegraphics[scale=0.5]{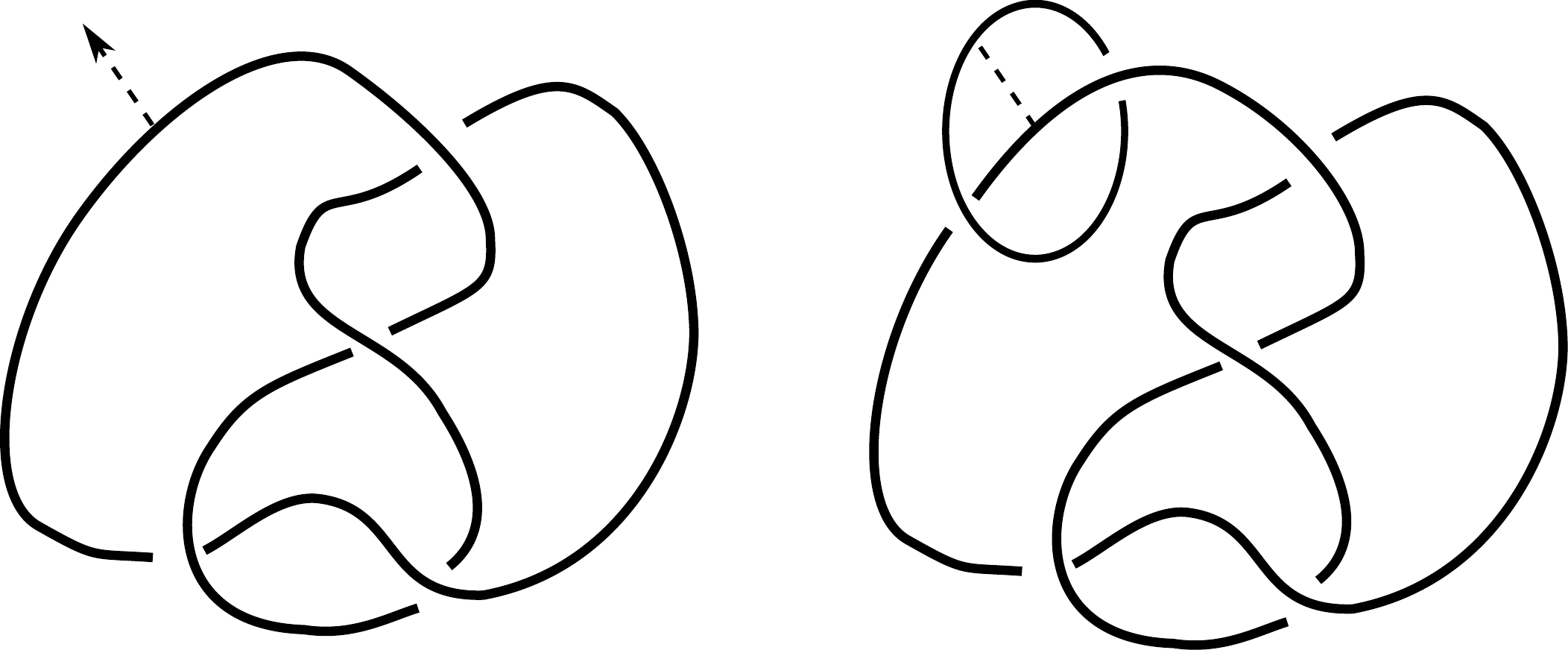}
    \end{center}
    \caption{\label{fig:KNatural}
    The link $K$ (left) and the link $K^{\natural}$ (right) 
    obtained from it by adding a meridional circle.}
\end{figure}
We shall usually omit $x$ and $v$ (particularly $v$) from our
notation, and simply write $I^{\natural}(Y,K)$ for this invariant. 

The role of the normal vector $v$ here is in making the construction
functorial. We can construct a category in which a morphism from
$(Y_{0}, K_{0}, x_{0}, v_{0})$ to $(Y_{1}, K_{1}, x_{1}, v_{1})$ is a
quadruple, $(W,S,\gamma,v)$, where $(W,S,\gamma)$ is a
cobordism of triples (so that $\gamma$ is 1-manifold with boundary
$\{x_{0}, x_{1}\}$) and $v$ is a normal vector to $S$ along
$\gamma$, coinciding with the given $v_{0}$, $v_{1}$ on the
boundary. Only $W$ is required to be oriented here, providing an
oriented cobordism between the $3$-manifolds as usual. We call this
category $\linkstar$.

The
construction that forms $K^{\natural}$ from $K$ can be applied (in a
self-evident manner) to morphisms $(W,S,\gamma,v)$ in this
category: one replaces $S$ with a new cobordism
\[
         S^{\natural} = S \amalg T
\]
where $T$ is a normal circle bundle along $\gamma$, sitting in the tubular
neighborhood of $S\subset W$; and one takes $\omega$ to be the
$I$-bundle over $\gamma$ with tangent direction $V$ along
$\gamma$. Thus we obtain a functor
\[
       \natural : \linkstar \to \buncoblinkw.
\]
Applying instanton homology gives us a projective functor,
\[
                 I^{\natural} : \linkstar \to \pgroup.
\]
Of course, when considering $I^{\natural}(Y,K)$ as a group only up to
isomorphism, we can regard it as an invariant of a link $K\subset Y$,
with a marked \emph{component}. 

We have the following basic calculation:

\begin{proposition}\label{prop:unknot-calc}
    For the unknot $U\subset S^{3}$, we have $\Inat(S^{3}, U) = \Z$.
\end{proposition}

\begin{proof}
    The link $K^{\natural}$ is a Hopf link and $\omega$ is an arc
    joining the components.
    The set of critical points for the unperturbed Chern-Simons
    functional is the set of representations (up to conjugacy) of the
    fundamental group of $S^{3}\setminus (K^{\natural} \cup \omega)$ in
    $\SU(2)$, subject the constraint that the holonomy on the links of
    $K^{\natural}$ is conjugate to $\bi$ and the holonomy on the links
    of $\omega$ is $-1$. There is one such representation up to
    conjugation, and it represents a non-degenerate critical point. So
    the complex that computes $\Inat(S^{3},U)$ has just a single generator.
\end{proof}

To obtain an invariant of a link $K\subset Y$ without need of a
basepoint or marked component, we can always replace $K$ by $K\amalg
U$, where $U$ is a new unknotted circle contained in a ball disjoint
from $K$. We put the basepoint $x$ on the new component $U$, and we
define
\[
        K^{\sharp} = (K\amalg U)^{\natural}.
\]
To say this more directly, $K^{\sharp}$ is the disjoint union of $K$
and a Hopf link $\Hopf$ contained in a ball disjoint from $K$, as
illustrated in Figure~\ref{fig:KSharp}.
\begin{figure}
    \begin{center}
        \includegraphics[scale=0.5]{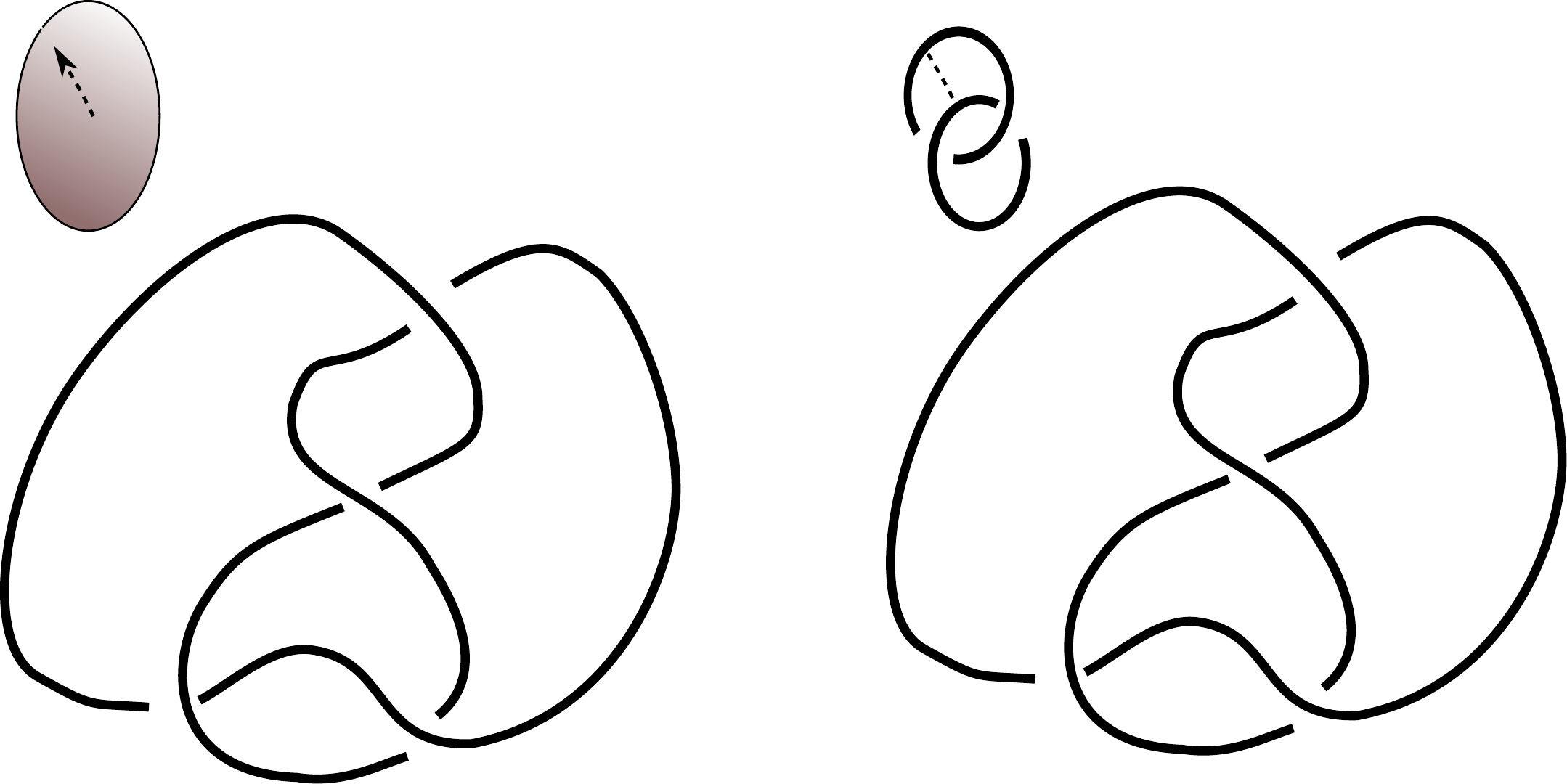}
    \end{center}
    \caption{\label{fig:KSharp}
    The link $K$ (left) and the link $K^{\sharp}$ obtained 
    as the union of $K$ with a Hopf link located
    at a marked basepoint.}
\end{figure}
 It comes
with an $\omega$ which is a standard arc joining the two components of
$\Hopf$. Thus we can define an invariant of links $K$ without basepoint by
defining
\[
\begin{aligned}
    I^{\sharp}(Y,K) &= I^{\omega}(Y,K^{\sharp})\\
            &=  I^{\omega}(Y, K \amalg \Hopf) 
\end{aligned}
\]

To make the construction $\sharp$ functorial, we need to be given $(Y,K)$
together with:
\begin{itemize}
\item a basepoint $y\in Y$ disjoint from $K$;
\item a preferred $2$-plane $p$ in $T_{y}Y$;
\item a vector $v$ in $p$.
\end{itemize}
We can then form $\Hopf$ by taking the first circle $U$ to be a standard
small circle in $p$, and then applying the $\natural$ construction to
$U$, using $v$ to define the framed basepoint on $U$. Thus
$I^{\sharp}$ is a projective functor from a category whose objects are
such marked pairs $(Y,K, y, p, v)$. The appropriate definition of the
morphisms in this category can be modeled on the example of
$\linkstar$ above.

\medskip
The constructions $\Inat$ and $\Isharp$ are the ``reduced'' and
``unreduced'' variants of singular instanton Floer homology. The
$\Isharp$ variant can be applied to the empty link, and the following
proposition is a restatement of Proposition~\ref{prop:unknot-calc}.

\begin{proposition}\label{prop:isharpemptyset}
    For the empty link $\emptyset$ in $S^{3}$, we have $\Isharp(S^{3},
    \emptyset) = \Z$. \qed
\end{proposition}

When dealing with $\Isharp$ for classical knots and links $K$, we will
regard $K$ as lying in $\R^{3}$ and take the base-point at infinity in
$S^{3}$. We will simply write $\Isharp(K)$ in this context.

\subsection{Resolving the ambiguities in overall sign}
\label{subsec:i-orient}

Thus far, we have been content with having an overall sign ambiguity
in the homomorphisms on Floer homology groups which arise from
cobordisms. We now turn to consider what is involved in resolving
these ambiguities.

We begin with the case of $\buncoblink$, in which a typical morphism
from $(Y_{1}, K_{1}, \bP_{1})$ to $(Y_{0}, K_{0}, \bP_{0})$ is
represented by a cobordism with singular bundle data, $(W,S,\bP)$. We
already have a rather tautological way to deal with the sign issue in
this case: we need to enrich our category by including all the data we
used to make the sign explicit. Thus, we can define a category
$\widetilde\buncoblink$ in which an object is a tuple \[(Y_{0},
K_{0},\bP_{0},\aux_{0}),\] where $\aux_{0}$ is the auxiliary data
consisting of a choice of Riemannian metric $\orbig_{0}$, a choice of
perturbation $\pert_{0}$, and a choice of basepoint $\theta_{0}$. (See
equation~\eqref{eq:aux-data} above.) A morphism in
$\widetilde\buncoblink$, from $(Y_{1}, K_{1},\bP_{1},\aux_{1})$ to
$(Y_{0}, K_{0},\bP_{0},\aux_{0})$, consists of the previous data
$(W,S,\bP)$ together with a choice of an $I$-orientation for
$(W,S,\bP)$ (Definition~\ref{def:I-orientation}). 
With such a definition, we have a functor (rather than a
projective functor)
\[
            I:   \widetilde\buncoblink \to \group.
\]

We can make something a little more concrete out of this for the
functor 
\[
     \Inat : \linkstar \to \pgroup.
\]
We would like to construct a category $\widetilde\linkstar$ and a
functor
\[
      \Inat :\widetilde\linkstar \to \group.
\]
To do this, we first define the objects of $\widetilde\linkstar$ to be
quadruples $(Y,K,x,v)$, where $K$ is now an \emph{oriented} link in
$Y$ and $x$ and $v$ are a basepoint and normal vector to $K$ as
before. (The orientation is the only additional ingredient here.)
Given this data, we can orient the link $K^{\natural}$ by orienting
the new component $L$ so that it has linking number $1$ with $K$ in
the standard ball around the basepoint $x$. There is a standard
cobordism of oriented pairs, $(Z,F)$, from $(Y,K)$ to
$(Y,K^{\natural})$: the $4$-manifold $Z$ is a product $[0,1]\times Y$
and the surface $F$ is obtained from the product surface $[0,1]\times
K$ by the addition of a standard embedded $1$-handle. Alternatively,
$(Z,F)$, is a boundary-connect-sum of pairs, with the first summand
being $[0,1]\times (Y,K)$ and the second summand being $(B^{4}, A)$,
where $A$ is a standard oriented annulus in $B^{4}$ bounding the
oriented Hopf link in $S^{3}$. The arc $\omega$ in $Y$ joining $K$ to
$L$ in the direction of $v$ is part of the boundary of a disk $\omega_Z$ in
$Z$ whose boundary consists of $\omega\subset Y$ together with a standard arc
lying on $F$. 

Although $(Y,K,\emptyset)$ may not satisfy the non-integral condition,
we can nevertheless form the space of connections $\bonf(Y,K) =
\bonf^{\emptyset}(Y,K)$ as before: this is a space of $\SU(2)$
connections on the link complement, with holonomy asymptotic to the
conjugacy class of the element
\begin{equation}\label{eq:bi}
         \bi = \begin{pmatrix} i & 0 \\ 0 & -i \end{pmatrix}.
\end{equation}
We can now exploit that fact that there is (to within some inessential
choices) a preferred basepoint $\theta$ in $\bonf(Y,K)$ arising from a
reducible connection, as in section 3.6 of
\cite{KM-knot-singular}. Thus, the singular connection $\theta$ is
obtained from the trivial product connection in $\SU(2)\times Y$ by
adding a standard singular term
\[
                \beta(r) \frac{\bi}{4} \eta
\]
where $\beta$ is a cut-off function on a tubular neighborhood of $K$
and $\eta$ is (as before) a global angular $1$-form, constructed this
time using the given orientation of $K$. If $\beta$ is a critical
point for the perturbed functional in
$\bonf^{\omega}(Y,K^{\natural})$, we let $A_{\theta,\beta}$ be any
chosen connection in $\bonf^{\omega_Z}(Z,F; \theta,\beta)$ (i.e. a
connection on the cobordism, asymptotic to $\theta$ and $\beta$ on the
two ends). We then \emph{define}
\begin{equation}
    \label{eq:Lambda-nat}
     \Lambda^{\natural}(\beta)
\end{equation}
to be the two-element set of orientations for the determinant line
$\det(\calD_{A_{\theta,\beta}})$.

To summarize, we have defined $\Lambda^{\natural}(\beta)$ much as we
defined $\Lambda(\beta)$ earlier in \eqref{eq:Lambda-beta}. The
differences are that we are now using the non-trivial cobordism
$(Z,F,\omega_Z)$ rather than the product, and we are exploiting the
presence of a distinguished reducible connection on the other end of
this cobordism, defined using the given orientation of $K$. We can
define a chain complex
\[
            C^{\natural}(Y,K) = \bigoplus_{\beta} \Z\Lambda^{\natural} (\beta),
\]
and we can regard $\Inat(Y,K)$ as being defined by the homology of
this chain complex.

Consider next a morphism in $\linkstar$, say $(W,S,\gamma,v)$ from
$(Y_{1}, K_{1},x_{1}, v_{1})$ to $(Y_{0},K_{0},x_{0},v_{0})$. We shall
choose orientations for $K_{1}$ and $K_{0}$, as we did in the previous
paragraphs. We do not assume that the surface $S$ is an oriented
cobordism, but we do require that it looks like one in a neighborhood
of the path $\gamma$: that is, we assume that if $u_{i}$ is an
oriented tangent vector to $K_{i}$ at $x_{i}$, then there is an
oriented tangent vector to $S$ along $\gamma$, normal to $\gamma$,
which restricts to $u_{1}$ and $u_{0}$ at the two ends.

We have
an associated morphism $(W, S^{\natural}, \omega)$ between $(Y_{1},
K^{\natural}_{1},\omega_{1})$ and $(Y_{0},K^{\natural}_{0},
\omega_{0})$. Let 
\[
\begin{aligned}
    \beta_{1} & \in \bonf^{\omega_{1}}(Y_{1}, K_{1}) \\
    \beta_{0} & \in \bonf^{\omega_{0}}(Y_{0}, K_{0})
\end{aligned}
\]
be critical points, let $[A_{\beta_{1},\beta_{0}}]$ be a connection in
$\bonf^{\omega}(W,S^{\natural};\beta_{1},\beta_{0})$, and consider the
problem of orienting the determinant line \[\det(\calD_{A_{\beta_{1},\beta_{0}}}).\]
Let $(Z_{1}, F_{1})$ and $(Z_{0}, F_{0})$ be the standard cobordisms,
described above,
\[
\begin{aligned}
    (Z_{i}, F_{i}) : (Y_{i}, K_{i}) \to (Y_{i}, K^{\natural}_{i}).
\end{aligned}
\]
There is an evident diffeomorphism, between two different composite
cobordisms, from $(Y_{1}, K_{1})$ to $(Y_{0}, K_{0}^{\natural})$:
\[
           (Z_{1}, F_{1}, \omega_{Z_{1}}) \cup_{Y_{1}} (W,
           S^{\natural}, \omega) = (W, S, \emptyset) \cup_{Y_{0}}
           (Z_{0}, F_{0}, Y_{0}).
\]
From this we obtain an isomorphism of determinant lines,
\begin{equation}\label{eq:det-iso}
             \det(\calD_{A_{\theta_{1},\beta_{1}}}) \otimes
                   \det(\calD_{A_{\beta_{1},\beta_{0}}}) = 
                   \det(\calD_{A_{\theta_{1},\theta_{0}}}) \otimes \det(\calD_{A_{\theta_{0},\beta_{0}}}).
\end{equation}
Here $\theta_{i}$ are the preferred reducible connections in
$\bonf(Y_{i}, K_{i})$ as above, and $A_{\theta_{1},\theta_{0}}$ is a
connection joining them across the cobordism $(W,S)$. 

If we wish the cobordism $(W,S^{\natural},\omega)$ to give rise to a
chain map 
\[
       C^{\natural}(Y_{1},K_{1}) \to C^{\natural}(Y_{0}, K_{0})
\]
with a well-defined overall sign,
then we need to specify an isomorphism
\[
                    \det(\calD_{A_{\beta_{1},\beta_{0}}}) 
                       \to \mathrm{Hom}
                       \bigl(\Z\Lambda^{\natural}(\beta_{1}),
                       (\Z\Lambda^{\natural}(\beta_{0}) 
            \bigr);
\]
and by the definition of $\Lambda^{\natural}$ and the isomorphism
\eqref{eq:det-iso}, this means that we must orient
\[
          \det(\calD_{A_{\theta_{1},\theta_{0}}}).
\]
Thus we are led to the following definition:

\begin{definition}
    Let  $(Y_{1}, K_{1})$ and $(Y_{0}, K_{0})$ be two pairs of
    oriented links in closed oriented $3$-manifolds, and let
    et $(W,S)$ be a cobordism of pairs, from $(Y_{1}, K_{1})$ to
    $(Y_{0}, K_{0})$, with $W$ an oriented cobordism, and $S$ an
    oriented cobordism (and possible non-orientable). Then an
    \emph{$I^{\natural}$-orientation} for $(W,S)$ will mean an
    orientation for the determinant line
    $\det(\calD_{A_{\theta_{1},\theta_{0}}})$, where $\theta_{i}$ are
    the reducible singular $\SU(2)$ connections on $Y_{i}$, described
    above and determined by the given orientations of the $K_{i}$. 
  
    In the special case that $W$ is a product $[0,1]\times Y$, an
    $\Inat$-orientation for an embedded cobordism $S$ between oriented
    links will mean an $\Inat$-orientation for the pair $([0,1]\times
    Y, S)$.
    \CloseDef
\end{definition}

We are now in a position to define $\widetilde\linkstar$. Its 
objects are quadruples $(Y_{i}, K_{i}, x_{i}, v_{i})$
with $K_{i}$ an \emph{oriented} link, and its morphisms are quintuples
$(W,S,\gamma,v,\lambda)$, where
\begin{itemize}
\item $(W,S)$ is a cobordism of pairs, with $W$ and oriented
    cobordism;
\item $\gamma$ is a path from $x_{1}$ to $x_{0}$, with the property
    that $S$ is an oriented cobordism along $\gamma$;
\item $v$ is a normal vector to $S$ along $\gamma$, restricting to
    $v_{1}$, $v_{0}$ at the two ends; and
\item $\lambda$ is a choice of $\Inat$-orientation for $(W,S)$.
\end{itemize}
With this definition, we have a well-defined functor to groups.

\begin{remarks}
Our definition of $\Inat$-orientation still rests on an analytic
index, so some comments are in order. First of all, the definition
makes it apparent that there is a natural composition law for
$\Inat$-orientations of composite cobordisms. Second, if $S$ is
actually an \emph{oriented} cobordism from $K_{1}$ to $K_{0}$, then an
$\Inat$-orientation of $(W,S)$ becomes equivalent to a
homology-orientation of the cobordism, as discussed in \cite{KM-book}
and \cite{KM-knot-singular}. Indeed, the case that $S$ is oriented is
precisely the case considered in \cite{KM-knot-singular}, where
homology-orientations of the cobordisms are shown to fix the signs of
the corresponding chain-maps. In particular, if $W$ is a product
$[0,1]\times Y$, then an oriented cobordism $S$ between oriented links
in $Y$ has a canonical $\Inat$-orientation, and these canonical
$\Inat$-orientations are preserved under composition.
\end{remarks}

When using the unreduced functor $\Isharp$ for knots in $\R^{3}$, we
have adopted the convention of putting the extra Hopf link ``at
infinity'' in a standard position. With this setup, we have a category
\begin{equation}\label{eq:link-R3-cat}
            \widetilde\link(\R^{3})
\end{equation}
whose objects are oriented links in $\R^{3}$ and whose morphisms are
$\Isharp$-oriented cobordisms in $[0,1]\times \R^{3}$.

\subsection{\texorpdfstring{Absolute $\Z/4$ gradings}{Absolute Z/4 gradings}}

For a general $(Y,K,\bP)$ and its corresponding configuration space
$\bonf(Y,K,\bP)$, 
the path-dependent relative grading, $\gr_{z}(\beta_{1},\beta_{0}) \in
\Z$, descends to a path-independent relative grading,
\[\bar\gr(\beta_{1},\beta_{0}) \in \Z/4.\] 
(This is because any two paths differ by the addition of instantons
and monopoles, both of which contribute multiples of $4$ to the
relative grading.)  As a consequence, both $\Inat(K)$ and $\Isharp(K)$
are homology theories with an affine $\Z/4$ grading. In the case of
$\Inat$ however (and hence also for $\Isharp$), we can define an
\emph{absolute} $\Z/4$ grading in a fairly straightforward
manner. Such an absolute $\Z/4$ grading should assign to each $\beta
\in \crit_{\pert}(Y, K^{\natural})$ an element
\[
           \bar\gr(\beta) \in \Z/4
\]
(depending $\pert$ and the metric, as well as on $\beta$).

\begin{proposition}\label{prop:canonical-Z/4}
    There is a unique way to define an absolute $\Z/4$ grading,
    $\bar\gr(\beta) \in \Z/4$ for $\beta\in \crit_{\pert}(Y,
    K^{\natural})$  such that the following two conditions
    hold:
    \begin{enumerate}
    \item the grading is normalized by having $\bar\gr(\beta_{0}) = 0$
        for the unique critical point in the case of the unknot in
        $S^{3}$ with $\pert=0$;
    \item if $(W,S, \gamma, v, \lambda)$ is a morphism from $(Y_{1},
        K_{1}, x_{1},v_{1})$ to $(Y_{0}, K_{0},x_{0},v_{0})$ in the
        category $\linkstar$, and $\beta_{1}$, $\beta_{0}$ are
        corresponding critical points, then 
          \[
            \gr_{z}(W,S^{\natural},\beta_{1},\beta_{0})
                        = \bar\gr(\beta_{1}) - \bar\gr(\beta_{0})
                           + \iota(W,S) \pmod{4}
           \]
          where
          \begin{multline*}
              \iota(W,S) = 
         - \chi(\Tigma) +
        b_{0} (\partial^{+} \Tigma) - b_{0}(\partial^{-}\Tigma)  \\ -
        \frac{3}{2}(\chi(W) + \sigma(W))    + \frac{1}{2}(b^{1}(Y_{0})
        - b^{1}(Y_{1})).
            \end{multline*}
    \end{enumerate}
    Similarly, for $\Isharp$ there is a canonical $\Z/4$ grading such
    that the generator for $\Isharp(\emptyset)$ is in degree $0$.
\end{proposition}

\begin{proof}
    The uniqueness is clear. For the question of existence, we return
    first to a closed
   pair $(X,\Sigma)$, 
   with $w_{2}(P)=0$ so that $\Delta$ is trivial. The dimension
   formula \eqref{eq:old-k-l-dim} 
   tells us in this case that the index of the linearized operator
    $\calD$ satisfies
    \[
               \ind \calD =  4l -
        \frac{3}{2}(\chi(X) + \sigma(X)) + \chi(\Sigma) \pmod{4}.
     \]
     The instanton number $k$ is an integer, so the term $8k$ can be
     omitted; but the monopole number $l$ is potentially a
     half-integer and $2l$ is congruent to $\chi(\Sigma)$ mod 2. So
     the formula can be rewritten as
     \[
               \ind \calD =  -
        \frac{3}{2}(\chi(X) + \sigma(X)) - \chi(\Sigma) \pmod{4}.
     \]

    Given this formula, the existence of $\bar\gr$ now follows from
    the additivity of the terms involved.
\end{proof}

\section{Two applications of Floer's excision theorem}
\label{sec:an-excis-argum}

The proof of Proposition~\ref{prop:singular-to-sutured} rests on
Floer's excision theorem (slightly adapted to our situation). The
statement of the excision theorem generally involves $3$-manifolds $Y$
that may have more than one component, or cobordisms with more than
two boundary components. Disconnected $3$-manifolds do
not create any difficulties when defining instanton homology (see
below for a brief review); but they do introduce a new problem when we
look at cobordisms and functoriality: this problem stems from the fact
that the stabilizer of an irreducible connection on a disconnected
$3$-manifold is no longer just $\pm 1$, but is $(\pm 1)^{n}$, where
$n$ is the number of components; and not all of these elements of the
stabilizer will necessarily extend to locally constant gauge
transformations on the cobordism. This results in extra factors of two
when gluing. The way to resolve these problems is to carefully enlarge
the gauge group, by allowing some automorphisms of the $\SO(3)$ bundle
that do not lift to determinant-1 gauge transformations. Our first
task in this section is to outline how this is done. The issue
appears (and is dealt with) already in Floer's original proof of
excision (as presented in \cite{Braam-Donaldson}), but we need a more
general framework.

When enlarging the gauge group in this way, the standard approach to
orienting moduli spaces breaks down, and an alternative method is
needed. We turn to this first.

\subsection{Orientations and almost-complex structures}
\label{subsec:almost-complex}

Given an $\SO(3)$ bundle $P$ on a closed, oriented, Riemannian
$4$-manifold $X$, we have considered the moduli space of ASD
connections, $M(X,P)$, by which we mean the quotient of the space of
ASD connections by the \emph{determinant-1} gauge group
$\G(P)$. This moduli space, when regular, is orientable; and orienting
it amounts to trivializing the determinant line $\calD_{A}\to
\bonf(X,P)$. Two approaches to choosing a trivialization are available
and described in \cite{Donaldson-orientations}. The first relies on
having a $U(2)$ lift of $P$ (or equivalently, an integral lift $v$ of
$w_{2}(P)$) together with a homology orientation of $X$: this is the
standard approach generally used in defining Donaldson's polynomial
invariants. The second approach described in
\cite{Donaldson-orientations} relies on having an almost complex
structure $J$ on $X$: in the presence of $J$, there is a standard
homotopy from the operators $\calD_{A}$ to the complex-linear operator
\[
\bar\partial_{A} + \bar\partial^{*}_{A} :
\Omega^{0,1}_{X}\otimes_{\R}\g_{P} \to (\Omega^{0,0}_{X}\oplus
\Omega^{0,2}_{X})\otimes_{\R}\g_{P}.
\]
The complex orientation of the operator at the end of the homotopy
provides a preferred orientation for the determinant line.

The second of these approaches has the disadvantage of requiring the
existence of $J$. On the other hand, when $J$ exists, the argument
provides a simple and direct proof of the orientability of the
determinant line, because the homotopy can be applied to the entire
family of operators over $\bonf(X,P)$. More importantly for us, this
second approach to orientations can be used to establish the
orientability of the determinant line over the quotient of
$\conf(X,P)$ by the \emph{full} automorphism group $\Aut(P)$, not just
the group $\G(P)$ of determinant-1 gauge transformations.

To set up this approach to orientations, we consider a pair $(Y,K)$ as
usual with singular bundle date $\bP$. Recall that the center $Z$ of
$\G(\bP)$ is $\{\pm 1\}$ and that we
have the isomorphism
\[
       \Aut(\bP)/ (\G(\bP)/Z) \cong H^{1}(Y ; \Z/2).
\]
Thus $H^{1}(Y ; \Z/2)$ acts on $\bonf = \bonf(Y,K,\bP)$. Fix a
subgroup $\phi \subset H^{1}(Y;\Z/2)$, and consider the quotient
\[
      \bar\bonf =   \bonf(Y,K,\bP)/\phi.
\]
We shall require that the non-integral condition hold as usual, so
that all critical points of the Chern-Simons functional are
irreducible; but we also want $\phi$ to act freely on the set of
critical points. This can be achieved (as the reader may verify) by
strengthening the non-integral condition as follows. 

\begin{definition}\label{def:ph-nonint}
We say that
$(Y,K,\bP)$ satisfies the \emph{$\phi$-non-integral} condition if
there is a non-integral surface $\Sigma\subset Y$ (in the sense of
Definition~\ref{def:non-int}) satisfying the additional constraint
that $\phi|_{\Sigma}$ is zero.\CloseDef
\end{definition}

When such a condition holds, there is no difficulty in choosing
$\phi$-invariant holonomy perturbations so that all critical points
are non-degenerate and all moduli spaces of trajectories are
regular. We will write $\bar\alpha$, $\bar\beta$ for typical critical
points in the quotient $\bar\bonf$, and $M(\bar\alpha,\bar\beta)$ for
the moduli spaces of trajectories. To show these moduli spaces are
orientable (and to orient eventually orient them), we start on the
closed manifold $S^{1}\times Y$ which we equip with an
$S^{1}$-invariant orbifold-complex structure: in a neighborhood of
$S^{1}\times K$, the model is the $\Z/2$ quotient of a complex disk
bundle over the complex manifold $S^{1}\times K$. Up to homotopy, such
an almost complex structure is determined by giving a non-vanishing
vector field on $Y$ that is tangent to $K$ along $K$. As described
above, the family of operators $\calD_{A}$ on the orbifold
$S^{1}\times \orbiY$ is homotopic to a family of complex operators. As
in the standard determinant-1 story, by
excision, we deduce that the moduli spaces
$M(\bar\alpha,\bar\beta)$ on $\R\times (Y,K)$ are orientable. We also
see that if $\bar z$ and $\bar z'$ are two homotopy classes of paths
from $\bar\alpha$ to $\bar\beta$, then an orientation for $M_{\bar
  z}(\bar\alpha,\bar\beta)$ determines an orientation for  $M_{\bar
  z'}(\bar\alpha,\bar\beta)$. This allows us to define
$\bar\Lambda(\bar\alpha,\bar\beta)$ as the two-element set that
orients all of these moduli spaces simultaneously. For any given
$\bar\beta$, we can also consider on $\R\times Y$ the $4$-dimensional
operator that interpolates between $\calD_{\bar\beta}$ at the
$+\infty$ end and its complex version $\bar\partial_{\bar\beta} +
\bar\partial^{*}_{\bar\beta}$ at the $-\infty$ end. If we define
$\bar\Lambda(\bar\beta)$ to be the set of orientations of the
determinant of this operator, then  we have isomorphisms
\[
            \bar\Lambda(\bar\alpha,\bar\beta) = 
                         \bar\Lambda(\bar\alpha)\bar\Lambda(\bar\beta).
\]
We are therefore able to define the Floer complex in this situation by
the usual recipe: we write it as
\[
        C(Y,K,\bP)^{\phi} = \bigoplus_{\bar\beta} \bar\Lambda(\bar\beta)
\]
and its homology as $I(Y,K,\bP)^{\phi}$.

The construction of $I(Y,K,\bP)^{\phi}$ in this manner depends on the
choice of almost-complex structure $J$ on $S^{1}\times Y$. If $J$ and
$J'$ are two $S^{1}$-invariant complex structures, then the class
\[
     (c_{1}(J) - c_{1}(J'))/2
\]
determines a character $\xi: \phi \to \{\pm 1\}$, and hence a local system
$\Z_{\xi}$ with fiber $\Z$ on $\bar\bonf = \bonf/\phi$. The
corresponding Floer homology groups $I$ and $I'$ defined using the
orientations arising from $J$ and $J'$ are related by
\[
              I'(Y,K,\bP)^{\phi} =  I(Y,K,\bP;\Z_{\xi})^{\phi},
\] 
as can be deduced from the calculations in
\cite{Donaldson-orientations}. 

In the special case that $\phi$ is
trivial, the group $I(Y,K,\bP)^{\phi}$ coincides with $I(Y,K,\bP)$ as
previously defined: a choice of isomorphism between the two depends
on a choice of trivialization of $J$-dependent $2$-element set 
$\bar\Lambda(\theta)$, where $\theta$ is the chosen
base-point. Another special case is the following:

\begin{proposition}\label{prop:simple-doubling}
    Suppose that $\phi$ is a group of order $2$ in $H^{1}(Y;\Z/2)$ and
    that we are in one of the following two cases:
    \begin{enumerate}
    \item the link $K$ is empty and the non-zero element of $\phi$ has
        non-trivial pairing with $w_{2}(P)$; or
    \item $K$ is non-empty and its fundamental class has non-zero
        pairing with $\phi$.
    \end{enumerate}
     Then we have
      \[
             I(Y,K,\bP) = I(Y,K,\bP)^{\phi} \oplus I(Y,K,\bP)^{\phi}.
       \]
\end{proposition}

\begin{proof}
    In the first case, the complex $C(Y,K,\bP)$ (or simply $C(Y,P)$)
    has a relative $\Z/8$ grading, and the action of the non-trivial
    element of $\phi$ on the set of critical points
    shifts the grading by $4$, so that $I(Y,P)^{\phi}$ is $\Z/4$
    graded. 
    In the second case, there is a relative
    $\Z/4$ grading on $C(Y,K,\bP)$, 
    and the action shifts the grading by $2$, so that
    $I(Y,K,\bP)^{\phi}$ has only a $\Z/2$ grading. In either case, the
    group $I(Y,K,\bP)$ is obtained from $I(Y,K,\bP)^{\phi}$ by simply
    unwrapping the grading, doubling the period from 4 to 8, or from
    2 to 4 respectively.
\end{proof}

As a simple example, we have the following result for the $3$-torus:

\begin{lemma}\label{lem:T3-phi-calc}
    In the case that $Y$ is a $3$-torus and $K$ is empty, let $P\to
    T^{3}$ be an $\SO(3)$ bundle and $\phi$ any two-element subgroup
    of $H^{1}(T^{3} ; \Z/2)$. Suppose that $w_{2}(P)$ pairs
    non-trivially with the non-zero element of $\phi$. Then
    $I(T^{3},P)^{\phi}=\Z$.
\end{lemma}

\begin{proof}
    This is an instance of the first item in the previous lemma, given
    that $I(T^{3},P) = \Z\oplus\Z$. Alternatively, it can be seen
    directly, as the representation variety in $\bonf(T^{3},P)/\phi$
    consists of a single point.
\end{proof}

We can similarly enlarge the gauge group on 4-dimensional cobordisms,
so as to make the groups $I(Y,K,\bP)^{\phi}$ functorial. Thus, suppose
we have a cobordism $(W,S,\bP)$ from $(Y_{1},K_{1},\bP_{1})$ to
$(Y_{0},K_{0},\bP_{0})$. Let $\phi \subset H^{1}(W;\Z/2)$ be chosen
subgroup, and let $\phi_{i}$ be its image in $H^{1}(Y_{i};\Z/2)$ under
the restriction map. Suppose that the singular bundle data $\bP_{i}$
satisfies the $\phi_{i}$-non-integral condition for $i=0,1$. After
choosing metrics and perturbations, the group
$\phi$ acts on the usual moduli spaces on
$M(W,S,\bP;\beta_{1},\beta_{0})$, and we have quotient moduli spaces
$\bar{M}(W,S,\bP;\bar\beta_{1},\bar\beta_{0})$. To set up
orientations, we need to choose an almost-complex structure $J$ on $(W,S)$:
i.e, an almost-complex structure on $W$ such that $S$ is an
almost-complex submanifold. This
is always possible when $W$ has boundary, \emph{as long as $S$ is
  orientable}. (The orientability of $S$ is a restriction on the
applicability of this framework.) We denote by $J_{i}$ the
translation-invariant complex structure on $\R\times Y_{i}$ which
arises from restricting $J$ to the ends, and we use $J_{i}$ in
constructing the Floer groups $I(Y_{i},K_{i},\bP_{i})^{\phi_{i}}$ as
above. Then $J$ on $(W,S)$ orients the moduli spaces appropriately and we
have a well-defined chain map
\[
                      C(Y_{1},K_{1},\bP_{1})^{\phi_{1}}
           \to  C(Y_{0},K_{0},\bP_{0})^{\phi_{0}},
\] 
and a map on homology,
\[
          I(W,S,\bP)^{\phi}: 
                      I(Y_{1},K_{1},\bP_{1})^{\phi_{1}}
           \to  I(Y_{0},K_{0},\bP_{0})^{\phi_{0}}.
\]
In this way we have a functor (not just a projective functor) from the
category whose morphisms are cobordisms $(W,S,\bP)$ equipped with a
subgroup $\phi$ of $H^{1}(W;\Z/2)$ and an almost-complex structure
$J$.

The use of almost-complex structures also provides a canonical mod 2
grading on the Floer groups $I(Y,K,\bP)^{\phi}$. In the chain complex,
a generator corresponding to a critical points $\bar\beta$ is in even
or odd degree according to the parity index of the operator whose
determinant line is $\bar\Lambda(\bar\beta)$, as defined above.

\subsection{Disconnected 3-manifolds}

We now extend the above discussion to the case of $3$-manifolds that
are not necessarily connected: we begin with a possibly disconnected
closed, oriented $3$-manifold $Y$, containing a link $K$ with singular
bundle data $\bP$.  The configuration space $\bonf(Y,K,\bP)$ is the
product of the configuration spaces of the components. This is acted
on by $H^{1}(Y;\Z/2)$, and we suppose that we are given a subgroup
$\phi$ of this finite group with which to form the quotient space
$\bar\bonf(Y,K,\bP)$.  We require that the $\phi$-non-integral
condition hold on each component of $Y$. We choose a
translation-invariant almost-complex structure $J$ on $(\R\times Y,
\R\times K)$ in order to determine orientations of moduli spaces, and
after choosing metrics and perturbations we arrive at Floer homology
groups
\[
      I(Y,K,\bP)^{\phi}
\]
with no essential changes needed to accommodate the extra
generality. We do not exclude the case that $Y$ is empty, in which
case its Floer homology is $\Z$.

Since the complex that computes $I(Y,K,\bP)^{\phi}$ is essentially a
product, there is a K\"unneth-type theorem that describes this
homology. 
To illustrate in the case of two
components, if $Y= Y^{1} \amalg Y^{2}$, and $\phi=\phi^{1}\times
\phi^{2}$, then there is an isomorphism of  chain complexes
\[
         C(Y^{1},K^{1},\bP^{1})^{\phi^{1}} \otimes   C(Y^{2},K^{2},\bP^{2})^{\phi^{2}}
            \to   C(Y,K,\bP)^{\phi}
\]
under which the Floer differential $d$ on the right becomes the
differential on the tensor product given by
\[
         d(a \otimes b) = d^{1}a \otimes b + \epsilon^{1}a \otimes  d^{2} b,
\]
where $\epsilon^{1}$ is the sign operator on $C(Y^{1}, K^{1},\bP^{1})^{\phi^{1}}$ 
that is $-1$ on generators which
have odd degree and $+1$ on
generators which have even degree.  As a result of this isomorphism on
chain complexes, there is  a short exact sequence
relating the homology groups $I = I(Y,K,\bP)^{\phi}$ to the groups $I_{i} =
I(Y_{i},K_{i},\bP_{i})^{\phi^{i}}$:
\begin{equation}\label{eq:Kunneth}
         \mathrm{Tor}( I^{1},
               I^{2} ) \hookrightarrow I
              \twoheadrightarrow I^{1}\otimes
                I^{2}.
\end{equation}
Note that
we also have an isomorphism of complexes
\[
         C(Y^{2},K^{2},\bP^{2})^{\phi^{2}} \otimes   C(Y^{1},K^{1},\bP^{1})^{\phi^{1}}
            \to   C(Y,K,\bP)^{\phi}
\]
where the differential on the left is $d^{2}\otimes 1 + \epsilon^{2}
\otimes d^{1}$. These two isomorphisms are intertwined by the map
\[
     a \otimes b \mapsto \epsilon^{2} b \otimes \epsilon^{1} a.
\]
The change in sign results from the need to identify the determinant line
of the direct sum of two operators with the tensor product of the
determinant lines (see \cite{KM-sutures} for a discussion).

Now consider again a cobordism $(W,S,\bP)$ between
(possibly disconnected) $3$-manifolds with singular bundle data,
$(Y_{1}, K_{1},\bP_{1})$ and $(Y_{0}, K_{0}, \bP_{0})$. We do not
require that $W$ is connected; but to avoid reducibles, we do suppose
that $W$ has no closed components. Let $\phi$ be
a subgroup of $H^{1}(W;\Z/2)$ and let $\phi_{i}$ be its restriction to
$Y_{i}$. Let $J$ be a complex structure on $(W,S)$ and $J_{i}$ its
restriction to the two ends, $i=1,0$. With metrics and perturbations
in place, we obtain a chain map and induced map of Floer homology
groups,
\[
 I(W,S,\bP)^{\phi} :    
I(Y_{1},K_{1},\bP_{1})^{\phi_{1}} \to  I(Y_{0},K_{0},\bP_{0})^{\phi_{0}} .
\]
The previous definitions need no modification. There is a difference
however when we consider the composition law and
functoriality. Suppose that the above cobordism $(W,S)$ is broken into
the union of two cobordisms along some intermediate manifold-pair
$(Y_{1/2}, K_{1/2})$, so  $(W,S) = (W',S') \cup (W'', S'')$, with $W'$
a cobordism from $Y_{1}$ to $Y_{1/2}$. By restriction, $\phi$ gives
rise to $\phi'$ and $\phi''$, and $J$ gives rise to $J'$ and $J''$.
The composite map is equal to
the map arising from the composite cobordism only if an additional
hypothesis holds:

\begin{proposition}
\label{prop:composite-phi}
    In the above setting, we have
\[
        I(W,S,\bP)^{\phi} =  I(W'',S'',\bP'')^{\phi''} \circ I(W',S',\bP')^{\phi'}
\]
   provided that the group $\phi\subset H^{1}(W;\Z/2)$ contains the
   image $\psi$ of the connecting map in the Mayer-Vietoris sequence,
\[
             H^{0}(Y_{1/2} ; \Z/2) \to H^{1} (W; \Z/2).
\]
\end{proposition}

\begin{remarks}
    Note that $\psi$ is non-zero only if $Y_{1/2}$ has more than one
    component. In general, knowing $\phi'$ and
    $\phi''$ does not determine $\phi$ without additional information;
    but  the additional hypothesis that $\psi$ is contained in
    $\phi$, is enough to determine $\phi$. 
\end{remarks}

\begin{proof}[Proof of the Proposition]
  The hypothesis of the proposition ensures that the relevant moduli
  spaces on the composite cobordism have a fiber-product
  description. To illustrate the point in a simpler situation (so as
  to reduce the amount of notation involved), consider a closed Riemannian
  $4$-manifold $X$ decomposed into $X'$ and $X''$ along a 3-manifold
  $Y$ with more than one component. Suppose that the metric is
  cylindrical near $Y$. Let $P$ be an $\SO(3)$ bundle on
  $X$, let $M(X,P)$ be the moduli space of anti-self-dual connections,
  and suppose that all these anti-self-dual connections restrict to
  irreducible connections on each component of $Y$. We can then
  consider the Hilbert manifolds of anti-self-dual connections,
  $M(X',P')$ and $M(X'', P'')$ on the two manifolds with boundary, in
  a suitable $L^{2}_{k}$ completion of the spaces of connections (as in
  \cite[section~24]{KM-book}). There are restriction maps
   \[
              \xymatrix{
                          M(X', P') \ar[dr] & & M(X'', P'')\ar[dl] \\
                          &         \bonf_{k-1/2}(Y, P|_{Y}) &                 
	   }
   \]
   and one can ask whether the fiber product here is the same as
   $M(X,P)$. The answer is no, in general, because of our use of the
   determinant-1 gauge group: in the determinant-1 gauge
   transformation, the stabilizer of an irreducible connection on $Y$
   is $\{\pm 1\}^{n}$, where $n$ is the number of components of $Y$;
   but not every element in the stabilizer can be expressed as a
   ratio $(g')^{-1}(g'')$ of elements in the stabilizer of the
   connections on $X'$ and $X''$. Thus $M(X,P)$ may map many-to-one
   onto the fiber product. The hypothesis on $\psi$ in the proposition
   ensures that we are using gauge groups for which the corresponding
   moduli space on $X$ is exactly the fiber product.
\end{proof}

We can summarize the situation by saying that we have a functor to the
category of abelian groups from a suitable cobordism category. The morphisms in
this category consist of cobordisms $(W,S,\bP)$ equipped with an
almost-complex structure $J$ on $(W,S)$ and a subgroup $\phi$ of
$H^{1}(W;\Z/2)$. When composing two cobordisms, we must define $\phi$ on
the composite cobordism to be the \emph{largest} subgroup of
$H^{1}(W;\Z/2)$ which restricts to the given subgroups on the two
pieces. 

There is a variant of the above proposition corresponding to the case
that we glue one outgoing component of $W$ to an incoming
component. The context for this is that we have $(W,S)$ a cobordism
from $(Y_{1}, K_{1})$ to $(Y_{0},K_{0})$, where
\[
\begin{aligned}
    Y_{1} &= Y_{1,1} \cup\dots\cup Y_{1,r} \\
    Y_{ 0} &= Y_{0,1} \cup\dots\cup Y_{0,s} 
\end{aligned}
\]
with $Y_{1,r}=Y_{1,s}$. We suppose also that $\bP$ is given so that
its restriction to $Y_{1,r}$ and $Y_{0,s}$ are identified. From
$(W,S,\bP)$ we then form $(W^{*},S^{*}, \bP^{*})$ by gluing these two
components. This is a cobordism between manifolds $(Y^{*}_{1},K_{1})$
and $(Y^{*}_{0}, K^{*}_{0})$ with $r-1$ and $s-1$ components
respectively. 
Let $\phi^{*}$ be a subgroup of $H^{1}(W^{*};\Z/2)$ \emph{which
contains the class dual to the submanifold} $Y_{1,r}=Y_{0,s}$ where the
gluing has been made. Let $\phi$ be the subgroup of $H^{1}(W;\Z/2)$
obtained by pull-back via the map $W\to W^{*}$.  We then have:

\begin{proposition}\label{prop:trace}
    The map
      \[
      I(W^{*}, S^{*},\bP^{*})^{\phi^{*}} : I(Y^{*}_{1},
      K^{*}_{1},\bP^{*}_{1})^{\phi_{1}^{*}} \to I(Y^{*}_{0},
      K^{*}_{0},\bP^{*}_{0})^{\phi_{0}^{*}}
      \]
    is obtained from the map
      \[
      I(W, S,\bP)^{\phi} : I(Y_{1},
      K_{1},\bP_{1})^{\phi_{1}} \to I(Y_{0},
      K_{0},\bP_{0})^{\phi_{0}}
      \]
    by taking the alternating trace at the chain level over the
    $\Z/2$-graded factor
    \[
                   \Hom(C(Y_{1,r} , K_{1,r})^{\phi_{1,r}}, 
                       C(Y_{0,s} , K_{0,s})^{\phi_{0,s}}).
     \]
\end{proposition}

\begin{proof}
    There are two issues here. The first is issue of signs, to verify
    that it is indeed the alternating trace that arises: this  issue
    is dealt with in \cite[Lemma~2.4]{KM-sutures}, and the same
    argument applies here. The second issue is the choice of correct
    gauge groups, and this is the same point that arises in the
    previous proposition.
\end{proof}

\begin{remark}
    As a simple illustration, suppose that $Y$ is connected and $(W,S)$ is
    a trivial product cobordism from $(Y,K)$ to $(Y,K)$.
    Suppose that $\phi$ is trivial and $\phi^{*}$ is
    generated by the class dual to $Y$ in the closed
    manifold $W^{*}$ obtained by gluing the two ends. Then the closed
    pair $(W^{*}, S^{*})$ has an associated moduli space
    $M(W^{*},S^{*},\bP^{*})^{\phi^{*}}$ and an associated integer
    invariant obtained by counting points with sign. The proposition
    asserts that this integer is the euler characteristic of
    $I(Y,K,\bP)$. If we had stuck with the determinant-1 gauge group
    and examined $M(W^{*},S^{*},\bP^{*})$ instead, then the integer
    invariant of this closed manifold would have been twice as large.
\end{remark}

\subsection{Floer's excision revisited}

Floer's excision principle \cite{Floer-Durham, Braam-Donaldson}
concerns the following situation (see also \cite{KM-sutures}). Let $Y$
be a closed, oriented $3$-manifold \emph{not necessarily connected},
and let $T_{1}$, $T_{2}$ be a pair of oriented tori in $Y$, supplied
with an identification $h: T_{1}\to T_{2}$. Let $Y'$ be obtained from $Y$
by cutting along $T_{1}$ and $T_{2}$ and reglueing using the given
identification (attaching each boundary component arising from the cut
along $T_{1}$ to the corresponding boundary component arising from the
cut along $T_{2}$, respecting the boundary orientations). Again, $Y'$
need not be connected. We suppose that $Y$ contains a link $K$,
disjoint from the tori, and we denote by $K'$ the resulting link in
$Y'$. We also suppose that singular bundle data $\bP$ is given on $Y$,
and that the identification $h : T_{1} \to T_{2}$ is lifted to an
identification $\bP|_{T_{1}} \to \bP|_{T_{2}}$, so that we may form
singular bundle data $\bP'$ on $Y'$. We require as a hypothesis that
$w_{2}(\bP)$ is non-zero on $T_{1}$ (and hence on $T_{2}$). From an
alternative point of view, we may regard $\bP$ as being determined by
a $1$-manifold $\omega$, in which case we ask that $\omega\cdot T_{1}$
is odd and that $h$ maps the transverse intersection $\omega\cap
T_{1}$ to $\omega\cap T_{2}$, so that we may form $\omega'$ in $Y'$.

When $K$ and $K'$ are absent (which was the case in Floer's original
setup), the condition that $\omega\cdot T_{1}$ is odd forces $T_{1}$
to be non-separating. In the presence of $K$, however, it may be that
$T_{1}$ separates $Y$, in which case $\omega$ must have an arc which
joins components of $K$ which lie in different components of
$Y\setminus T_{1}$.

Let $\phi \subset H^{1}(Y;\Z/2)$ be the subgroup generated by the
duals of $T_{1}$ and $T_{2}$, and let $\phi' \subset H^{1}(Y';\Z/2)$
be defined similarly using the tori in $Y'$. As long as $Y$ and $Y'$
have no components disjoint from  the tori, the $\phi$-non-integral
condition is satisfied on account of the presence of the surfaces
$T_{i}$. If there are any other components of $Y$ (and therefore of
$Y'$), we impose the non-integral condition as a hypothesis, as usual
(though such components are irrelevant in what follows). In order to
fix signs, we use almost-complex structures as in
section~\ref{subsec:almost-complex}: we fix an $\R$-invariant complex
structure $J$ on $(\R\times Y, \R\times K)$, which we choose in such a
way that the manifolds $\{0\}\times T_{i}$ are almost-complex
submanifolds (with their given orientations). By cutting and gluing we
also obtain an almost complex structure $J'$ on $Y'$.

\begin{theorem}\label{thm:excision}
    Under the above hypotheses, there are mutually-inverse
    isomorphisms
    \[
                 I(Y,K,\bP)^{\phi} \longleftrightarrow I(Y', K', \bP')^{\phi'},
     \]
     or equivalently
    \[
           I^{\omega}(Y,K)^{\phi} \longleftrightarrow I^{\omega'}(Y', K')^{\phi'},
     \]
     arising from standard cobordisms $(W,S)$ and $(\bar W, \bar S)$, from $(Y, K)$ to $(Y',
     K')$ and  from $(Y',K')$ to $(Y,K)$ respectively.
\end{theorem}

\begin{proof}
    The standard cobordism $W$ from $Y$ to $Y'$ is the one that
    appears in Floer's original theorem and is illustrated in
    \cite[Figure 2]{KM-sutures}. The relevant part of this cobordism is redrawn
    here in Figure~\ref{fig:excision-1}.
\begin{figure}
    \begin{center}
        \includegraphics[height=2in]{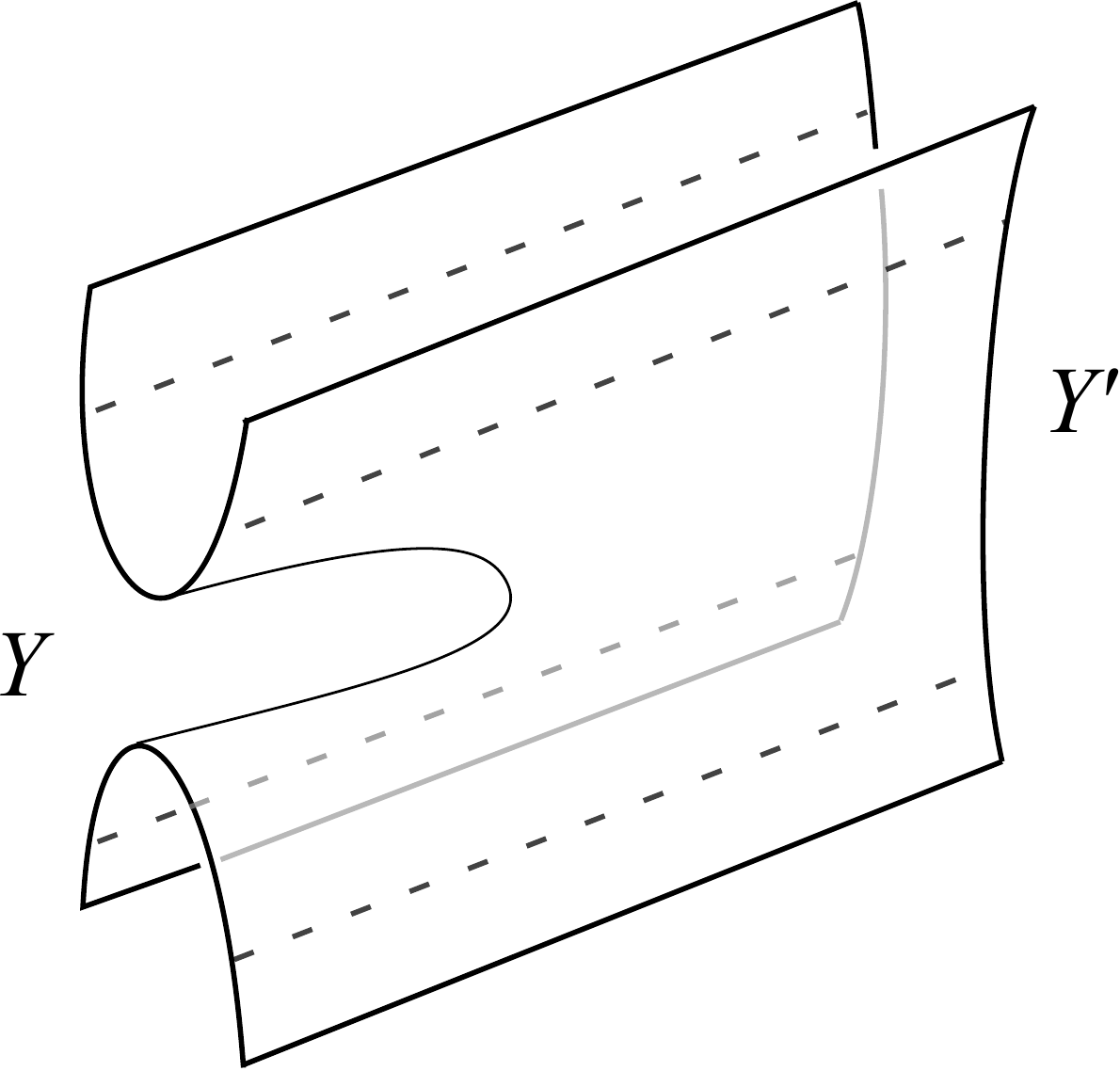}
    \end{center}
    \caption{\label{fig:excision-1}
    The excision cobordism $W$ from $Y$ to $Y'$.}
\end{figure}
    This part is a product
    $T\times U$, where $T$ is the torus obtained by identifying
    $T_{1}$ and $T_{2}$ using $h$ and $U$ is the $2$-manifold with
    corners depicted as the shaded part in the figure. The subset
    $T\times U\subset W$ meets $Y$ and $Y'$ in $2$-sided collar
    neighborhoods of $T_{1}\cup T_{1}$ and $T'_{1}\cup T'_{2}$.
    The links $K$
    and $K'$ are contained in the parts of $Y$ and $Y'$ that are
    disjoint from these collars, and there is a surface $S=[0,1]\times
    K$ lying in the remaining part of $W$, providing the cobordism
    between them. Our choices equip $W$ with singular bundle date
    $\bP_{W}$ and an almost-complex structure $J_{W}$ which is a
    product structure on the subset $T\times U$. 

    Let $T_{i}^{+}$ and $T_{i}^{-}$ be positive and negative push-offs
    of $T_{i}$ in $Y_{i}$ (for $i=1,2$). The dual classes to these
    four tori redundantly generate the same group $\phi\subset
    H^{1}(Y;\Z/2)$. Let $(T'_{i})^{\pm}$ be defined similarly in
    $Y'$. These eight tori sit over the eight corners of $U$ in the
    figure. For each of the four dotted edges of $U$, there is copy of
    $[0,1]\times T$ lying above it in $W$. 
    Let $\phi_{W}\subset H^{1}(W;\Z/2)$ be the subgroup
    generated by the classes dual to these four copies of $[0,1]\times
    T$. This subgroup restricts to $\phi$ and $\phi'$ on the two
    ends. Equipped with this data and its complex structure $J_{W}$,
    the cobordism thus gives rise to a map
\[
  I(W,S,\bP_{W})^{\phi_{W}} : I(Y,K,\bP)^{\phi} \longrightarrow I(Y',
  K', \bP')^{\phi'}.
\]
    An entirely symmetrical construction gives a map in the opposite
    direction,
\[
  I(\bar W,\bar S,\bP_{\bar W})^{\phi_{\bar W}} : I(Y',
  K', \bP')^{\phi'}\longrightarrow  I(Y,K,\bP)^{\phi} .
\]
 
   To show that the maps arising from these cobordisms are mutually
   inverse, we follow Floer's argument, as described in
   \cite{Braam-Donaldson, KM-sutures}. The setup is symmetrical, so we
   need only consider one composite, say the union
   \[
              V = W\cup_{Y'} \bar W
    \]
    with its embedded surface $S_{V} = S\cup\bar S$. This composite
    cobordism, part of which is depicted in
    Figure~\ref{fig:excision-composite}, 
\begin{figure}
    \begin{center}
        \includegraphics[height=2.4in]{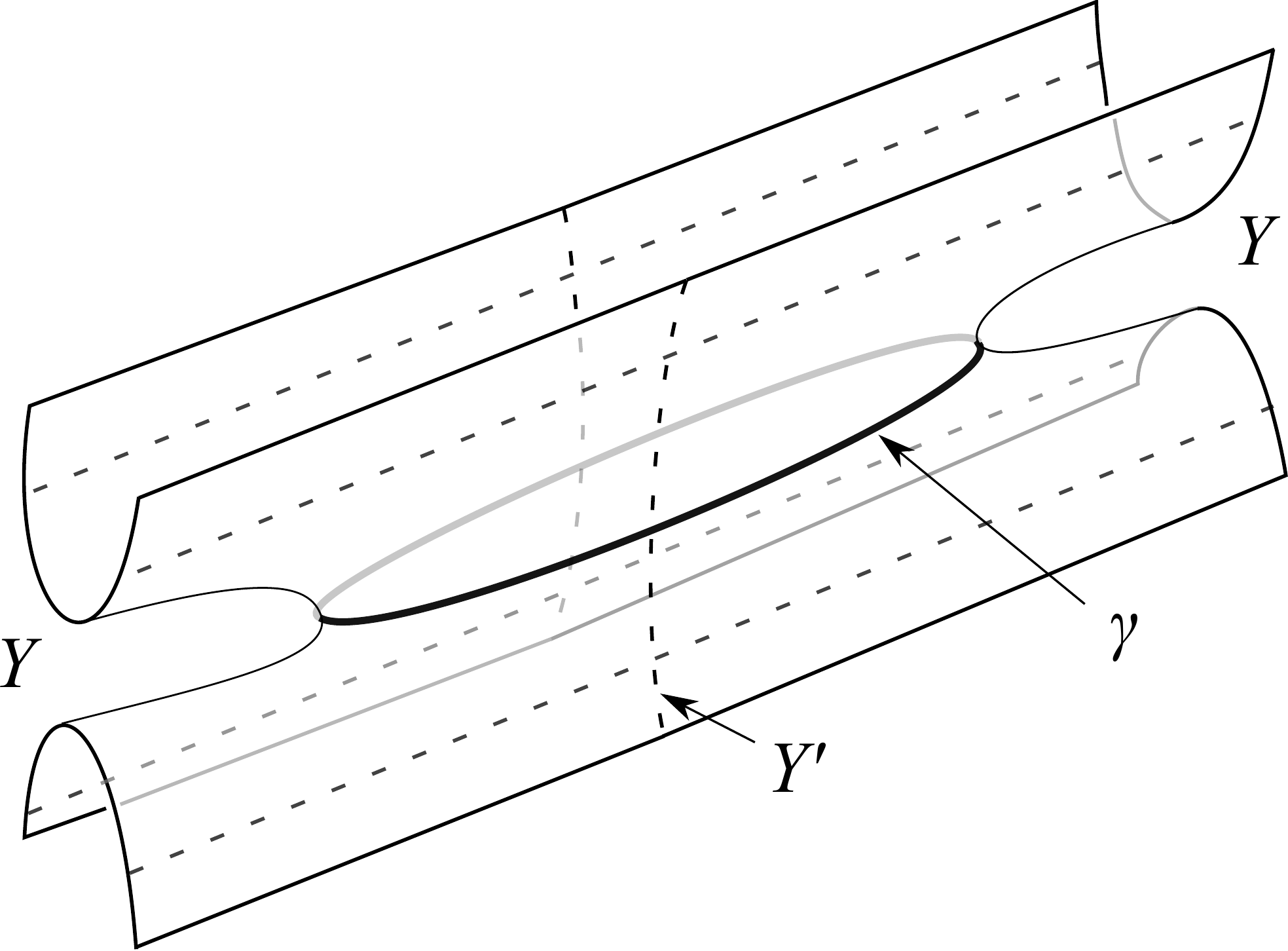}
    \end{center}
    \caption{\label{fig:excision-composite}
    The composite cobordism $V = W\cup_{Y'} \bar W$ from $Y$ to $Y$.}
\end{figure}
    contains four copies of
    $[0,2]\times T$ (lying above the dotted lines again). We take
     \[
              \phi_{V} \subset H^{1}(V;\Z/2)
      \]
      to be the subgroup generated by the dual classes of these four
      hypersurfaces together with the image of the boundary map in the
      Mayer-Vietoris sequence: i.e. the classes dual to the components
      of $Y'\subset V$. The complex structure on $V$ is the product
      structure on the subset $T\times (U \cup  \bar U)$ shown in the
      figure. By Proposition~\ref{prop:composite-phi}, the composite
      of the maps obtained from the cobordisms $(W,S)$ and
      $(\bar W,\bar S)$ is the map
       \[
                         I(V, S_{V}, \bP_{V})^{\phi_{V}}.
        \]
         So we must show that this map is the identity.

         Floer's proof rests on the fact that the cobordism $V$ can be
         changed to a product cobordism by cutting along the $3$-torus
         $T\times \gamma$ (where $\gamma$ is the curve depicted in the
         figure) and gluing two copies of $T\times D^{2}$ to the
         resulting boundary components. If we write $V'$ for the
         resulting product cobordism, we see that it contains  a
         product surface $S_{V}$ and acquires from $V$ (by cutting and
         gluing) a complex structure $J_{V'}$ which  respects the
         product structure. Furthermore, the subgroup $\phi_{V}$
         becomes a group $\phi_{V'}$ in $H^{1}(V;\Z/2)$ which is the
         pull-back of $\phi$ from $H^{1}(Y;\Z/2)$. Thus
       \[
                   I(V', S_{V'}, \bP_{V'})^{\phi_{V'}} = \mathrm{Id}
        \]
        and we are left with the task of showing that $V$ and $V'$
        (with their attendant structures) give the same map on
        $I(Y,S,\bP)^{\phi}$. This last task is easily accomplished by
        using the fact that $\phi_{V}$ restricts to the $3$-torus
        $T\times \gamma$ to give the non-trivial $2$-element subgroup
        of $H^{1}(T\times\gamma;\Z/2)$ generated by $T\times
        \{\mathrm{point}\}$, so that
         \[
                         I(T\times\gamma, P_{V})^{\phi_{V}} = \Z,
         \]
         by Lemma~\ref{lem:T3-phi-calc}. The subgroup $\phi_{V}$ also
         contains the class dual to $T\times\gamma$, so that
         Proposition~\ref{prop:trace} applies. The relative invariant of
         $T\times D^{2}$ is $1\in \Z$ in the Floer group of
         $T\times\gamma$, so the result follows, just as in Floer's
         original argument.
\end{proof}

\subsection{Proof of Proposition~\ref{prop:singular-to-sutured}}

We now apply the excision principle, Theorem~\ref{thm:excision}, to
prove Proposition~1.4 from the introduction. The proposition can be
generalized to deal with knots $K$ in $3$-manifolds other than
$S^{3}$, so we consider a general connected, oriented $Y$ and a knot 
$K\subset Y$. From $(Y,K)$ we form a new closed $3$-manifold
$T^{3}_{K}$, depending on a choice of framing for $K$, 
as follows. We take a standard coordinate circle $C\subset
T^{3}$ and we glue together the knot complements to form
\[
      T^{3}_{K} =    \bigl(  T^{3} \sminus N^{\circ}(C) \bigr) \cup
                     \bigl( Y \sminus N^{\circ}(K) \bigr),
\]
gluing the longitudes of $K$ (for the chosen framing) to the meridians
of $C$ and vice versa. In $T^{3}$, let $R$ be coordinate $2$-torus
parallel to $C$ and disjoint from it, and let $\psi \subset
H^{1}(T^{3}_{K};
\Z/2)$ be the $2$-element subgroup generated by the dual to $[R]$ in
$T^{3}_{K}$. Let $\omega_{1}\subset T^{3}$ be another coordinate circle
transverse to $R$ and meeting it in one point.

We may now consider the instanton Floer homology group
$I^{\omega_{1}}(T^{3}_{K})$, 
as well as its companion $I^{\omega_{1}}(T^{3}_{K})^{\psi}$ defined using
the larger gauge group. From the first half of
Proposition~\ref{prop:simple-doubling}, we have
\begin{equation}\label{eq:T3K-doubles}
I^{\omega_{1}}(T^{3}_{K}) = I^{\omega_{1}}(T^{3}_{K})^{\psi} \oplus
I^{\omega_{1}}(T^{3}_{K})^{\psi}.
\end{equation}
Our application of excision is the following result:

\begin{proposition}\label{singular-to-sutured-2}
    There is an isomorphism $\Inat(Y,K) \cong I^{\omega_{1}}(T^{3}_{K})^{\psi}$,
    respecting the relative $\Z/4$ grading of the two groups.
\end{proposition}

\begin{proof}
    By definition, we have $\Inat(Y,K) = I^{\omega}(Y,K^{\natural})$, where
    $K^{\natural}$ is the union of $K$ and a meridional circle $L$ and
    $\omega$ is an arc joining
    the two components.  The circle $L$ has a preferred framing,
    because it is contained in a small ball in $Y$. 
    Let $N(K)$ be a tubular neighborhood of $K$
    that is small enough to be disjoint from $L$, and let $T_{1}$ be
    its oriented boundary. Let $T_{2}$ be the oppositely-oriented
    boundary of a disjoint tubular neighborhood of $L$ in $Y$. Let
    $h:T_{1}\to T_{2}$ be an orientation-preserving diffeomorphism
    that maps the longitudes of $K$ to the meridians of $L$ and vice
    versa, and maps $\omega\cap T_{1}$ to $\omega\cap T_{2}$. We can
    cut along $T_{1}$ and $T_{2}$ and reglue using
    $h$. Theorem~\ref{thm:excision} applies. Note that because $T_{1}$
    and $T_{2}$ are null-homologous, the relevant subgroup $\phi$ in
    $H^{1}(Y;\Z/2)$ is trivial. The pair $(Y', K')$ that results from
    cutting and gluing has two components. One component is a sphere
    $S^{3}$ containing $K'$ which is the standard Hopf link $H$. 
   The other component
    is $T^{3}_{K}$. The tori $T'_{1}$ and $T'_{2}$ are respectively 
    a standard
    torus separating the two components of the Hopf link in $S^{3}$,
    and the torus $R$ in $T^{3}_{K}$. The subgroup of $H^{1}(Y';\Z/2)$
    that they generate is the 2-element group $\psi$, while $\omega'$
    is the union of an arc joining the two components of $H$ and the
    coordinate circle $\omega_{1}$ in $T^{3}$. 
   Since $I^{\omega'}(S^{3}, H) =\Z$,  the excision
    theorem gives
\[
\begin{aligned}
    I^{\omega}(Y,K^{\natural}) &\cong I^{\omega'}(S^{3}, H) \otimes
    I^{\omega'}(T^{3}_{K})^{\psi} \\
                     &=  I^{\omega_{1}}(T^{3}_{K})^{\psi} ,
\end{aligned}
\]
which is what the proposition claims.
\end{proof}

The group $I^{\omega_{1}}(T^{3}_{K})$ (defined using the determinant-1
gauge group) is exactly the group that Floer associates to a knot $K$
in \cite[section 3]{Floer-Durham}. In that paper, this homology group is
written $I_{*}(P, K)$, but to keep the distinctions a little clearer,
let us write Floer's group as $I^{\mathrm{Floer}}(Y,K)$. 
Because of the relation \eqref{eq:T3K-doubles},
we can recast the result of the previous proposition as
\[
            I^{\mathrm{Floer}}(Y,K) = \Inat(Y,K) \oplus \Inat(Y,K).
\]
On the other hand, in \cite{KM-sutures}, it was explained that, over
$\Q$ at least, one can decompose $I^{\mathrm{Floer}}(Y,K)$ into the
generalized eigenspaces of degree-4 operators $\mu(\mathrm{point})$,
belonging to the eigenvalues $2$ and $-2$. The generalized eigenspace
for $+2$ is, by definition, the group $\KHI(Y,K)$ of
\cite{KM-sutures}. Since the two generalized eigenspaces are of equal
dimension, we at least have
\[
    I^{\mathrm{Floer}}(Y,K;\Q) = \KHI(Y,K;\Q) \oplus  \KHI(Y,K;\Q)
\]
as vector spaces. Thus we eventually have
\[
           \Inat(Y,K) \otimes \Q \cong \KHI(Y,K;\Q)
\]
as claimed in Proposition~\ref{prop:singular-to-sutured}. \qed

\subsection{A product formula for split links}

For a second application of Floer's excision theorem, consider a pair
of connected $3$-manifolds $Y_{1}$, $Y_{2}$, and their connected sum
$Y_{1}\# Y_{2}$, as well as their disjoint union $Y=Y_{1}\cup Y_{2}$. 
Given links $K_{i}\subset Y_{i}$ for $i=1,2$, chosen so as to be
disjoint from the embedded balls that are used in making the connected
sum, we obtain a link $K_{1} \cup K_{2}$, in $Y_{1}\# Y_{2}$. In the special
case that $Y_{1}$ and $Y_{2}$ are both $S^{3}$, the resulting link
is a split link (as long as both $K_{i}$ are non-empty).  Of course,
we can also form the union $K$ in $Y=Y_{1}\cup Y_{2}$.  Note that any
cobordism $S$ from $K$ to $K'$ in the disjoint union $Y$ gives rise
also to a ``split'' cobordism $S$ between the corresponding links in
$Y_{1}\# Y_{2}$, as long it is disjoint from the balls.

Let
$H_{i}\subset Y_{i}$ be a Hopf  link contained in a standard ball, and
let $\omega_{i}$ be an arc joining its two components, so that
\[
         \Isharp(Y_{i}, K_{i}) = I^{\omega_{i}}(Y_{i}, K_{i} \cup H_{i}).
\]
Write $\omega = \omega_{1} \cup \omega_{2}$. The group
\[
           I^{\omega}(Y, K \cup H_{1} \cup H_{2})
\]
is isomorphic to the tensor product 
\[
 \Isharp(Y_{1}, K_{1}) \otimes  \Isharp(Y_{2}, K_{2}) 
\]
over the rationals, on a account of the K\"unneth theorem
\eqref{eq:Kunneth}. On the other hand, it is also related to the
connected sum:

\begin{proposition}
    There is an isomorphism 
     \[
           I^{\omega}(Y, K \cup H_{1} \cup H_{2})
                  \cong \Isharp (Y_{1}\# Y_{2}, K)
     \]
   which respects the $\Z/4$ gradings, and is natural for ``split cobordisms''.
\end{proposition}

\begin{proof}
    In $Y$, let $T_{1}$ be a the boundary of a tubular neighborhood of
    one of the two components of the Hopf link $H_{1}\subset Y_{1}$,
    and let $T_{2}\subset Y_{2}$ be defined similarly, but with the
    opposite orientation. Choose an orientation-preserving
    diffeomorphism $h$ between these tori, interchanging longitudes
    with meridians. The manifold $(Y', K')$ obtained by cutting and
    gluing as in the excision theorem is disconnected. Its first
    component $Y'_{1}$ is a 3-sphere containing a standard Hopf link
    $K'_{1}$. Its second component $Y'_{2}$ is the connected sum
    $Y_{1}\# Y_{2}$ containing the link
    \[
                       K'_{2} = K_{1} \cup K_{2} \cup H',
     \]
     where $H'$ is a standard Hopf link. All the tori involved in this
     application of excision are null-homologous, so no non-trivial
     subgroups of $H^{1}(Y;\Z/2)$ are involved. Thus we obtain from
     Theorem~\ref{thm:excision} an isomorphism
     \begin{equation}\label{eq:split-excision}
         \begin{aligned}
             I^{\omega}(Y, K \cup H_{1} \cup H_{2}) & \cong
             I^{\omega'}(Y'_{1}, K'_{1}) \otimes I^{\omega'}(Y'_{2},
             K'_{2}) \\
            & = I^{\omega'}(Y'_{1}, K'_{1}) \otimes I^{\omega'}(Y_{1}
            \# Y_{2},
             K_{1} \cup K_{2} \cup H') \\
         \end{aligned}
      \end{equation}
      On the right, the first factor is $\Z$. The curve $\omega'$ in
      the second factor is an arc joining the two components of the
      Hopf link $H'$. So the right-hand side is simply
      $\Isharp(Y_{1}\# Y_{2}, K_{1}\cup K_{2})$ as desired.
\end{proof}

\begin{corollary}\label{cor:split-link}
    If at least one of $\Isharp(Y_{i}, K_{i})$ is torsion-free, we
    have an isomorphism
    \[
          \Isharp(Y_{1}, K_{1}) \otimes \Isharp(Y_{2}, K_{2}) 
                    \to \Isharp (Y_{1} \# Y_{2}, K_{1} \cup K_{2})
     \]
     arising from an excision cobordism.  The isomorphism is natural
     for the maps induced by ``split'' cobordisms. \qed
\end{corollary}

\section{Cubes}
\label{sec:cubes}
\subsection{The skein cobordisms}

We consider three links $K_{2}$, $K_{1}$ and $K_{0}$ in a closed
$3$-manifold $Y$ which are related by the unoriented skein moves, as
shown in Figure~\ref{fig:Tetrahedra-skein}. What this means is that
there is a standard $3$-ball in $Y$ outside which the three links
coincide, while inside the ball the three links appear as shown. Two
alternative views are given in the figure. In the top row, we draw the
picture as it is usually presented for classical links, when a
projection in the plane is given: here the links $K_{1}$ and $K_{0}$
have one fewer crossings that $K_{2}$. In the bottom row of the
figure, an alternative picture is drawn which brings out the symmetry
more clearly: we see that there is a round ball $B^{3}$ in $Y$ with the
property that all three links meet the boundary sphere in four
points. These four points
form the vertices of a regular tetrahedron, and the links $K_{2}$,
$K_{1}$ and $K_{0}$ are obtained from the three different ways of 
joining the four vertices in two pairs, by pairs of arcs isotopic to
pairs of edges
of the tetrahedron.

\begin{figure}
    \begin{center}
        \includegraphics[scale=0.5]{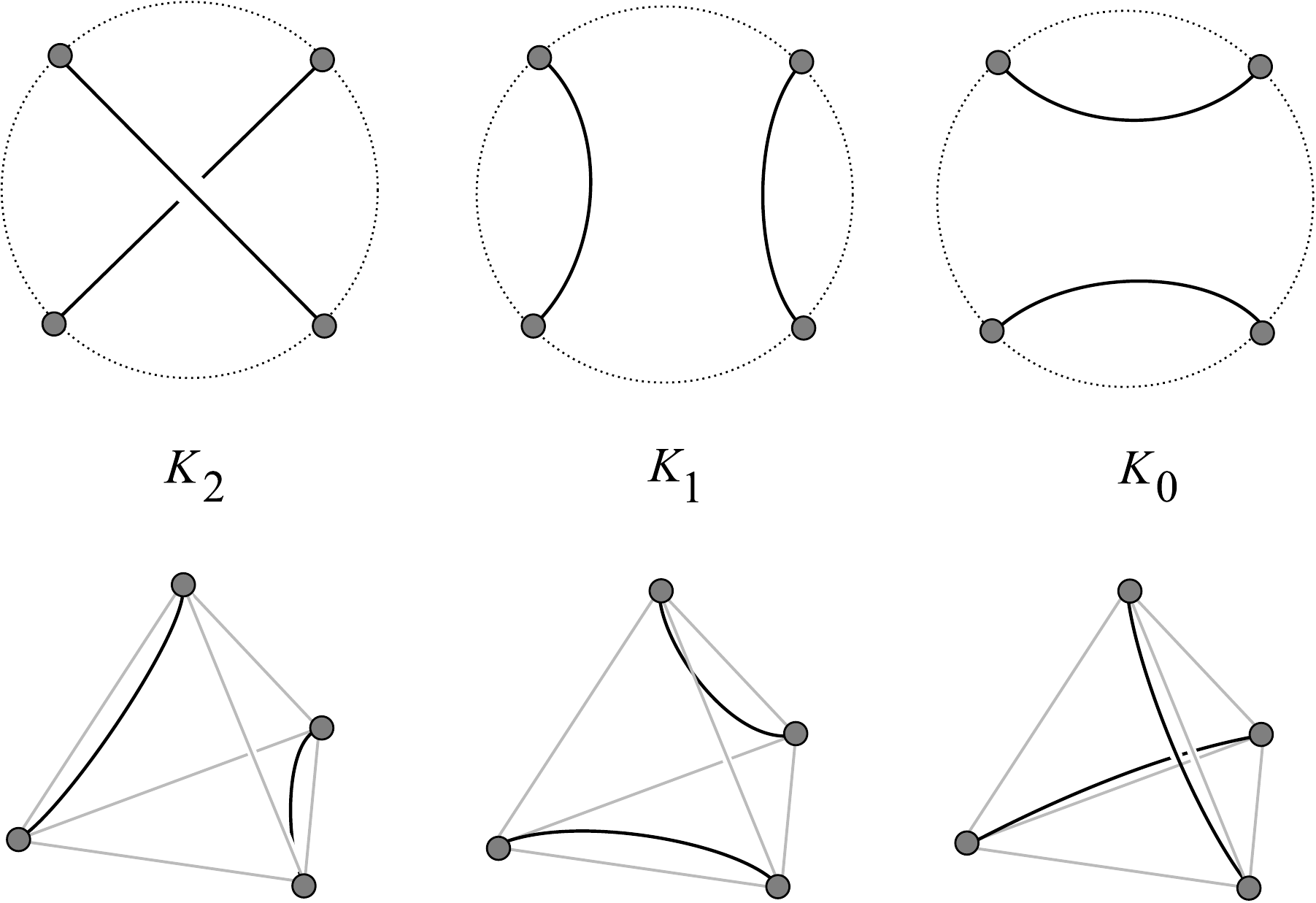}
    \end{center}
    \caption{\label{fig:Tetrahedra-skein}
    Knots $K_{2}$, $K_{1}$ and $K_{0}$ differing by the unoriented
    skein moves, in two different views.}
\end{figure}

The second view in Figure~\ref{fig:Tetrahedra-skein} makes clear the
cyclic symmetry of the three links. Note that there \emph{is} a
preferred cyclic ordering determined by the pictures: if the picture
of $K_{i}$ is rotated by a right-handed one-third turn about any of
the four vertices of the tetrahedron, then the result is the picture
of $K_{i-1}$. We may consider links $K_{i}$ for all integers $i$ by
repeating these three cyclically.

\begin{figure}
    \begin{center}
        \includegraphics[scale=0.5]{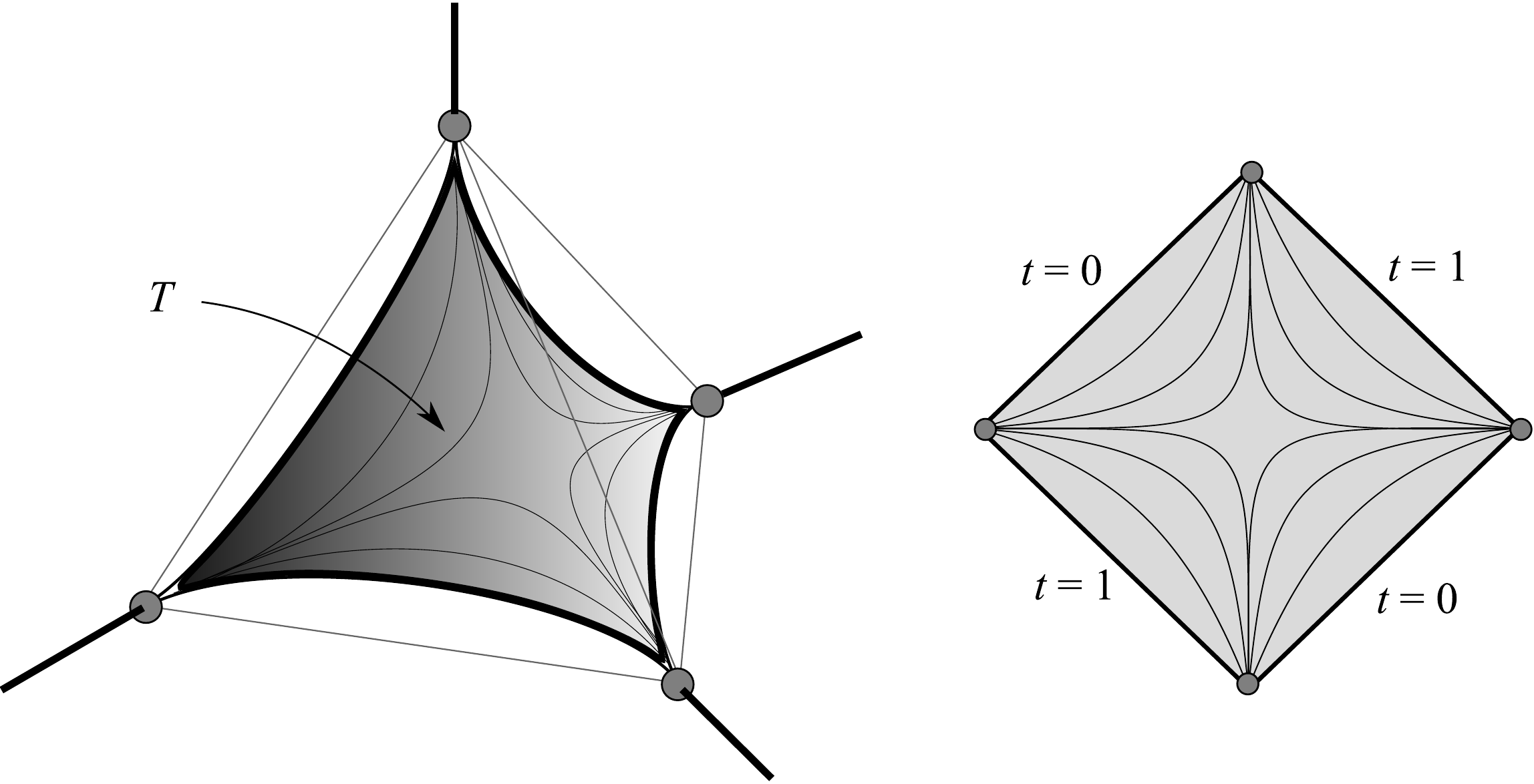}
    \end{center}
    \caption{\label{fig:Tetrahedron-2}
    The twisted rectangle $T$ that gives rise to the cobordism
    $\Tigma_{2,1}$ from $K_{2}$ to $K_{1}$.}
\end{figure}

We next describe a standard cobordism surface $\Tigma_{i,i-1}$ from
$K_{i}$ to $K_{i-1}$ inside the cylindrical $4$-manifold $[0,1]\times
Y$, for each $i$. Because of the cyclic symmetry, it is sufficient to
describe the surface $\Tigma_{2,1}$. This cobordism will be a product
surface outside $[0,1]\times B^{3}$. Inside $[0,1]\times B^{3}$, the
first coordinate $t\in[0,1]$ will have a single index-1 critical point
on $\Tigma_{2,1}$. The intrinsic topology of $\Tigma_{2,1}$ is
therefore described by the addition of a single $1$-handle.  To
describe the geometry of the embedding, begin with the $1/4$-twisted
rectangular surface $T\subset B^{3}$ shown in
Figure~\ref{fig:Tetrahedron-2}, and let $T^{o}$ be the complement of
the $4$ vertices of $T$. Let $t$ be a Morse function on $T^{o}$ with
$t=0$ on the two arcs of $K_{2}$ and $t=1$ on the two arcs of $K_{1}$,
and with a single critical point on the center of $T$ with critical value $1/2$. The graph of
this Morse function places $T^{o}$ into $[0,1]\times B^{3}$. The
cobordism $\Tigma_{2,1}$ is the union of this graph with the product
part outside $[0,1]\times B^{3}$.

We put an orbifold Riemannian metric $\orbig$ on $[0,1] \times Y$ (with
cone-angle $\pi$ along the embedded surface $\Tigma_{2,1}$ as
usual). We choose the metric so that it is a cylindrical product
metric of the form
\[
              dt^{2} + \orbig_{Y}
\]
on the subset $[0,1] \times (Y\setminus B^{3})$. We also require that
the metric be cylindrical in collar neighborhoods of $\{0\} \times Y$
and $\{1\}\times Y$.
\bigskip\bigskip

Suppose now that instead of a single ball we are
given $N$ disjoint balls $B_1,\ldots, B_N$ in $Y$. Generalizing 
the above notation, we may consider a collection of links $K_v \subset
Y$ for $v\in \{0,1,2\}^N$: all these links coincide outside the union
of the balls, while $K_{v}\cap B_{i}$ consists of a pair of arcs as in
Figure~\ref{fig:Tetrahedra-skein}, according as the $i$-th coordinate
$v_{i}$ is $0$, $1$ or $2$. We extend this family of links to a family
parametrized by $v\in \Z^{N}$, making the family periodic with period
$3$ in each coordinate $v_{i}$.

We give $\Z^N$ the product partial order. We also define norms
\[
\begin{aligned}
    |v|_{\infty} &= \sup_{i} |v_{i}| \\
    |v|_{1} &= \sum_{i} |v_{i}| . \\
\end{aligned}
\]
Then for a pair
$v,u\ \in \Z^N$ with $v\ge u$ and $|v-u|_{\infty}=1$, we define a cobordism $\Tigma_{vu}$
from $K_v$ to $K_u$ in $[0,1]\times Y$ by repeating the construction of $\Tigma_{v_{i},u_{i}}$
for each ball $B_i$ for which $v_{i}=u_{i}+1$. For the sake of uniform
notation, we also write $\Tigma_{vv}$ for the product cobordism. These
cobordisms satisfy
\[
          \Tigma_{wu}=\Tigma_{vu}\comp \Tigma_{wv}
\]
whenever $w\ge v \ge u$ with $|w-u|_{\infty}\le 1$.

It is notationally convenient to triangulate $\R^{N}$ as a
simplicial complex $\DDelta^{N}$  with vertex
set $\Z^{N}$ by
declaring the $n$-simplices to be all ordered $(n+1)$-tuples of vertices
$(v^{0}, \dots, v^{n})$ with
\[
       v^{0} > v^{1} > \dots > v^{n}
\]
and $|v^{0}-v^{n}|_{\infty} \le 1$. In this simplicial decomposition,
each unit cube in $\Z^{N}$ is decomposed into $N!$ simplices of
dimension $N$. The non-trivial cobordisms $\Tigma_{vu}$ correspond to
the $1$-simplices $(v,u)$ of $\DDelta^{N}$. We also talk of
\emph{singular} $n$-simplices for this triangulation, by which we mean
$(n+1)$-tuples $(v^{0}, \dots, v^{n})$ with
\[
       v^{0} \ge v^{1} \ge \dots \ge v^{n}
\]
and $|v^{0}-v^{n}|_{\infty} \le 1$. We can regard these singular
simplices as the generators of the singular simplicial chain complex,
which computes the homology of $\R^{N}$.

We will be applying instanton homology, to associate a chain complex
$C^{\omega}(Y,K_{v})$ to each link $K_{v}$. To do so, we need to have
an $\omega$ so that $(Y,K_{v},\omega)$ satisfies the non-integral
condition. We choose an $\omega$ which is disjoint from all the
balls $B_{1},\dots, B_{N}$. Such a choice for one $K_{v}$ allows us to
use the same $\omega$ for all other $K_{u}$. When considering the
cobordisms $\Tigma_{vu}$, we extend $\omega$ as a product. We impose
as a hypothesis the condition \emph{that $(Y,K_{v},\omega)$ satisfies
  the non-integral condition, for all $v$}.

We put an orbifold metric $\orbig_{v}$ on $(Y,K_{v})$ for every $v$,
arranging that these are all isometric outside the union of the $N$
balls. As in the case $N=1$ above, we then put an orbifold metric
$\orbig_{vu}$ on the cobordism of pairs, $([0,1]\times Y, S_{vu})$ for
every $1$-simplex $(v,u)$ of $\DDelta^{N}$. We choose these again so that
they are product metrics in the neighborhood of $\{0\}\times Y$ and
$\{1\}\times Y$, and also on the subset $[0,1] \times (Y\setminus B)$,
where $B$ is the union of the balls. Inside $[0,1]\times B_{i}$ we can
take standard metric for the cobordism, which depends only on $i$, not
otherwise on $u$ or $v$.

In order to have chain-maps with a well-defined sign, we need to
choose $I$-orientations for all the cobordisms that arise. From the
definition, this entails first choosing singular bundle data
\[
           \bP_{v} \to (Y,K_{v})
\]
for each $v$, corresponding to the chosen $\omega$. It also entails
choosing auxiliary data $\aux_{v}$ (a metric, perturbation and
basepoint in $\bonf^{\omega}(Y,K_{v})$), for all $v$. The metric is
something we have already discussed, but for the perturbation and
basepoint we make arbitrary choices. We fix
$(\bP_{v},\aux_{v})$ once and for all, and make no further reference
to them. For each cobordism $\Tigma_{vu}$, we have a well-defined
notion of an $I$-orientation, in the sense of
Definition~\ref{def:I-orientation}. We wish to choose $I$-orientations
for all the cobordisms, so that they behave coherently with respect to
compositions. The following lemma tells us that this is possible.

\begin{lemma}\label{lem:orientation-consistency}
    It is possible to choose $I$-orientations $\mu_{vu}$ for each
    cobordism $\Tigma_{vu}$, so that whenever $(w,v,u)$ is a singular $2$-simplex
    of $\DDelta^{N}$, the corresponding
    $I$-orientations are consistent with the composition, so that
\begin{equation}\label{eq:comp-mu}
           \mu_{wu}= \mu_{vu} \comp \mu_{wv}.
\end{equation}
\end{lemma}

\begin{proof}
    The proof only depends on the fact that the composition law for
    $I$-orientations is associative.
    Begin by choosing an arbitrary $I$-orientation $\mu'_{vu}$
    for each singular $1$-simplex $(v,u)$. For any singular $2$-simplex
    $(w,v,u)$, define $\eta(w,v,u) \in \Z/2$ according to whether or
    not the desired composition rule \eqref{eq:comp-mu} holds: that
    is,
\[
                      \mu'_{wu}= (-1)^{\eta(w,v,u)}\mu'_{vu} \comp \mu'_{wv}.
\]
   We seek  new orientations
\[
 \mu_{vu}= (-1)^{\theta(v,u)} \mu'_{vu}
\]
  so that the \eqref{eq:comp-mu} holds for all $2$-simplices. The
  $\theta$ that we seek can be viewed as a $1$-cochain on $\DDelta^{N}$
  with values in $\Z/2$, and the desired relation \eqref{eq:comp-mu}
  amounts to the condition that the coboundary of $\theta$ is $\eta$:
\[
          \delta\theta = \eta.
\]
  Because the second cohomology of $\R^{N}$ is zero, we can find such
  a $\theta$ if and only if $\eta$ is coclosed.

 To verify that $\eta$ is indeed coclosed, consider a singular $3$-simplex
$(w,v,u,z)$.  The value of $\delta \eta$ on this $3$-simplex is the sum of
the values of $\eta$ on its four faces. Choose any $I$-orientation $\nu$
for $\Tigma_{wz}$. There are four paths from $w$ to $z$ along oriented
$1$-simplices, $\gamma_{0}, \dots, \gamma_{3}$. Viewing these as
singular $1$-chains, they are:
\[
\begin{aligned}
          \gamma_{0} &= (w,z) \\
          \gamma_{1} &= (w,u) + (u,z) \\
          \gamma_{2} &= (w,v) + (v,u) + (u,z) \\
          \gamma_{3} &= (w,v) + (v,z) .
\end{aligned}
\]
For each path $\gamma_{a}$, $a=0,\dots, 3$, define $\phi_{a}\in \Z/2$
by declaring that $\phi_{a} = 0$ if and only if the composite of the
chosen $I$-orientations $\mu'$ along the $1$-simplices of $\gamma_{a}$
is equal to $\nu$. Thus, for example,
\[
          \mu'_{wv}\comp \mu'_{vu} \comp \mu'_{uz} = (-1)^{\phi_{2}}\nu.
\]
In the given cyclic ordering of the $1$-chains $\gamma_{a}$, the
differences $\gamma_{a+1} - \gamma_{a}$ is the boundary of a face of
the $3$-simplex, for each $a\in\Z/4$. Thus we see,
\[
\begin{aligned}
             \phi_{1}-\phi_{0} &= \eta(w,u,z) \\
             \phi_{2}-\phi_{1} &= \eta(w,v,u) \\
             \phi_{3}-\phi_{2} &= \eta(v,u,z) \\
             \phi_{0}-\phi_{3} &= \eta(w,v,z) .
\end{aligned}
\]
The sum of the value of $\eta$ on the $4$ faces of the singular $3$-simplex is
therefore zero, as required.
\end{proof}

To define maps in the Floer homology we make the pair
$([0,1]\times Y, \Tigma_{vu})$ into a pair of cylindrical manifolds,
by adding the half cylinders $(-\infty,0\,]\times (Y,K_v)$ and
$[\,1,\infty)\times (Y,K_u)$. The four manifold is simply
$\R\times Y$, and 
we call  the
$\R$ coordinate $t$.
The metric is locally a product metric everywhere
except on the subset
\[
             [0,1] \times (B_{i_{1}} \cup \dots \cup B_{i_{d}})
\]
where the union is over all $i$ with $v_{i}\ne u_{i}$. With a slight
abuse of notation, we continue to denote by $S_{vu}$ the non-compact
embedded surface with cylindrical ends,
\[
         S_{vu} \subset \R\times Y.
\]

This orbifold metric on $\R\times Y$ is one of a natural family
of metrics, which we now describe. Let us write
\[
\begin{aligned}
    I &= \{ \, i_{1}, \dots, i_{d}\,\} \\
      &= \supp (v-u) \\
      & \subset \{ \, 1,\dots, N \,\},
\end{aligned}
\]
and let us denote our original orbifold metric by $\orbig_{vu}(0)$.
For each
\[
    \tau = (\tau_{i_{1}}, \dots , \tau_{i_{d}}) \in \R^{I} 
   \cong \R^{d},
\]
we construct a Riemannian manifold
\[
              ( \R\times Y, \orbig_{vu}(\tau) )
\]
by starting with $( \R\times Y, \orbig_{vu}(0) )$,
cutting out the subsets $\R\times B_{i_{m}}$ for $m=1,\dots, d$, and
gluing them back via the isometry of the boundaries $\R\times  S^{3}$
given by translating the $t$ coordinate by $\tau_{i_{m}}$. After an
adjustment of our parametrization, we can assume that the $t$
coordinate on $\R\times\orbiY$ has exactly $d$ critical points when
restricted to the singular locus $S_{vu}$, and that these occur in
$\{\tau_{i_{m}}\} \times B_{i_{m}}$ for $m=1,\dots, d$. (See
Figure~\ref{fig:metrics}.)
\label{page:Gvu-setup}
 We write $G_{vu}$ for this family of
Riemannian metrics.
\begin{figure}
    \begin{center}
        \includegraphics[height=2in]{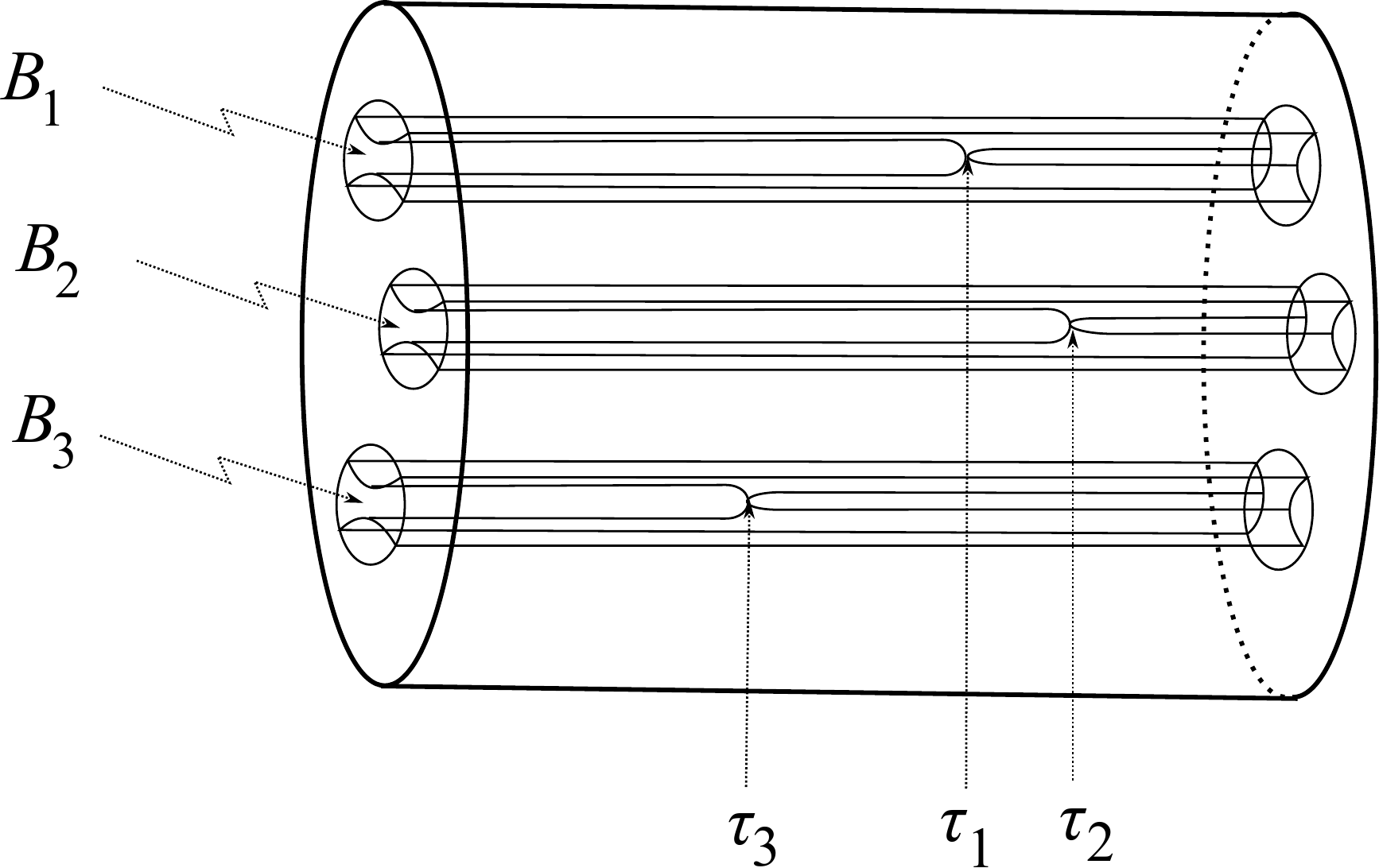}
    \end{center}
    \caption{\label{fig:metrics}
    The family of metrics $G_{vu}$ parametrized by $\tau\in\R^{I}$
    ($\R^{3}$ in this example).}
\end{figure}

Suppose now that $(w,v,u)$ be a singular $2$-simplex.  Let $I$ and $J$
be the support of $w-v$ and $v-u$ respectively, so that $I\cup J$ is
the support of $w-u$ and $I\cap J$ is empty. There is a natural
identification
\begin{equation}\label{eq:simple-product}
       G_{wu} \to G_{wv} \times G_{vu}
\end{equation}
arising from $\R^{I\cup J} \to \R^{I} \times\R^{J}$.

\begin{lemma}\label{lem:G-orient}
    Orientations can be chosen for $G_{vu}$ for all singular $1$-simplices
    $(v,u)$ of $\DDelta^{N}$ such that for all singular $2$-simplices $(w,v,u)$, the
    natural identification \eqref{eq:simple-product} is
    orientation-preserving.
\end{lemma}

\begin{proof}
   The proof is essentially the same as the proof of the previous lemma.
\end{proof}

There is an action of $\R$ on the space of metrics $G_{vu}$ by translation
(adding a common constant to each 
coordinate of $\tau \in \R^{d}$). We can therefore normalize $\tau$ by requiring that
\[
   \sum \tau_{i_{a}}=0.
\]  
We write $\bG_{vu} \subset G_{vu}$ for this normalized family (which
we can also regard as the quotient of $G_{vu}$ by the action of
translations).  As coordinates on $\bG_{vu}$ we can take the differences
\[
        \tau_{i_{a+1}}-\tau_{i_{a}}, \qquad a=1,\dots, d.
\]
There is a natural compactification of $\bG_{vu}$, which we can think
of informally as resulting from allowing some of the differences
to become infinite: it parametrizes a family of broken Riemannian metrics of
the sort considered in section~\ref{subsec:families}. This
compactification, which we call $\bG_{vu}^{+}$, is constructed as
follows. We consider all simplices $\sigma  = (v^{0},\dots, v^{n})$ with
$v^{0}=v$ and $v^{n}=u$ (including the $1$-simplex $(v,u)$ itself
amongst these). For each such simplex $\sigma$, we write
\[
             \bG_{\sigma} = \bG_{v^{0}v^{1}} \times \dots \times \bG_{v^{n-1}v^{n}}.
\]
The compactification $\bG^{+}_{vu}$ is then the union
\begin{equation}\label{eq:bG-compactification}
              \bG^{+}_{vu} = \bigcup_{\sigma} \bG_{\sigma}.
\end{equation}
The definition of the topology on this union follows the usual
approach for broken trajectories: see for example \cite{KM-book}.
The space $\bG^{+}_{vu}$ is a polytope: if $\sigma$ is
an $n$-simplex, then the corresponding subset
$\bG_{\sigma}\subset\bG^{+}_{vu}$ is the interior of a face of
codimension $n-1$. In particular, the codimension-1 faces of the
compactification are the parts
\[
            \bG_{vs}\times \bG_{su}
\]
for all $s$ with $v>s>u$. Thus each face parametrizes broken
Riemannian metrics broken along a single cut $(Y, K_{s})$, for some
$s$.
(A family of Riemannian metrics with much the
same structure occurs in the same context in \cite{Bloom}, where it is
observed that the polytope is a permutohedron.)

In section~\ref{subsec:families}, when we considered general families
of broken metrics, we chose to orient the boundary faces of the family
using the boundary orientation. In our present situation, we need to
compare the naturally-arising orientations to the boundary orientation:

\begin{lemma}\label{lem:bG-sign}
    Suppose that orientations have been chosen for $G_{vu}$ for all
    $1$-simplices $(v,u)$ so as to satisfy the conditions of
    Lemma~\ref{lem:G-orient}. Orient $\bG_{vu}$ by making the
    identification
\[
             G_{vu} = \R\times \bG_{vu}
\]
   where the $\R$ coordinate is the center of mass of the coordinates
   $\tau_{i_{a}}$ on $G_{vu}$. Then for any $2$-simplex $(w,v,u)$ the
   product orientation on $\bG_{wv}\times \bG_{vu}$ differs from the
   boundary orientation of
\[
                     \bG_{wv}\times \bG_{vu} \subset \partial \bG_{wu}
\]
  by the sign $(-1)^{\dim \bG_{wv}}$.
\end{lemma}

\begin{proof}
    From the identification $G_{wu} = G_{wv}\times G_{vu}$ we obtain
    an orientation-preserving identification
\[
           \R_{0} \times \bG_{wu} \cong (\R_{1}\times \bG_{wv}) \times
           (\R_{2}\times \bG_{vu}),
\]
where the $\R_{0}$, $\R_{1}$, $\R_{2}$ factors correspond to centers
of mass of the appropriate $\tau_{i}$. Thus, the $\R_{0}$ coordinate
on the left is a positive weighted sum of the $\R_{1}$ and $\R_{2}$
coordinates on the right. This becomes an orientation-preserving
identification
\[
            \bG_{wu} \cong  \bG_{wv}  \times
           \R_{3}\times \bG_{vu},
\]
where the coordinate $\R_{3}$ is related to the previous $\R_{1}$ and
$\R_{2}$ coordinates by $t_{3}=t_{2}-t_{1}$. The boundary component
$\bG_{wv}\times\bG_{vu}$ in $\partial \bG_{wu}^{+}$ arises by letting
the $\R_{3}$ coordinate go to $+\infty$, so
\[
 \bG_{wu}^{+} \supset   \bG_{wv}  \times
           (-\infty,+\infty] \times \bG_{vu}.
\]
The orientation of the boundary is determined by the
outward-normal-first convention, which involves switching the order of
the first two factors on the right. This introduces the sign $(-1)^{\dim \bG_{wv}}$.
\end{proof}

\subsection{Maps from the cobordisms}

We continue to consider the collection of links $K_{v}$ in $Y$
indexed by $v \in \Z^{N}$. Recall that we have singular bundle data
$\bP_{v}$ over each $(Y,K_{v})$ satisfying the non-integral condition,
and fixed auxiliary data $\aux_{v}$, so that we realize explicit chain
complexes for each Floer homology group: we write
\[
        C_{v} =  C_{*}(Y,K_{v},\bP_{v})
\]
for this chain complex. Its homology is $I^{\omega}(Y,K_{v})$. 
We write $\Crit_{v}$ for
the critical points, so that
\[
      C_{v} = \bigoplus_{\beta\in \Crit_{v}} \Z\Lambda(\beta).
\]

Now suppose that $(v,u)$ is a singular $1$-simplex, and let
$\beta \in \Crit_{v}$, $\alpha\in \Crit_{u}$. The cobordism of pairs,
with cylindrical ends attached, namely the pair
\[
         (\R\times Y, \Tigma_{vu}),
\]
carries the family of Riemannian metrics $G_{vu}$ (trivial in the case
that $v=u$). We choose generic secondary perturbations as in
section~\ref{subsec:families}, and we write
\[
   M_{vu}(\beta,\alpha)  \to G_{vu}
\]
for the corresponding parametrized moduli space. 

There is an action of $\R$ on $M_{vu}$ by translations, covering the
 action of $\R$ on $G_{vu}$ given by
\[
    (\tau_{i_{1}}, \dots, \tau_{i_{d}}) \mapsto 
         (\tau_{i_{1}}-t, \dots, \tau_{i_{d}}-t).
\]
(The choice of sign here is so as to match a related convention in
\cite{KM-book}.) In the special case that $v=u$, the cobordism is a
cylinder, and  $t\in \R$ acts on $M_{vv}$ by pulling back by the
translation $(\tau, y) \mapsto (\tau+t,y)$ of $\R\times Y$. When
$v=u$, we exclude the translation-invariant part of
$M_{vv}(\alpha,\alpha)$, so $\Mu_{vv}$ is the quotient by $\R$ only of
the non-constant solutions. When $v\ne
u$, the action of $\R$ on $G_{vu}$ is free and we can form the
quotient  $\bG_{vu}$ considered earlier, so that we have
\[
  \Mu_{vu}(\beta,\alpha)  \to \bG_{vu}.
\]
We can also choose to normalize by the condition
\[
 \sum_{a=1}^{d} \tau_{i_{a}}=0
\]
and so regard $\bG_{vu}$ as a subset of $G_{vu}$.

We have not specified the
instanton and monopole numbers here, so each $\Mu_{vu}(\beta,\alpha)$
is a union of pieces of different dimensions. We write
\[
   \Mu_{vu}(\beta,\alpha)_{d} \subset \Mu_{vu}(\beta,\alpha)
\]
for the union of the $d$-dimensional components, if any. 

As in section~\ref{subsec:families}, we may consider a natural completion of the space
$\Mu_{vu}(\beta,\alpha)$ over the polytope of broken Riemannian
metrics 
$\bG^{+}_{vu}$ (defined at \eqref{eq:bG-compactification} above):  we
call this completion $\Mubk_{vu}(\beta,\alpha)$. 
To describe it explicitly in the present set-up, 
we consider all singular simplices $\sigma  = (v^{0},\dots, v^{n})$ with
$v^{0}=v$ and $v^{n}=u$ (including the $1$-simplex $(v,u)$ itself
amongst these). For each such simplex $\sigma$ and each sequence
\[
   \bbeta = (\beta^{0}, \dots \beta^{n})
\]
with $\beta^{0}=\beta$ and $\beta^{n}=\alpha$, we consider the product
\begin{equation}\label{eq:Msigma}
   \Mu_{\sigma}(\bbeta) =  
 \Mu_{v^{0}v^{1}}(\beta^{0},\beta^{1}) \times \dots \times \Mu_{v^{n-1}v^{n}}(\beta^{n-1},\beta^{n}).
\end{equation}
As a set, the completion is the union
\[
    \Mubk_{vu}(\beta,\alpha) =  \bigcup_{\sigma} \bigcup_{\bbeta} \Mu_{\sigma}(\bbeta).
\]
For each
singular simplex $\sigma$, there is a map
\[
      \Mu_{\sigma}(\bbeta) \to \bG_{\sigma'}
\]
where $\sigma'$ is obtained from $\sigma$ by removing repetitions
amongst the vertices. The union of these maps is a map
\[
       \Mubk_{vu}(\beta,\alpha) \to \bG^{+}_{vu}.
\]
The space $\Mubk_{vu}(\beta,\alpha)$ is given a topology by the same
procedure as in the spaces of broken trajectories (see \cite{KM-book} again).

\begin{proposition}\label{prop:1-manifold-boundary}
    For a fixed singular $1$-simplex $(v,u)$, 
   and any $\beta$, $\alpha$ for which the
    $1$-dimensional part $\Mu_{vu}(\beta,\alpha)_{1}$ is non-empty,
    the completion $\Mubk_{vu}(\beta,\alpha)_{1}$ is a compact
    1-manifold with boundary. Its boundary consists of all
    zero-dimensional products of the form \eqref{eq:Msigma} with
    $n=2$,
\[
     \Mu_{v,v^{1}}(\beta,\beta^{1}) \times \Mu_{v^{1},u}(\beta^{1},\alpha),
\]
   corresponding to singular $2$-simplices $(v,v^{1},u)$.
\end{proposition}

\begin{proof}
    This follows from the general discussion in
    section~\ref{subsec:families}. There it is explained that the
    compactified moduli space over the family has three types of
    boundary points, described as
    \ref{item:boundary-a}--\ref{item:boundary-c} on
    page~\pageref{item:boundary-a}. The three case described there
    corrrespond to the cases:
\begin{enumerate}
\item
    the case $v > v^{1} > u$ (i.e. a face of
    $\bG^{+}_{vu}$ arsing from a non-degenerate $2$-simplex);
\item
    the case $v^{1}=v$, corresponding to a singular $2$-simplex;
\item
    the case $v^{1}=u$, which is also  a singular $2$-simplex, and
    which may coincide with the previous case if $v=u$, i.e. if the
    original $1$-simplex is singular.
\end{enumerate}
\end{proof}

We now consider orientations in the context of the proposition
above. For this purpose, let us fix $I$-orientations $\mu_{vu}$ for
all cobordism $\Tigma_{vu}$ satisfying the conclusion of
Lemma~\ref{lem:orientation-consistency}, and let us fix also
orientations for all $G_{vu}$ satisfying the conclusion of
Lemma~\ref{lem:G-orient}.  If we are then given elements of
$\Lambda(\alpha$ and $\Lambda(\beta)$, we may orient
$M_{vu}(\beta,\alpha)$ using the fiber-first convention as in
section~\ref{subsec:families}.  Having oriented $M_{vu}(\beta,\alpha)$
we then orient $\Mu_{vu}(\beta,\alpha)$ as a quotient: giving $\R$ its
standard orientation, and putting it first, we write
\begin{equation}\label{eq:orient-Mu}
     M_{vu}(\beta,\alpha) = \R\times \Mu_{vu}(\beta,\alpha)
\end{equation}
as oriented manifolds. Note that there is another way to orient
$\Mu_{vu}$ that is different from this one: we could orient $\bG_{vu}$
(as we have done) as the quotient of $G_{vu}$, and then orient
$\Mu_{vu}$ as a parametrized moduli space over $\bG_{vu}$. The
difference between these two orientations is a sign
\begin{equation}\label{eq:Mu-sign-fudge}
   (-1)^{\dim G_{vu}-1} = (-1)^{|v-u|_{1}-1}.
\end{equation}
We shall always use the first orientation \eqref{eq:orient-Mu} for $\Mu_{vu}$.

Having so oriented our moduli space $\Mu_{vu}(\beta,\alpha)$, we
obtain a group homomorphism
\[
        \brm_{vu} : C_{v} \to C_{u}
\]
by counting points in zero-dimensional moduli spaces, as in
section~\ref{subsec:families}.

\begin{lemma}\label{lem:M-boundary-sign}
     For every singular $1$-simplex $(w,u)$ we have
\[
         \sum_{v} (-1)^{|v-u|_{1} (|w-v|_{1}-1) + 1} \brm_{vu}\circ
         \brm_{wv}
                    =0.
\]
   where the sum is over all $v$ with $w\ge v\ge u$.
\end{lemma}

\begin{proof}
    There is a degenerate case of this lemma, when $w=u$. In this
    case, $\brm_{ww}$ is the Floer differential $d_{w}$ on $C_{w}$, and
    the lemma states that $-d_{w}^{2}=0$. For the non-degenerate case,
    where $w>u$, the sum involves two special terms, namely the terms
    where $v=w$ and $v=u$. Extracting these terms separately, we can
    recast the formula as
\begin{multline*}
                 \Bigl(  \sum_{v\ne w,u} (-1)^{|v-u|_{1} (|w-v|_{1}-1) + 1} \brm_{vu}\circ
         \brm_{wv}\Bigr)
             \\        +
           (-1)^{|w-u|_{1} + 1}    \brm_{wu} \circ d_{w}  - d_{u}\circ \brm_{wu}
                    =0,
\end{multline*}
or equivalently
\begin{multline}
    \Bigl( \sum_{v\ne w,u} (-1)^{(\dim \bG_{vu}+1) \dim\bG_{wv} + 1}
    \brm_{vu}\circ \brm_{wv}\Bigr) \\ + (-1)^{\dim\bG_{wu}} \brm_{wu}
    \circ d_{w} - d_{u}\circ \brm_{wu} =0.
\end{multline}
This formula is simply a special case of the general chain-homotopy
formula \eqref{eq:chain-homotopy-faces}, but to verify this we need to
compare the signs here to those in \eqref{eq:chain-homotopy-faces}. To
make this comparison, let us first rewrite the last formula in terms
of the homomorphism $\bar{m}_{vu}$, defined in the same way as
$\brm_{vu}$ but using the orientation convention of
section~\ref{subsec:families}, so that 
\[
        \brm_{vu} =  (-1)^{\dim \bG_{vu}} \bar{m}_{vu}
\]
as in \eqref{eq:Mu-sign-fudge}. After multiplying throughout by
$(-1)^{\dim\bG_{wu}}$, the formula becomes
\begin{multline*}
    \Bigl( \sum_{v\ne w,u} (-1)^{(\dim \bG_{vu} +1)\dim\bG_{wv}}
    \bar{m}_{vu}\circ \bar{m}_{wv}\Bigr) \\
     +  (-1)^{\dim\bG_{wu}}  \bar{m}_{wu}
    \circ d_{w}  -d_{u}\circ \bar{m}_{wu} =0.
\end{multline*}
In this form, the formula resembles the formula
\eqref{eq:chain-homotopy-faces}, with the only difference being the
extra $+1$ in the first factor of the first exponent. This extra term
is accounted for by Lemma~\ref{lem:bG-sign} and is present because the
product orientation on  $\bG_{wv}\times \bG_{vu}$ is not equal
to its boundary orientation as a face of $\bG_{wu}$.
\end{proof}

In order to define away some of the signs, we observe that the formula
in the lemma above can be written as
\[
 (-1)^{\sum w_{i} + \msign(w,u)} \sum_{v} (-1)^{\msign(v,u) +
   \msign(w,v)}\brm_{vu}\circ
         \brm_{wv} = 0,
\]
where $\msign$ is given by the formula
\begin{equation}\label{eq:m-formula}
         \msign(v,u) =  \frac{1}{2}|v-u|(|v-u|-1) + \sum v_{i}.
\end{equation}
With our choices of homology orientations $\mu_{vu}$ etc.~still
understood, we make the following definition:

\begin{definition}\label{def:f}
    In the above setting, we define homomorphisms
\[
         f_{vu}: C_{v}\to C_{u}
\]
by the formula
\[
              f_{vu} = (-1)^{\msign(v,u)} \brm_{vu}
\]
for all singular $1$-simplices $(v,u)$. \CloseDef
\end{definition}

Note that in the case $v=u$ we have
\[
   f_{vv} = (-1)^{\sum_{i}v_{i}} d_{v}.
\]
With these built-in sign adjustments, the previous lemma takes the
following form.

\begin{proposition}
    For any singular $1$-simplex $(w,u)$, we have
\[
            \sum_{v} f_{vu}f_{wv}=0,
\]
where the sum is over all $v$ with $w\ge v\ge u$. \qed
\end{proposition}

For each singular $1$-simplex $(v,u)$, we now introduce
\[
     \bC[vu] =  \bigoplus_{  v\ge v' \ge u} C_{v'}.
\]
(Here $v'$ runs over the unit cube with extreme vertices $v$ and $u$.)
We similarly define
\[
       \bF[vu] = \bigoplus_{v\ge v' \ge u' \ge u} f_{v'u'},
\]
so we have
\[
            \bF[vu]:\bC[vu] \to \bC[vu].
\]
The previous proposition then becomes the statement
\[
             \bF[vu] ^{2} = 0,
\]
so that $(\bC[vu], \bF[vu] )$ is a complex. The basic case here is to
take $v=(1,\dots,1)$ and $u=(0,\dots,0)$, in which case $\bC[vu]$
becomes
\begin{equation}\label{eq:standard-C}
           \bC := \bigoplus_{v'\in \{0,1\}^{N}} C_{v'},
\end{equation}
with one summand for each vertex of the unit $N$-cube. The
corresponding differential $\bF$ has a summand $f_{v'v'} : C_{v'}\to
C_{v'}$ for each vertex, together with summands $f_{v'u'}$ for each
$v' > u'$. The general $\bC[vu]$ is also a sum of terms indexed by
$v'$ running over the vertices of a cube of side-length $1$, though
the dimension of the cube is $|v-u|_{1}$ in general, which may be less
than $N$.

The remainder of this subsection and the following one are devoted to
proving the following theorem, which states that the homology of the
cube $\bC[vu]$ coincides with the homology of a single $C_{w}$ for
appropriate $w$:

\begin{theorem}\label{thm:big-cube}
    Let $(v,u)$ be any $1$-simplex, and let $w = 2v - u$. Then there
    is a chain map
\[
          (C_{w}, f_{ww}) \to (\bC[vu], \bF[vu])
\]
   inducing an isomorphism in homology. Thus in the case that
   $v=(1,\dots,1)$ and $u=(0,\dots,0)$, the homology of $(\bC, \bF)$,
   where $\bC$ is as in \eqref{eq:standard-C}, is isomorphic to the
   homology of $(C_{w}, f_{w})$, where $w=(2,\dots,2)$.
\end{theorem}

To amplify the statement of the theorem a little, we can point out
first that in the case $v=u$, the result is a tautology, for both of
the chain complexes then reduce to $(C_{v}, f_{vv})$. Next, we can
look at the case $|v-u|_{1}=1$. In this case, $K_{v}$ and $K_{u}$
differ only inside one of the $N$ balls, and we may as well take
$N=1$. In the notation of Figure~\ref{fig:Tetrahedra-skein}, we can
identify $K_{v}$ and $K_{u}$ with $K_{1}$ and $K_{0}$, in which case
$K_{w}$ is $K_{2}$. The complex $\bC[vu]$ is the sum of the chain
complexes for the two links $K_{1}$ and $K_{0}$,
\[
          C_{1} \oplus C_{0}
\]
and
\[
         \bF[vu]= \begin{pmatrix}  -d_{1} & 0 \\
                                     f_{10}& d_{0} \end{pmatrix}
\]
where $f_{10}$ is minus the chain map induced by $\Tigma_{10}$. Thus
 $(\bC[vu], \bF[vu] )$ is, up to sign, the mapping cone of the chain
map induced by the cobordism. The theorem then expresses the fact that
the homologies of $C_{2}$, $C_{1}$ and $C_{0}$ are related by a long
exact sequence, in which one of the maps arises from the
cobordism $\Tigma_{10}$. The proof of the theorem will also show that
the remaining maps in the long exact sequence can be taken to be the
ones arising from $\Tigma_{21}$ and $\Tigma_{02}$.
This is the long exact sequence of the unoriented skein
relations, mentioned in the introduction. For
larger values of $|v-u|_{1}$, the differential $\bF[vu]$ still has a
lower triangular form, reflecting the fact that there is a filtration
of the cube (by the sum of the coordinates) that is preserved by the
differential.

\begin{corollary}\label{cor:spectral-sequence}
In the situation of the theorem above, 
there is a spectral sequence whose $E_{1}$ term is
\[
          \bigoplus_{v' \in \{0,1\}^{N}} 
                I^{\omega}_{*}(Y, K_{v'})
\]
and which abuts to the instanton Floer homology $I^{\omega}_{*}(Y,
K_{w})$, for $w=(2,\dots,2)$.  \qed
\end{corollary}

The signs of the maps in this spectral sequence are determined by
choices of $I$-orientations for the cobordisms $\Tigma_{vu}$ and
orientations of the families of metrics $G_{vu}$, subject to the
compatibility conditions imposed by
Lemmas~\ref{lem:orientation-consistency} and~\ref{lem:G-orient}. These
compatibility conditions still leave some freedom. For both lemmas,
the compatibility conditions mean that
the orientations for all $1$-simplices $(v,u)$ are determined by the
orientations of $\Tigma_{vu}$ and $G_{vu}$
for the \emph{edges}, i.e. the $1$-simplices $(v,u)$ for which
$|v-u|_{1}=1$. These are the simplices that will contribute to the
differential $d_{1}$ in the spectral sequence above. For these, the
space $\breve{G}_{vu}$ is a point, so an orientation is determined by
a sign $(-1)^{\delta(v,u)}$ for each $v,u$. 
The condition of Lemma~\ref{lem:G-orient}
is equivalent to requiring the following: for every $2$-dimensional
face of the cube, with diagonally opposite vertices $w> u$ and
intermediate vertices $v$ and $v'$, we need
\begin{equation}
    \label{eq:square-anti-commutes}
     \delta(w,v) \delta(v,u) = 1+\delta(w,v')\delta(v',u) \pmod{2}.
\end{equation}
We can achieve this condition by an explicit choice of sign, such as
\begin{equation}
    \label{eq:delta-sign-that-works}
      \delta(v,u) = \sum_{i=0}^{i_{0}-1} v_{i},
\end{equation}
where $i_{0}$ is the unique index at which $v_{i}$ and $u_{i}$
differ. 

\begin{corollary}\label{cor:spectral-sequence-signs}
 Let $I$-orientations be chosen for the cobordisms $\Tigma_{v'u'}$ so
 that the conditions of Lemma~\ref{lem:orientation-consistency}
 hold. For each edge $(v', u')$ of the cube, let
 $I^{\omega}(S_{v'u'})$ 
be the map
 $I^{\omega}(Y,K_{v'}) \to I^{\omega}(Y, K_{u'})$ induced by
 $\Tigma_{v'u'}$ with this $I$-orientation. Then
 the spectral sequence in the previous corollary can be set up so that 
  the differential $d_{1}$ is the sum of the
maps
\[
      (-1)^{\tilde \delta(v',u')} I^{\omega}(S_{v'u'})
\]
  over all edges of the cube, where
\[
      \tilde\delta(v', u') = \sum_{i=i_{0}}^{N} v'_{i}
\]
and $i_{0}$ is the index at which $v'$ and $u'$ differ.
\end{corollary}

\begin{proof}
    The difference between $\delta$ and the $\tilde\delta$ that
    appears here is $\sum_{i}{v'_{i}}$ mod 2, 
   which is $\msign(v',u')$ mod 2
    in the case of an edge $(v',u')$.
\end{proof}

Theorem~\ref{thm:big-cube} will be proved by an inductive argument. 
The special case described at the end of the theorem is equivalent to
the general case, so we may as well take
\[
\begin{aligned}
    w &= (2,\dots, 2) \\
    v &= (1,\dots, 1) \\
    u &= (0,\dots,0) .
\end{aligned}
\]
Thus $v$ and $u$ span an $N$-cube. We again write $(\bC, \bF)$ for
$(\bC[uv], \bF[uv])$ in this context.
For each $i\in\Z$, let us set
\[
        \mathbf{C}_{i} =  \bigoplus_{v'\in \{1,0\}^{N-1}} C_{v',i}.
\]
Each $\bC_{i}$ can be described as $\bC[v'u']$ for some $v', u'$ with
$|v' - u'|_{1}=N-1$.
We have
\[
        \bC = \bC_{1} \oplus \bC_{0},
\]         
generalizing the case $N=1$ considered above, and we can similarly
decompose $\bF$ in block form as
\[
       \bF =
       \begin{pmatrix}
           \bF_{11} & 0 \\
           \bF_{10} & \bF_{00}
       \end{pmatrix},
\]
where
\[
       \bF_{ij} = \bigoplus_{v',u'\in \{1,0\}^{N-1}} f_{(v'i)(u'j)}.
\]
To prove the theorem, we will establish:

\begin{proposition}\label{prop:blocks}
    There is a chain map
\[
         (\bC_{2}, \bF_{22}) \to (\bC, \bF)
\]
inducing isomorphisms in homology.
\end{proposition}

This proposition expresses an isomorphism between the homologies of
two cubes, one of dimension $N-1$, the other of dimension
$N$. Theorem~\ref{thm:big-cube} for the given $w$, $v$ and $u$ is an
immediate consequence of $N$ applications of this proposition. Just as
in the case $N=1$, the proposition will be proved while establishing
that there is a long exact sequence in homology arising from the anti-chain
maps, 
\[
            \cdots \longrightarrow ( \bC_{3} ,\bF_{33}) \stackrel{\bF_{32}} \longrightarrow       
           ( \mathbf{C}_{2}, \bF_{22}) 
        \stackrel{\bF_{21}} \longrightarrow 
               ( \mathbf{C}_{1}, \bF_{11})
 \stackrel{\bF_{10}} \longrightarrow 
               ( \mathbf{C}_{0}, \bF_{00}) 
   \longrightarrow \cdots
\]
(In the above sequence, the chain groups and anti-chain maps are
periodic mod 3, up to sign.)

\section{Proof of Proposition~\ref{prop:blocks}}
\label{sec:proof-prop-blocks}

\subsection{The algebraic setup}

The proof is based on an algebraic lemma which appears (in a mod $2$
version) as Lemma~4.2 in \cite{OS-double-covers}. We omit the proof of
the lemma:

\begin{lemma}[\protect{\cite[Lemma~4.2]{OS-double-covers}}]
    Suppose that for each $i\in \Z$ we have a complex $(C_{i}, d_{i})$
    and anti-chain maps
     \[
                 f_{i} : C_{i} \to C_{i-1}.
      \]
     Suppose that the composite chain map $f_{i-1} \circ
     f_{i}$ is chain-homotopic to $0$ via a chain-homotopy
     $j_{i}$, in that
     \[
                  d_{i-1} j_{i} + j_{i} d_{i} + f_{i-1} 
                   f_{i} =0
      \]
       for all $i$. Suppose furthermore that for all $i$, the map
       \[
                       j_{i-1}f_{i} + f_{i-2} j_{i} :
                       C_{i} \to C_{i-3}
      \]
       (which is a chain map under the hypotheses so far) induces
       an isomorphism in homology.  Then the induced maps in homology,
       \[
              (f_{i})_{*} : H_{*}(C_{i},d_{i}) \to H_{*}(C_{i-1},d_{i-1})
        \]
        form an exact sequence; and for each $i$ the anti-chain map
        \[
        \begin{aligned}
            \Phi: s &\mapsto (f_{i} s, j_{i} s) \\
            \Phi: C_{i} &\to \mathrm{Cone}(f_{i})
        \end{aligned}
          \]
        induces isomorphisms in homology. Here
        $\mathrm{Cone}(f_{i})$ denotes $C_{i}\oplus C_{i-1}$
        equipped with the differential
\[
\begin{pmatrix}
    d_{i} & 0 \\
    f_{i} & d_{i-1}
\end{pmatrix}.
\]
\qed
\end{lemma}

Given the lemma, our first task is to construct (in the case $i=2$,
for example) a map
\[
           \bJ_{20} : \bC_{2} \to \bC_{0}
\]
satisfying
\begin{equation}\label{eq:first-homotopy}
          \bF_{00} \bJ_{20}  +  \bJ_{20}\bF_{22}  + \bF_{10}\bF_{21} = 0.
\end{equation}
After constructing these maps, we will then need to show that the
maps such as
\[
         \bF_{10}\bJ_{31} + \bJ_{20}\bF_{32} : \bC_{3} \to \bC_{0}
\]
(which is a chain map) induce isomorphisms in homology. This second
step will be achieved by constructing a chain-homotopy
\[
         \bK_{30} :  \bC_{3} \to \bC_{0}
\]
with
\begin{equation}\label{eq:second-homotopy}
              \bF_{00}\bK_{30} + \bK_{30}\bF_{33} + 
                   \bF_{10}\bJ_{31} + \bJ_{20}\bF_{32} + \tilde{\mathbf{Id}} = 0,
\end{equation}
where $\tilde{\mathbf{Id}}$ is a chain-map that is chain-homotopic to
$\pm 1$. The construction of $\bJ$ and $\bK$ and the
verification of the chain-homotopy formulae \eqref{eq:first-homotopy}
and \eqref{eq:second-homotopy} occupy the remaining two subsections of
this section of the paper.

\subsection{\texorpdfstring{Construction of $\bJ$}{Construction of J}}

We will construct $\bJ_{i,i-2}$ for all $i$ so that (in the case
$i=2$, for example) the relation \eqref{eq:first-homotopy} holds.
Recalling the definition of $\bC_{i}$, we see that we should write
\begin{equation}\label{eq:bJ-pieces}
       \bJ_{20} = \sum j_{(v'2)(u'0)} 
\end{equation}
where the sum is over all $v' \ge u'$ in $\{0,1\}^{N-1}$. The desired
relation then expands as the condition that, for all $w'$, $u'$ in $\{0,1\}^{N-1}$,
\begin{equation}\label{eq:first-chain-homotopy-pieces}
      \sum_{v'}\Bigl(  f_{(v'0)(u'0)} j_{(w'2)(v'0)} + j_{(v'2)(u'0)}
              f_{(w'2)(v'2)} +  f_{(v'1)(u'0)} f_{(w'2)(v'1)} \Bigr) = 0. 
\end{equation}
Our task now is to define $j_{vu}$ for $v=v'2$ and $u=u'0$ in
$\Z^{N}$, where
$(v',u')$ some singular $1$-simplex in the
triangulation of $\DDelta^{N-1}$ of $\R^{N-1}$. 

We have previously considered $I$-orientations for $\Tigma_{vu}$ and
orientations of families of metrics $G_{vu}$ in the case that $(v,u)$
is a singular $1$-simplex. We now extend our constructions to the case
of an arbitrary pair $(v,u)$ with $v\ge u$. (So we now allow
$|v-u|_{\infty}$ to be larger than $1$. At present we are most
interested in the case $|v-u|_{\infty}=2$; and in the next section,
$3$ will be relevant.) We still have natural
cobordisms $\Tigma_{vu}$ when $|v-u|_{1}>1$, obtained by concatenating
the cobordisms we used previously. So for example, when $N=1$, the
cobordism $\Tigma_{20}$ is the composite of the cobordisms
$\Tigma_{21}$ and $\Tigma_{10}$.

 For the $I$-orientations, we
can begin by choosing $I$-orientations as before for
$1$-simplices $(v,u)$, so that the conditions of
Lemma~\ref{lem:orientation-consistency} hold. Then we simply extend to all
pairs $v\ge u$ so that the consistency condition
\[
 \mu_{wu}= \mu_{vu} \comp \mu_{wv}
\]
holds for all $w\ge v \ge u$.  

For arbitrary $v\ge u$, the cobordism
$(\R\times Y, \Tigma_{vu})$ 
also carries a family of metrics $G'_{vu}$ of dimension
$n = |v-u|_{1}$. We define $G'_{vu}$ first in the case $N=1$. In this
case, we can regard $\Tigma_{vu}$ as the composite of $n$ cobordisms,
where $n=|v-u|_{1}$ and each cobordism is a surface on which the $t$
coordinate as a single critical point. Unlike the previous setup,
these critical points cannot be re-ordered, as they all lie in the same
copy of $\R\times B_{1}$ (where $B_{1}$ is a $3$-ball), rather than in
distinct copies $\R\times B_{i}$. As an appropriate parameter space in
this case, we define $G'_{vu}$
to be
\begin{equation}\label{eq:G-primed-def}
  G'_{vu} =  \{\, (\tau_{1},\dots \tau_{n}) \in \R^{n} \mid \tau_{m+1} \ge 1+
   \tau_{m},\; \forall m<n \,\},
\end{equation}
and construct the metrics so that $\tau_{m}$ is the $t$ coordinate of
the critical point in the $m$'th cobordism (much as we did in the
construction of $G_{vu}$ earlier, on page~\pageref{page:Gvu-setup}). For
larger $N$, we construct $G'_{vu}$ as a product of $G'_{v_{i}u_{i}}$
over all $i=1,\dots ,N$. In the case that $|v-u|_{\infty}=1$, the
space $G'_{vu}$ coincides with $G_{vu}= \R^{n}$ as defined before;
while if $|v_{N}-u_{N}|=2$ and
$|v_{j}-u_{j}|\le 1$ for $j\le N-1$, then $G'_{vu}$ is a half-space. 
Whenever $w \ge v
\ge u$, we have self-evident maps
\[
        G'_{wu} \to G'_{wv} \times G'_{vu}.
\]
These maps are either surjective, or have image equal to the
intersection of the codomain with a product of half-spaces. Following
Lemma~\ref{lem:G-orient}, we can choose orientations for all the
$G'_{vu}$ so these maps are always orientation-preserving.

Next we examine the topology of the cobordism $\Tigma_{vu}$ in the
case that $N=1$ and $|v-u|_{\infty} = 2$. For the following
description, we revert to considering $S_{vu}$ as a compact surface in
a product $I\times Y$, rather than a surface with cylindrical ends in
$\R\times Y$.

\begin{lemma}\label{lem:composite-RP2}
    In the situation depicted in Figures~\ref{fig:Tetrahedra-skein}
    and~\ref{fig:Tetrahedron-2}, the composite cobordism
    $\Tigma_{2,0} = \Tigma_{1,0}\circ\Tigma_{2,1}$ from $K_{2}$  to
    $K_{0}$ in $I\times Y$ has the form
\[
            ( I \times Y , V_{2,0} ) \# (S^{4}, \RP^{2})
\]
    where the $\RP^{2}$ is standardly embedded in $S^{4}$ with
    self-intersection $+2$, as described in
    section~\ref{subsec:RP2-examples}.
\end{lemma}

\begin{remark}
    The cobordism $V_{2,0}$ that appears in the above lemma is
    diffeomorphic to  $\Tigma_{3,2}$,
    viewed as a cobordism from $K_{2}$ to $K_{3}$ by reversing the
    orientation of $I\times Y$.
\end{remark}

\begin{proof}[Proof of the Lemma]
      Arrange the composite cobordism $\Tigma_{2,0}$ so that the $t$
      coordinate runs from $0$ to $1$ across $\Tigma_{2,1}$ and from
      $1$ to $2$ across $\Tigma_{1,0}$. The projection of
      $\Tigma_{2,1}$ to $Y$ meets the ball $B^{3}$ in the twisted
      rectangle $T=T_{2,1}$ depicted in
      Figure~\ref{fig:Tetrahedron-2}, while the projection of
      $\Tigma_{1,0}$ similarly meets $B^{3}$ in a twisted rectangle
      $T_{1,0}$. The intersection $T_{2,1}\cap T_{1,0}$ in $B^{3}$ is
      a closed arc $\delta$, joining two points of $K_{1}$. The
      preimages of $\delta$ in $\Tigma_{2,1}$ and $\Tigma_{1,0}$ are two
      arcs in $\Tigma_{2,0}$ whose union is a simple closed curve
\[
               \gamma\subset \Tigma_{2,0}.
\] 
On $\gamma$, the $t$ coordinate takes values in $[1/2,3/2]$.
 A regular
      neighborhood of $[0,2]\times \delta$ in $[0,2]\times Y$ is a
      $4$-ball meeting $\Tigma_{2,0}$ in a M\"obius band: the band
      is the neighborhood of $\gamma$ in $\Tigma_{2,0}$. 

      This M\"obius band in the $4$-ball can be seen as arising from
      pushing into the ball
      an unknotted M\"obius band in the $3$-sphere.  The M\"obius band
      $M$ in the $3$-sphere is the union of three pieces:
      \begin{enumerate}
      \item a neighborhood of $\{0\} \times \delta$ in $\{0\}\times
          T_{2,1}$;
      \item a neighborhood of $\{2\} \times \delta$ in $\{2\}\times
          T_{1,0}$;
      \item the pair of rectangles $[0,2]\times \epsilon$, where
          $\epsilon$ is a pair of arcs, one in each component of
          $K_{1}\cap B^{3}$.    
\end{enumerate}
       This M\"obius band $M$
      possesses a left-handed
      half-twist. The half-twist is the result of two quarter-turns,
      one in each of the first two pieces of $M$ in the list above.
      The signs of the quarter-turns can be seen in 
     Figure~\ref{fig:Tetrahedron-3}:
\begin{figure}
    \begin{center}
        \includegraphics[scale=0.5]{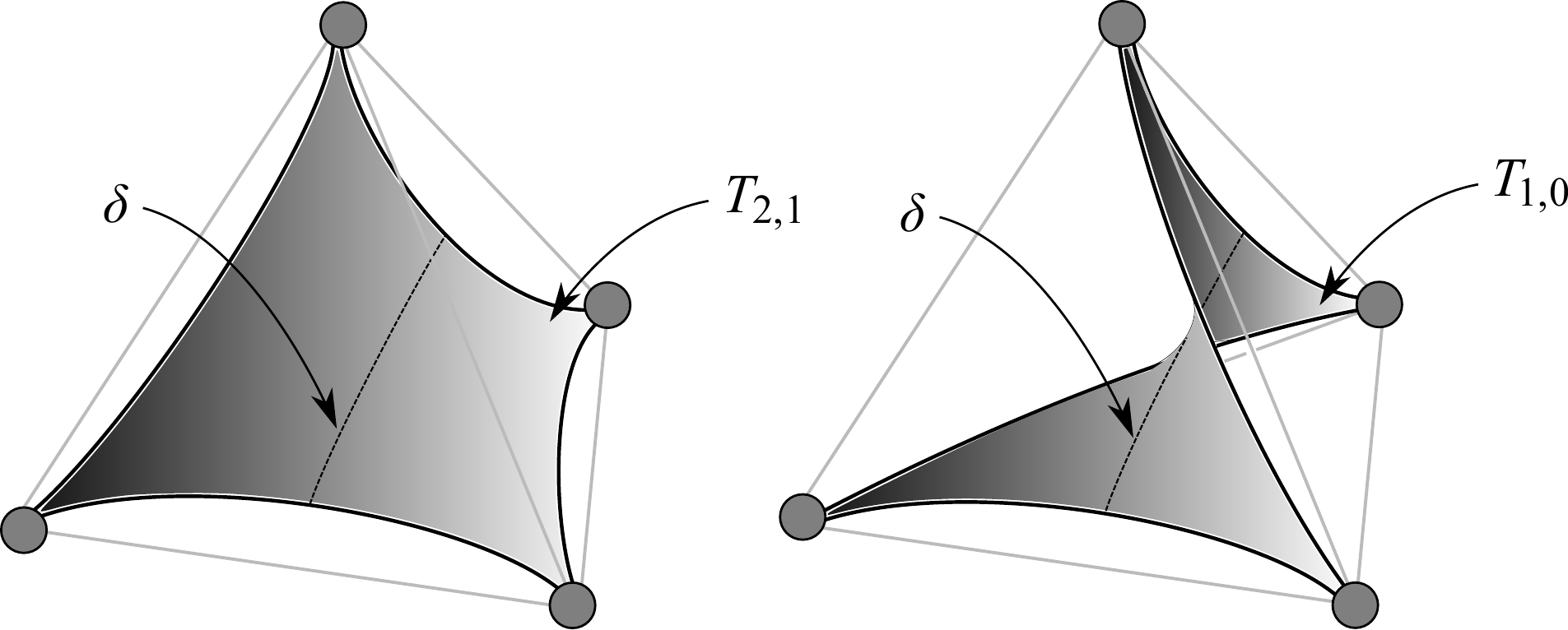}
    \end{center}
    \caption{\label{fig:Tetrahedron-3}
   The arc $\delta$ as the  intersection of $T_{2,1}$ and $T_{1,0}$.}
\end{figure}
        a neighborhood of $\delta$ in $T_{2,1}$ has a \emph{right}-hand
       quarter turn for the standard orientation of $B^{3}$, but this
       $3$-ball occurs in the boundary of the $4$-ball with its opposite
       orientation; and  a neighborhood of $\delta$ in $T_{1,0}$ has a \emph{left}-hand
       quarter turn for the standard orientation of $B^{3}$, and this
       $3$-ball occurs  with its positive orientation in the boundary
       of the $4$-ball. Thus the M\"obius band in the boundary of the
       $4$-ball has a left-hand half-twist resulting from two
       left-handed quarter-turns.
 
      The $\RP^{2}$ obtained from M\"obius band
      with a left-handed half-twist is the standard $\RP^{2}$ with
      self-intersection $+2$.
\end{proof}

The fact that the composite cobordism $\Tigma_{20}$ from $(Y,K_{2})$ to $(Y,K_{0})$
splits off a summand $(S^{4}, \RP^{2})$ (as stated in the lemma above)
implies, by standard stretching arguments, that this composite
cobordism induces the zero map in homology:
\[
  (f_{10}\comp f_{21})_{*} = 0 : I^{\omega}(Y, K_{2}) \to I^{\omega}(Y, K_{0}).
\]
This is essentially the same point as the vanishing theorem for the
Donaldson invariants of connected sums. It is important here that
the summand $(S^{4}, \RP^{2})$ carries no reducible solutions, which
might live in moduli spaces for which the index of $\calD$ is
negative: see the examples of moduli spaces in
Proposition~\ref{prop:RP2-example-prop}.
Although it is zero at the level of homology, at the \emph{chain
  level}, 
the map induced by the composite may be non-zero. The
one-parameter family of metrics involved in the stretching provides a
chain homotopy, showing that the map is chain-homotopic to zero. This
is what will be used to construct the chain-homotopy $\bJ_{20}$
in Definition~\ref{def:jvu-def} below.
But
first, we must make the family of metrics explicit and extend our
notation to the case of more than one ball in $Y$.

Staying for a moment with the case of one ball, we have already set up a
family of metrics $G'_{20}$, which in this case is a $2$-dimensional
half-space:
\[
  G'_{20} = \{\, (\tau_{1},\tau_{2}) \mid \tau_{2} \ge 1 + \tau_{1} \,\}.
\]
We extend this family of metrics to a family
\begin{equation}\label{eq:G-extended}
         G_{20} = G'_{20} \cup G''_{20}
\end{equation}
as follows. The boundary of $G'_{20}$ consists of the family of
metrics with $\tau_{2}=\tau_{1}+1$, all of which are isometric to each
other, by translation of the coordinates. Fixing any one of these, say
at $\tau_{1}=0$, we construct a family of metrics parametrized by the
negative half-line $\R^{-}$, by stretching along the sphere
$S^{3}$  which splits off the summand $(S^{4},\RP^{2})$ in
Lemma~\ref{lem:composite-RP2}. (This family of metrics can be completed
to a family parametrized by  $\R^{-}\cup \{-\infty\}$, where the added
point is a broken metric, cut along this $S^{3}$.)
Putting back the translation parameter,
we obtain our family of metrics $G''_{20}$ parametrized by
$\R^{-}\times \R$. The space $G_{20}$ is the union of these two
half-spaces, along their common boundary.

Suppose now that $N$ is arbitrary, and that $v_{N} - u_{N} =2$ and
$v_{j}- u_{j} = 0$ or $1$ for $j<N$. The space of metrics $G'_{vu}$ is
a half-space: it is a product
\[
             G'_{vu} =  \R^{m-1} \times G'_{20}
\]
where $m$ is the number of coordinates in which $v$ and $u$ differ,
and $G'_{20}$ is a 2-dimensional half-space as above. The coordinates
on $\R^{m-1}$ are the locations $\tau$ of the critical points in the
balls $B_{i}$ corresponding to coordinates $i<N$ where $v$ and $u$
differ. 
We extend this
family of metrics to a family
\begin{equation*}
    \begin{aligned}
        G_{vu} &= \R^{m-1} \times (G'_{20} \cup G''_{20}) \\
        &= G'_{vu} \cup G''_{vu} 
    \end{aligned}
\end{equation*}
where $G''_{20}$ is as before. We again use the notation $\bG_{vu}$
for the quotient by translations:
\[
      \bG_{vu}= G_{vu} /\R.
\] 

Let us consider the natural compactification $\bG^{+}_{wu}$ of
$\bG_{wu}$, where $w=w'2$, $u=u'0$ and $|w'-u'| \le 1$. This is a
family of broken Riemannian metrics whose codimension-1 faces are as follows.
\begin{enumerate}
\item \label{item:faces-a}
   First, there are the families of broken metrics which are cut
    along $(Y, K_{v})$ where $w>v>u$. This face is parametrized by
   $\bG_{wv}\times \bG_{vu}$. These faces we can classify 
    further into the cases
    \begin{enumerate}
    \item
      \label{item:face-a1}
     the case $v_{N}=0$, in which case the first factor
        $\bG_{wv}$ has the form $\bG'_{wv}\cup \bG''_{wv}$, where
        $\bG''_{wv}$ involves stretching across the $S^{3}$;
    \item the similar case $v_{N}=2$, where the second factor has the
        form $\bG'_{vu}\cup \bG''_{vu}$;
    \item 
      \label{item:face-a3}
    the case $v_{N}=1$, in which case $\bG_{wv}$ and $\bG_{vu}$
        are  both the simpler families described in the previous
        subsection leading to the construction of the maps $f_{vu}$ etc.
    \end{enumerate}
    \item  \label{item:face-b}
         Second, there is the family of broken metrics which are cut
        along the $S^{3}$.
\end{enumerate}

Now let $\beta \in \Crit_{v}$ and
$\alpha\in \Crit_{u}$ be critical points, corresponding to generators
of the complexes $C_{v}$ and $C_{u}$ respectively. The family of
metrics $G_{vu}$ gives rise to a parametrized moduli space
\[
       M_{vu}(\beta,\alpha) \to G_{vu}.
\]
Dividing out by the translations, we also obtain
\[
    \Mu_{vu}(\beta,\alpha) \to \breve{G}_{vu}.
\]
We have already oriented the subset $G'_{vu} \subset G_{vu}$, so we
have chosen orientation for $G_{vu}$. As before, we orient
$M_{vu}(\beta,\alpha)$ using our chosen $I$-orientations and a
fiber-first convention, and we orient $\Mu_{vu}(\beta,\alpha)$ as the
quotient of $M_{vu}(\beta,\alpha)$ with the $\R$ factor first. The
zero-dimensional part
\[
\Mu_{vu}(\beta,\alpha)_{0} \subset \Mu_{vu}(\beta,\alpha) 
\]
(if any) is a finite set of oriented points as usual, and we define
$j_{vu}$ by counting these points, with an overall correction factor
for the sign:

\begin{definition}\label{def:jvu-def}
    Given $v \ge u$ in $\Z^{N}$ with $v_{N}- u_{N}=2$ and $v_{j}-u_{j}
    \le 1$ for $j<N$, we define
\[
         j_{vu} : C_{v} \to C_{u}
\]
     by declaring the matrix entry from $\beta$ to $\alpha$ to be the
     signed count of the points in the zero-dimensional 
    moduli space $\Mu_{vu}(\beta,\alpha)_{0}$ (if any), adjusted by
    the overall sign $(-1)^{\msign(v,u)}$, where $\msign(v,u)$ is again defined
    by the formula \eqref{eq:m-formula}.
\CloseDef
\end{definition}

Having defined $j_{vu}$ in this way, we can now construct $  \bJ_{20}
: \bC_{2} \to \bC_{0}$ in terms of $j_{vu}$ by the formula
\eqref{eq:bJ-pieces}. We must now prove the chain-homotopy formula
\eqref{eq:first-homotopy}, or equivalently the formula
\eqref{eq:first-chain-homotopy-pieces}, which we can equivalently
write as
\begin{equation}\label{eq:long-hand-j-htpy}
      \sum_{\{ v \mid v_{N}=0\}} f_{vu} j_{wv} + \sum_{\{ v \mid
        v_{N}=2\}}  j_{vu}
              f_{wv} +   \sum_{\{ v \mid
        v_{N}=1\}}  f_{vu} f_{wv}  = 0. 
\end{equation}

As usual, the proof that this expression is zero is to interpret the
matrix entry of this map, from $\gamma$ to $\alpha$, as the number of
boundary points of an oriented $1$-manifold, in this case the manifold
$\Mu_{wu}(\gamma,\alpha)$. This is in essence an example of the
chain-homotopy formula \eqref{eq:chain-homotopy-faces}, resulting from
counting ends of one-dimensional moduli-spaces
$\Mubk_{wu}(\gamma,\alpha)_{1}$ over $\bG^{+}_{wv}$. The three types
of terms in the above formulae capture the three types of boundary
faces \ref{item:faces-a} above, together with the terms of the form
``$\partial\circ m_{G} \pm m_{G} \circ \partial$'' in
\eqref{eq:chain-homotopy-faces}. This is just the same set-up as the
proof that $\bF_{10}\circ\bF_{10}=0$ in the previous section, and our
signs are once again arranged so that all terms contribute with
positive sign.

The only remaining issue for the proof of \eqref{eq:long-hand-j-htpy}
is the question of why there is no additional term in this formula to
account for a contribution from the face
\ref{item:face-b} of $\bG^{+}_{wu}$. This face does not fall into the
general analysis, because the cut $(S^{3}, S^{1})$ does not satisfy
the non-integral condition. (We have $w=0$ on this cut.)
 Analyzing the
contribution from this type of boundary component follows the standard
approach to a connected sum -- in this case, a connected sum with the
pair $(S^{4},\RP^{2})$ along a standard $(S^{3}, S^{1})$. There is no
contribution from this type of boundary component, however, by the
usual dimension-counting argument for connected sums, because all
solutions on $(S^{4},\RP^{2})$ are irreducible and the unique critical
point for $(S^{3}, S^{1})$ is reducible.

\subsection{\texorpdfstring{Construction of $\bK$}{Construction of K}}

We turn to the construction of $\bK_{30}$ and the proof the formula
\eqref{eq:second-homotopy}. We start with a look at the  topology of
the composite cobordism \[ \Tigma_{30} = \Tigma_{10} \comp \Tigma_{21}
\comp \Tigma_{32}\] from $K_{3}$ to $K_{0}$ in the case $N=1$. 
Our discussion is very closely modeled on the exposition of
\cite[section~5.2]{KMOS}.

Arrange the $t$ coordinate on the composite cobordism $\Tigma_{30}$ so
that $t$ runs from $3-i$ to $3-j$ across $\Tigma_{ij}$. In the
previous subsection we exhibited a M\"obius band (called $M$ there)
inside $\Tigma_{20}$. Let us now call this M\"obius band $M_{20}$. It
is the intersection of $\Tigma_{20}$ with a 4-ball arising as the
regular neighborhood of $[1,3]\times \delta$. Just as we renamed $M$,
let us now write $\delta_{20}$ for the arc $\delta$.  There is a
similar M\"obius band $M_{31}$ in $\Tigma_{31}$, arising as the
intersection of $\Tigma_{31}$ with the regular neighborhood of
$[0,2]\times\delta_{31}$. 

The arcs $\delta_{20}$ and $\delta_{31}$ in $B^{3}$ both lie on the
surface $T_{21}$, where they meet at a single point at the center of
the tetrahedron. The two M\"obius bands $M_{31}$ and $M_{20}$ meet
$T_{21}$ in regular neighborhoods of these arcs; so the intersection
$M_{31}\cap M_{20}$ is a neighborhood in $T_{21}$ of this point. (See
Figure~\ref{fig:two-bands}.) The union
\[
           M_{30} := M_{31} \cup M_{20}
\]
\begin{figure}
    \begin{center}
        \includegraphics[height=1.7in]{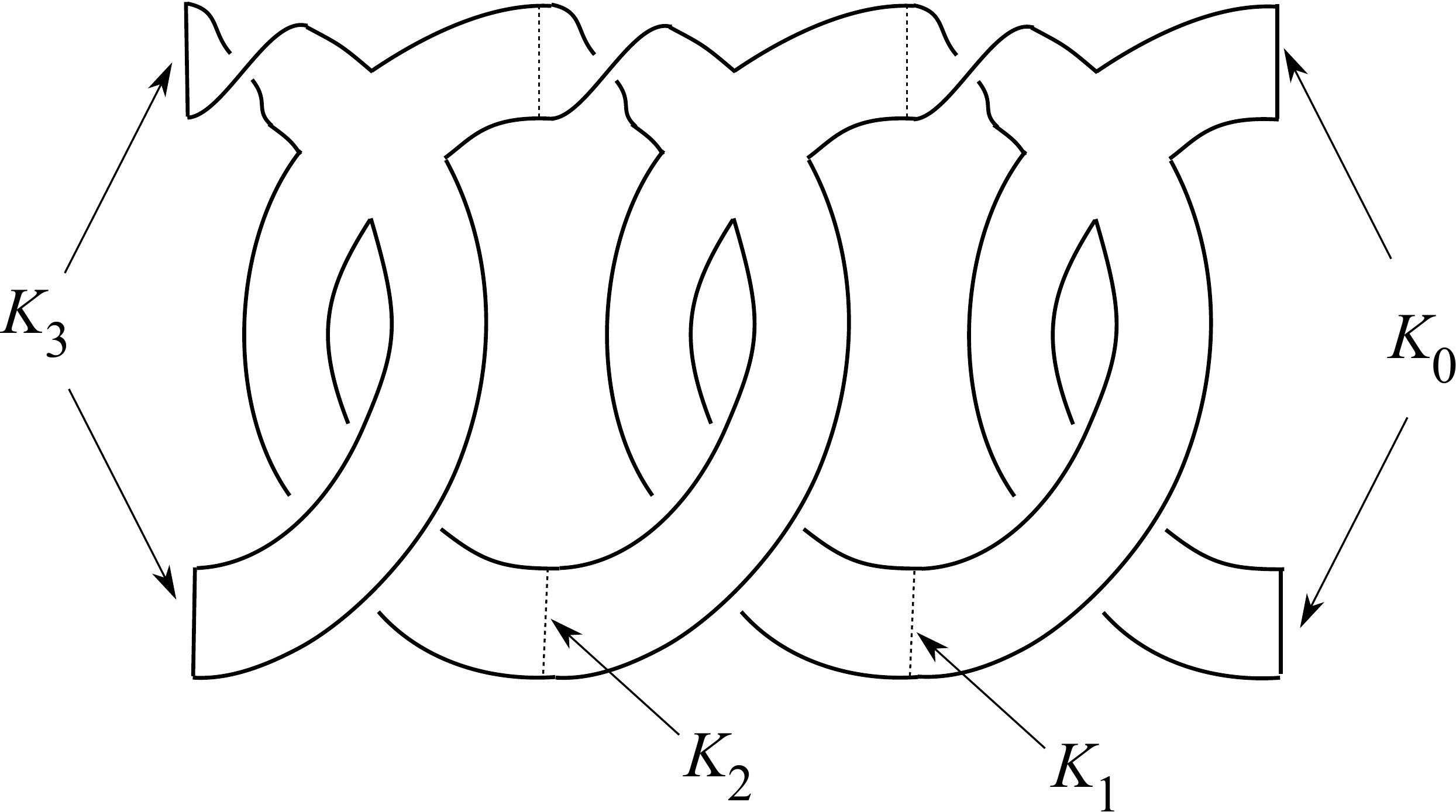}
    \end{center}
    \caption{\label{fig:two-bands}
    Two intersecting M\"obius bands, $M_{31}$ and $M_{20}$ inside $\Tigma_{30}$.}
\end{figure}
has the topology of a twice-punctured $\RP^{2}$ and it sits in a 4-ball
$B_{30}$ obtained as regular neighborhood of the union of the previous
two balls, $B_{31}$ and $B_{20}$. The $3$-sphere $\SS_{30}$ which forms
the boundary of $B_{30}$ meets $M_{30}$ in two unknotted, unlinked
circles. The following lemma helps to clarify the topology of the
cobordism $\Tigma_{30}$.

\begin{lemma} \label{lem:two-disks-product} 
If we remove $(B_{30}, M_{30})$ from the pair $([0,3]\times Y,
\Tigma_{30})$ and replace it with $(B_{30}, \Delta)$, where $\Delta$
is a union of two standard disks in the 4-ball, then the resulting
cobordism $\bar\Tigma$ from $K_{3}$ to $K_{0} = K_{3}$ is the trivial
cylindrical cobordism in $[0,3]\times Y$.
\end{lemma}

\begin{proof}
    This is clear.
\end{proof}

Altogether, we can identify five separating $3$-manifolds in $[0,3]
\times Y$, namely the three $3$-spheres $\SS_{30}$, $\SS_{31}$ and
$\SS_{20}$ obtained as the boundaries of the three balls, and the two
copies of $Y$,
\[
\begin{aligned}
    Y_{2} &= \{1\} \times Y \\
    Y_{1} &= \{2\} \times Y
\end{aligned}
\]
which contain the links $K_{2}$ and $K_{1}$. Just as in
\cite[section~5.2]{KMOS}, each of these five $3$-manifolds intersects
two of the others transversely, in an arrangement indicated schematically in
Figure~\ref{fig:5-surfaces}, and each non-empty intersection is a
2-sphere. 
\begin{figure}
    \begin{center}
        \includegraphics[height=1.9in]{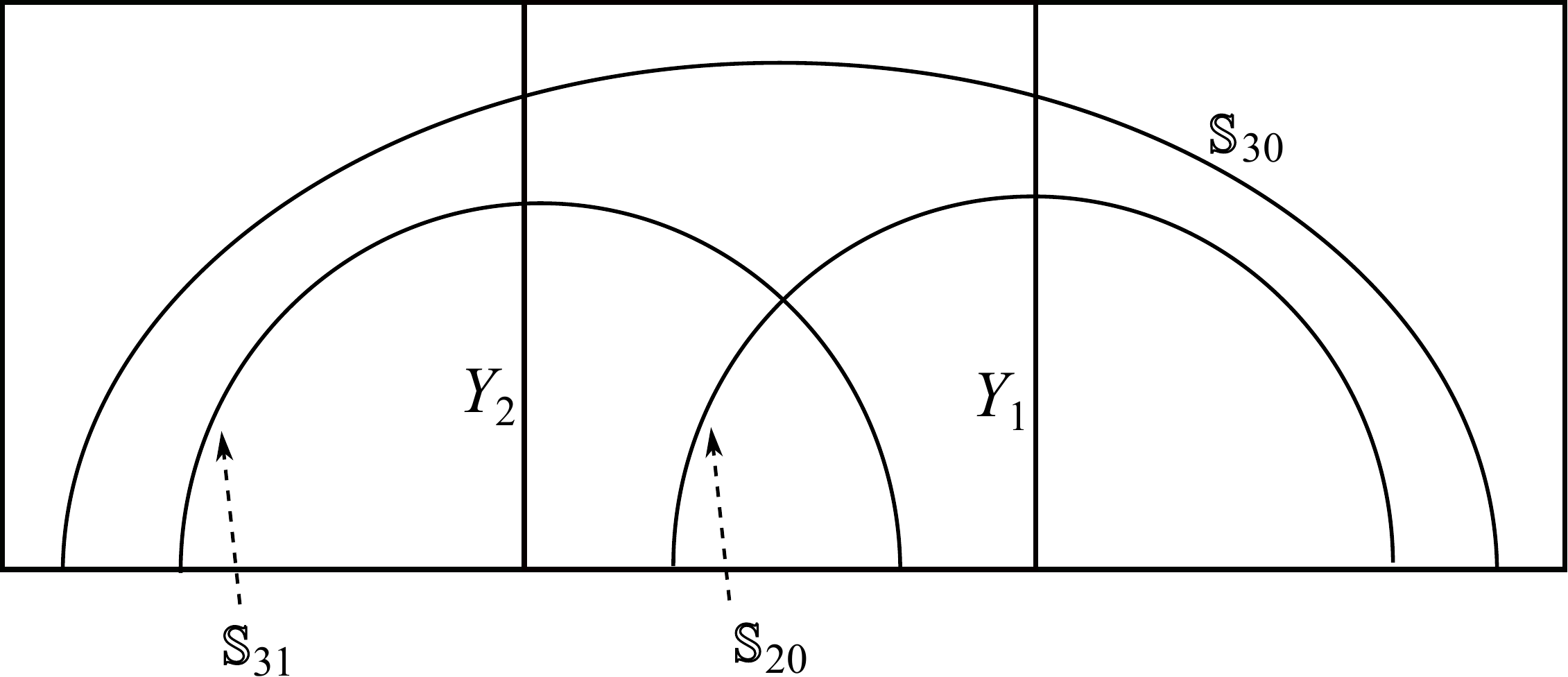}
    \end{center}
    \caption{\label{fig:5-surfaces}
    The five $3$-manifolds, $Y_{2}$, $Y_{1}$, $\SS_{30}$, $\SS_{31}$
    and $\SS_{20}$ in the composite cobordism $([0,3]\times Y,
   \Tigma_{30})$.}
\end{figure}

We can form a family of Riemannian metrics $\bG_{30}$ on this
cobordism whose compactification $\bG^{+}_{30}$ is a 2-dimensional
manifold with corners -- in fact, a pentagon -- parametrizing a
family of broken Riemannian metrics. The five edges of this pentagon
correspond to broken metrics for which the cut is a single one of the
five separating $3$-manifolds,
\[
             \SS \in \{ \SS_{30}, \SS_{31}, \SS_{20} , Y_{2}, Y_{1}\}.
\]
We denote the corresponding face by
\[
        Q(\SS) \subset \bG^{+}_{30}.
\]
The five corners of the pentagon correspond to broken metrics where
the cut has two connected components,  $\SS \cup \SS'$, where
\[
   \{ \SS, \SS' \} \subset \{ \SS_{30}, \SS_{31}, \SS_{20} , Y_{2}, Y_{1}\},
\]
is a pair of $3$-manifolds that do \emph{not} intersect. (There are
exactly five such pairs.) In the neighborhood of each edge and each
corner, the family of metrics has the model form described in
section~\ref{subsec:families}.  As special cases, we have
\[
\begin{aligned}
    Q(Y_{1}) &= \breve{G}_{31} \times \breve{G}_{10} \\
    Q(Y_{2}) &= \breve{G}_{32} \times \breve{G}_{20} .
\end{aligned}
\]

We can, if we wish, regard $\breve{G}_{30}$ as the quotient by
translations of a larger family $G_{30}$ of dimension $3$. We can
regard the previously-defined family $G'_{30} \cong
\R^{+}\times\R^{+}$ as a subset of $G_{30}$ in such a way that its
image in $\bG_{30}$ is the indicated quadrilateral in
Figure~\ref{fig:pentagon}.
\begin{figure}
    \begin{center}
        \includegraphics[height=2in]{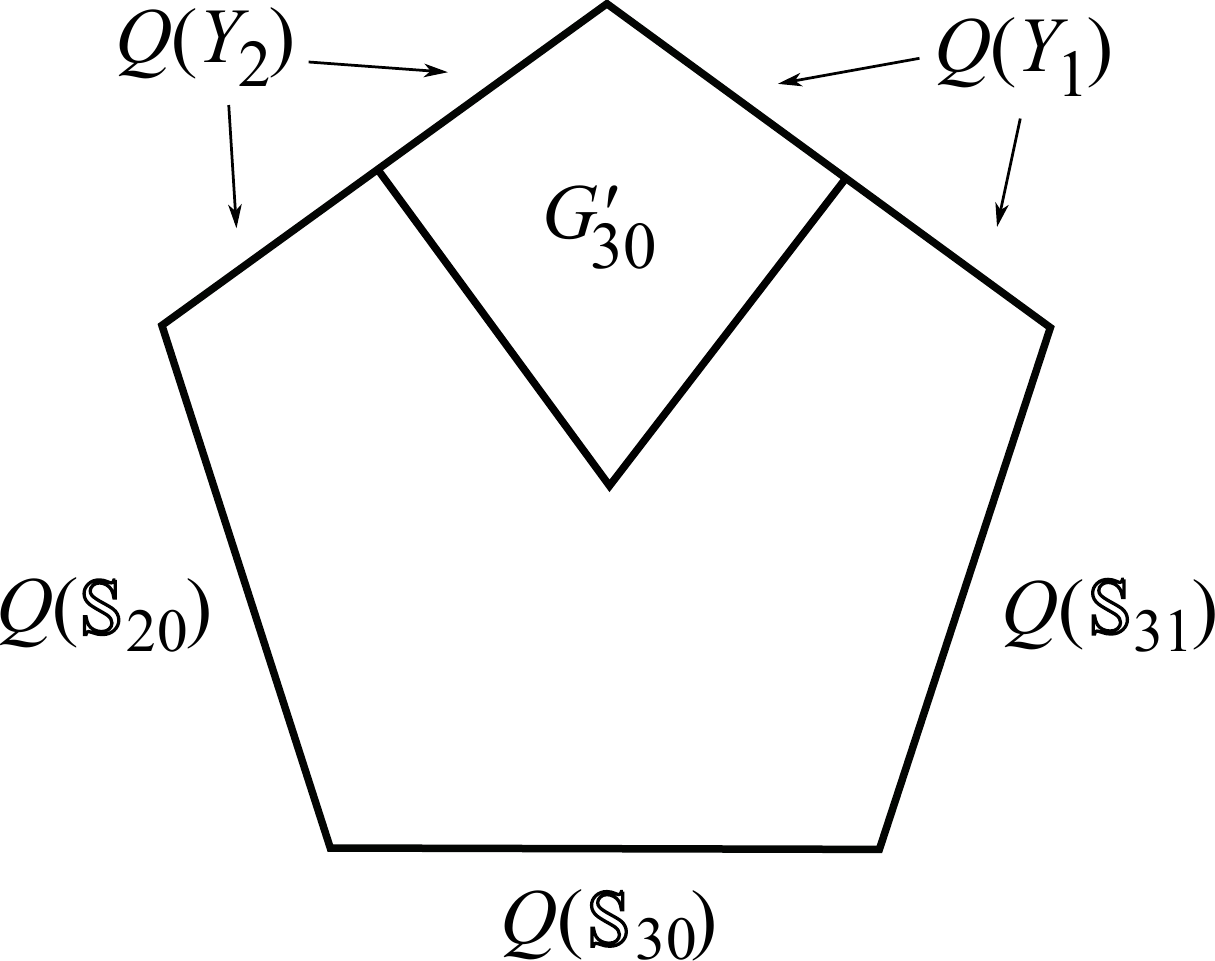}
    \end{center}
    \caption{\label{fig:pentagon}
    The family of metrics $\bG_{30}$ containing the image of family $G'_{30}$.}
\end{figure}

Turning now to the case of arbitrary $N$, we proceed as we did in the
previous subsection. That is, we suppose that we have $w$ and $u$ in
$\Z^{N}$, with $w_{N}-u_{N}=3$ and $w_{j}-u_{j}=0$ or $1$ for
$j<N$. The space of metrics $G'_{wu}$ is the product of $\R^{m-1}$
with a $2$-dimensional 
quadrant, where $m$ is again the number of coordinates in which
$w$ and $u$ differ:
\[
      G'_{wu}  = \R^{m-1} \times G'_{30}.
\]
If we write $w= w'3$ and $u = u'0$ with $w', u' \in
\Z^{N-1}$, then we can identify the $\R^{m-1}$ here with $G_{w'u'}$.
We extend $G'_{wu}$ to a larger family
\[
\begin{aligned}
    G_{wu} &= \R^{m-1} \times G_{30} \\
           &= G_{w'u'} \times G_{30}
\end{aligned}
\]
where $G_{30}\supset G'_{30}$ is the interior of the 
pentagon just described, and we set
$\breve{G}_{wu} = G_{wu}/\R$. By suitably normalizing the coordinates,
we can choose to identify
\[
   \breve{G}_{wu} = G_{w'u'} \times \breve G_{30}.
\]
We can complete $\breve{G}_{wu}$ to a family of broken Riemannian
metrics $\bG^{+}_{wu}$ whose codimension-1 faces are as follows.
\begin{enumerate}
\item First, the faces of the form $\breve{G}_{wv} \times \breve{G}_{vu}$ with
    $w > v > u$, parametrizing metrics broken at $(Y, K_{v})$. These
    we further subdivide as:
\begin{enumerate}
\item \label{item:vN=3} the cases with $v_{N}=w_{N}$ ;
\item\label{item:vN=0} the cases with $v_{N}=u_{N}$ ;
\item \label{item:vN=2} the cases with $v_{N}=w_{N}-1$ (these correspond to the edge
    $Q(Y_{2})$ of the pentagon, in the case $N=1$);
\item \label{item:vN=1} the cases with $v_{N}=w_{N}-2$ (these correspond to the edge
    $Q(Y_{1})$ of the pentagon, in the case $N=1$).
\end{enumerate}
\item \label{item:S-other} Second, the faces of the form $G_{w'u'} \times Q(\SS)$ for $\SS=\SS_{31},
    \SS_{20}$ or $\SS_{30}$.
\end{enumerate}

Our chosen orientation of $\breve G'_{wu}$  determines an
orientation for the larger space $\breve{G}_{wu}$. Over the
compactification $\bG^{+}_{wu}$ we have moduli spaces
$\Mubk_{wu}(\alpha,\beta)$ as usual. Mimicking
Definition~\ref{def:jvu-def},
we define the components of $\bK$ as follows:

\begin{definition}\label{def:kvu-def}
    Given $v \ge u$ in $\Z^{N}$ with $v_{N}- u_{N}=3$ and $v_{j}-u_{j}
    \le 1$ for $j<N$, we define
\[
         k_{vu} : C_{v} \to C_{u}
\]
     by declaring the matrix entry from $\beta$ to $\alpha$ to be the
     signed count of the points in the zero-dimensional 
    moduli space $\Mu_{vu}(\beta,\alpha)_{0}$ (if any), adjusted by
    the overall sign $(-1)^{\msign(v,u)}$, where $\msign(v,u)$ is again defined
    by the formula \eqref{eq:m-formula}.
\CloseDef
\end{definition}

The map $ \bK_{30} :  \bC_{3} \to \bC_{0}$ is defined in terms of
these $k_{vu}$ by
     \begin{equation}\label{eq:bK-pieces}
       \bK_{30} = \sum k_{(v'3)(u'0)} .
\end{equation}
The last stage of the argument is now to prove the formula \eqref{eq:second-homotopy}:

\begin{proposition}
    The anti-chain-map \[ \bF_{00}\bK_{30} + \bK_{30}\bF_{33} + 
                   \bF_{10}\bJ_{31} + \bJ_{20}\bF_{32}\] from
                   $\mathbf{C}_{3}$ to $\mathbf{C}_{0}$ is
                   chain-homotopic to $\pm 1$.
\end{proposition}

\begin{remark}
    Our definitions mean that $\mathbf{C}_{0}$ and $\mathbf{C}_{3}$
    are the same group, but the differential $\bF_{33}$ is
    $-\bF_{00}$, because of the sign $(-1)^{\msign(v,u)}$ in \eqref{eq:m-formula}.
\end{remark}

\begin{proof}
    Let $w \ge u$ be given, with $w=(w',3)$ and $u = (u',0)$, with
    $w', u' \in \{0,1\}^{N-1}$.
    We must prove a formula of the shape:
\begin{equation}\label{eq:k-homotopy-pieces}
    \begin{aligned}
        \sum_{\{ v \mid v_{N}=0\}} f_{vu} k_{wv}  + \sum_{\{ v \mid
          v_{N}=3\}} k_{vu} f_{wv} &+ \sum_{\{ v \mid v_{N}=1\}} f_{vu}
        j_{wv} \\ & + \sum_{\{ v \mid v_{N}=2\}} j_{vu} f_{wv} + \pm n_{wu}
        = 0
    \end{aligned}
\end{equation}
where $n_{wu}$ are the components of a map $\mathbf{N}$
chain-homotopic to $\pm 1$ from $\mathbf{C}_{3}$ to
$\mathbf{C}_{0}$. As usual, the proof goes by equating the
matrix-entry of the left-hand side, from $\gamma$ to $\alpha$, with
the number of ends of an oriented $1$-manifold, in this case the
$1$-manifold
\[
    \Mubk_{wu}(\gamma,\alpha)_{1}.
\]
As such, the above formula has again the same architecture as the
general chain-homotopy formula  \eqref{eq:chain-homotopy-faces}. In
the latter formula, the terms $\partial \circ m_{G}$ and
$m_{G}\circ\partial$ correspond to 
special cases of the first two terms of
\eqref{eq:k-homotopy-pieces}, 
of the special form
\[
    f_{uu} k_{wu} \qquad\text{or} \qquad k_{wu} f_{ww}.             
\]
With the exception of these terms and the term $\pm n_{wu}$, the terms
in \eqref{eq:k-homotopy-pieces} in the four summations are the
contributions from the first four types of faces of $\bG^{+}_{wu}$. Specifically,
the case \ref{item:vN=3} gives
rise to the terms $k_{vu}f_{wv}$ with $w>v>u$ in
\ref{eq:k-homotopy-pieces}; the case \ref{item:vN=0} gives rise
similarly to the terms $f_{vu}k_{wv}$; the cases \ref{item:vN=2} and
\ref{item:vN=1} provide the terms $j_{vu}f_{wv}$ and
$f_{vu}j_{wv}$. 

The terms from faces \ref{item:S-other} of type
$G_{w'u'}\times Q(\SS_{31})$ and $G_{w'u'}\times Q(\SS_{20})$ are all
zero, for the same reason as in the previous subsection: for these
families of broken metrics we have pulled off a connect-summand
$(S^{4}, \RP^{2})$.

What remains is the contribution corresponding to the face of the form
$G_{w'u'}\times Q(\SS_{30})$. We will complete the proof of the lemma by
showing that these contributions are
the matrix entries of a map $n_{wu}$ which is chain-homotopic to $\pm
1$. That is, we define $n_{wu}(\gamma,\alpha)$ by counting with sign the
ends of $M_{wu}(\gamma,\alpha)_{1}$ which lie over this face; we
define $n_{wu}$ to be the map with matrix entries
$n_{wu}(\gamma,\alpha)$, and we define $\mathbf{N}$ to be the map
$\mathbf{C}_{3}$ to $\mathbf{C}_{0}$ whose components are the
$n_{wu}$. With this understood, we then have
\[
   \bF_{00}\bK_{30} + \bK_{30}\bF_{33} + 
                   \bF_{10}\bJ_{31} + \bJ_{20}\bF_{32}  \pm \mathbf{N}
                   = 0.
\]
From this it follows formally that $\mathbf{N}$ is an anti-chain map; and
to complete the proof of the proposition, we must show:
\begin{equation}
    \label{eq:to-prove-N}
     \text{\itshape{the map $\mathbf{N}$ is chain-homotopic to the identity.}}
\end{equation}

The face $G_{w'u'}\times Q(\SS_{30})$ parametrizes metrics on a broken
Riemannian manifold with two components, obtained by cutting along
$\SS_{30}$. Recall that $\SS_{30}$ is a $3$-sphere meeting the embedded
surface $\Tigma_{wv}$ in a $2$-component unlink. One component of the
broken manifold is the pair
$(B_{30}, M_{30})$, equipped with a cylindrical end, were $B_{30}$ is
the standard $4$-ball described above: it contains the embedded surface $M_{30}$
obtained by plumbing two M\"obius bands. The second component has
three cylindrical ends: we denote it by $(W', \Tigma')$, and it is 
obtained by
removing $B_{30}$ from $(\R\times Y, \Tigma_{wu})$ 
and attaching a cylindrical end. The manifold-pair
$(B_{30},M_{30})$ carries the $1$-parameter family of metrics
$Q(\SS_{30})$, obtained by stretching along $\SS_{20}$ or $\SS_{31}$ as
$T\to -\infty$ or $+\infty$ respectively, while the cobordism $W'$
with cylindrical ends carries a family of metrics $G_{w'u'}$. The
dimension of $G_{w'u'}$ is equal to $|w'-u'|_{1}$. 

As in Lemma~\ref{lem:two-disks-product}, we consider now the cobordism 
$(\bar W, \bar\Tigma)$ obtained from $(W' ,\Tigma')$ by
attaching to the $\SS_{30}$ end a pair $(B_{30}, \Delta)$, where
$\Delta$ is a pair of standard disks in the 4-ball $B_{30}$ having
boundary the unlink.  The manifold $\bar W$ is topologically a
cylinder on $Y$, and in the case $N=1$ the embedded surface
$\bar\Tigma$ is also a trivial, cylindrical cobordism from
$K_{3}$ ($=K_{0}$) to $K_{0}$, as Lemma~\ref{lem:two-disks-product}
states. 
For larger $N$, we can identify
$\bar\Tigma$ with the cobordism standard $\Tigma_{\bar{w}u}$ from
$K_{\bar{w}}$ to $K_{u}$, where $\bar{w} = w'0$ and $u=u'0$. (So
$K_{\bar w}$ is the same link as $K_{w}=K_{w'3}$.) 

\begin{lemma}
   The cobordism $(I\times Y, \Tigma_{\bar{w}u})$ from $K_{\bar{w}}$
   to $K_{u}$, equipped with the family of metrics $G_{\bar{w}u}$,
   gives rise to the identity map from $C_{\bar{w}}$ to $C_{u}$ in the
   case $\bar{w}=u$ and the zero map otherwise. 
\end{lemma}

\begin{proof}
    The family of metrics $G_{\bar{w}u}$ includes the redundant $\R$
    factor, so the induced map counts only translation-invariant
    instantons. These exist only when $\bar{w}=u$, in which case they
    provide the identity map.
\end{proof}

In light of the lemma, we can prove the
assertion~\eqref{eq:to-prove-N}, if we can show that
$\mathbf{N}$ is chain-homotopic (up to an overall sign) to the map
obtained from the cobordism $\bar\Tigma$ with the family of metrics
$G_{\bar{w}u}$. We will do this by introducing a third map,
$\mathbf{N}'$, whose components $n_{wu}'$ count solutions on the
pair $(W' ,\Tigma')$ with its three cylindrical ends. 

To define $\mathbf{N}'$ in more detail, recall again that the third end of
this pair is a cylinder on the pair $(\SS_{30}, \partial \Delta)$, which
is a 2-component unlink.  For the pair $(\SS_{30},\partial \Delta)$, 
the critical points comprise a closed interval 
  \[ \Crit(\SS_{30}, \partial \Delta) = [0,\pi]. \] 
To
see this, note that the fundamental group of the link complement is free on
two generators and we are looking at homomorphisms from this free
group to $\SU(2)$ which send each generator to a point in the
conjugacy class of our preferred element \eqref{eq:bi}. This
conjugacy class is a $2$-sphere, and (up to conjugacy) the
homomorphism is determined by the great-circle distance between the
images of the two generators. In this closed interval, the interior
points represent irreducible representations, while the two endpoints
are reducible. To be more precise, in 
order to identify the critical points with $[0,\pi]$
in this way, we need to choose a relative orientation of the two
components of the unlink $\partial \Delta$, because without any
orientations the two generators of the free group are well-defined
only up to sign. Changing our choice of orientation will change our
identification by flipping the interval $[0,\pi]$.

For a generic perturbation of the equations, any solution
on $(W' ,\Tigma')$ lying in a zero-dimensional moduli space is
asymptotic to a critical point in the interior of the interval on this
end \cite[Lemma 3.2]{K-obstruction}. We define $n'_{wu}(\gamma,\alpha)$
by counting these solutions over the family of metrics $G_{w'u'}$, 
and we define $\mathbf{N}'$ as usual in
terms of its components $n'_{wu}(\gamma,\alpha)$.

Each critical point on $(\SS_{30},\partial\Delta)$ extends uniquely to a
flat connection on the pair $(B_{30},\Delta)$. So we can regard
$\mathbf{N}'$ also as obtained by counting solutions on the broken
manifold with two pieces: $(W',\Tigma')$ and $(B_{30},\Delta)$, with
their cylindrical ends. Since this broken manifold is obtained in turn
from $(I\times Y, \Tigma_{\bar{w}u})$ by stretching across $\SS_{30}$,
we see by an argument similar to the previous ones that $\mathbf{N}'$ is
chain-homotopic to the map arising from the cobordism $(I\times Y,
\Tigma_{\bar{w}u})$ with its family of metrics $G_{\bar{w}u}$: i.e.~to
the identity map, by the lemma. (See also
\cite[section~3.3]{K-obstruction}.) All that remains now is to prove:
\begin{equation}
     \text{$\mathbf{N} = \mathbf{N}'$, \itshape{up to an overall sign}.}
\end{equation}

The components $n'_{wu}(\gamma,\alpha)$ of $\mathbf{N}'$ count the
points of the moduli spaces $M(W',\Tigma' ; \gamma, \alpha)_{0}$ on the
three-ended manifold $W'$, and as stated above, this moduli space
comes with a map to the space of critical points on the $\SS_{30}$ end:
\[
\begin{aligned}
    r: M(W',\Tigma' ; \gamma, \alpha)_{0} &\to
    \Crit(\SS_{30}, \partial \Delta)  \\
    &= [0,\pi].
\end{aligned}
\]
The components $n_{wu}(\gamma,\alpha)$ of $\mathbf{N}$ on the other
hand count the points of a fiber product of the map $r$ with a map
\[
        s: M_{Q(\SS_{30})}(B_{30},M_{30})_{1} \to  \Crit(\SS_{30}, \partial \Delta)
\]
where the left-hand side is the $1$-dimensional part of the moduli
space on the pair $(B_{30}, M_{30})$ equipped with a cylindrical end
and carrying the $1$-parameter family of metrics $Q(\SS_{30})$. To show
that $\mathbf{N} = \mathbf{N}'$ up to sign, it suffices to show that the map
$s$ is a proper map of degree $\pm 1$ onto the interior of the
interval $[0,\pi]$. (The actual sign here depends on a choice of
orientation for the moduli space  $M_{Q(\SS_{30})}(B_{30},M_{30})$.) 

The two ends of the family of metrics $Q(\SS_{30})$ on $(B_{30},
M_{30})$ correspond to two different connected-sum decompositions of
$(B_{30}, M_{30})$, both of which have the form
\[
          (B_{30}, M_{30}) = (B_{30}, A) \# (S^{4}, \RP^{2})
\]
where $A$ is a standard annulus in the 4-ball, with $\partial
A=\partial \Delta$, and $\RP^{2}$ is (as before) a standard $\RP^{2}$
with positive self-intersection. Since these are two different
decompositions, we really have two different annuli $A$ involved here;
so we should write the first summand as $(B_{30}, A_{+})$ or $(B_{30},
A_{-})$ to distinguish the two cases.  The two annuli can be
distinguished as follows: either annulus determines a preferred
isotopy-class of diffeomorphisms between its two boundary components
(the two components of the unlink $\partial  \Delta$); but the annuli
$A_{+}$ and $A_{-}$ determine isotopy classes of diffeomorphisms with
the opposite orientation.

Considering the gluing problem for
this connected sum, we see that the parametrized moduli
space has two ends, one for each end of the parameter space
$Q(\SS_{30})$, and that each end is obtained by gluing the standard
irreducible solution on $(S^{4}, \RP^{2})$ to a flat, reducible
connection on $(B_{30},A_{\pm})$. The limiting value of the map $s$ on
the two ends is equal to critical point in
$\Crit(\SS_{30},\partial\Delta)$ arising as the restriction to
$(\SS_{30},\partial\Delta)$ of the unique flat solution on
$(B_{30},A_{\pm})$. In each case, this value is one of the two ends of
the interval $[0,\pi]$; and if $A_{+}$ gives rise to the endpoint $0
\in [0,\pi]$, then $A_{-}$ will give rise to the endpoint $\pi$,
because the two different annuli provide identifications of the two
boundary components that differ in orientation, as explained above.
\end{proof}

\subsection{\texorpdfstring{The absolute $\Z/4$ grading}{The absolute
    Z/4 grading}}

We return briefly to Theorem~\ref{thm:big-cube}, which expresses the
existence of a quasi-isomorphism between two complexes. The
complex $(C_{w}, f_{ww})$ is just the complex that computes
$I^{\omega}(K_{w})$, to within an immaterial change of sign in
$f_{ww}$; so this complex carries a relative $\Z/4$ grading. In the
spirit of Proposition~\ref{prop:canonical-Z/4}, we can fix absolute
$\Z/4$ gradings on all the complexes $C_{v}$, for $v\in \Z^{N}$, in
such a way that the maps $f_{vv'}$ when $|v-v'|_{1}=1$ have degree
\[
        -\chi(\Tigma_{vv'}) - b_{0}(K_{v}) + b_{0}(K_{v'})
         =  1 - b_{0}(K_{v}) + b_{0}(K_{v'}).
\]
For general $v \ge u$, let us also write
\[
\begin{aligned}
    \iota(v,u) &:= -\chi(\Tigma_{vu}) - b_{0}(K_{v}) +
    b_{0}(K_{u}) \\
    &= |v-u|_{1} - b_{0}(K_{v}) + b_{0}(K_{u})
\end{aligned}
\]
Let us denote by $\shift{n}$ a shift of
grading by $n$ mod $4$, so that if $A$ has a generator in degree $i$ then
$A\shift{n}$ has a generator in degree $i-n$. Then the cobordism
$\Tigma_{vu}$ equipped with just a fixed metric (not a family)
induces a chain map of degree $0$,
\[
            C_{v} \to C_{u}\shift{\iota(v,u)}.
\]

Having fixed an absolute $\Z/4$ grading for $(C_{w}, f_{ww})$ in this
way, we can ask how we may grade the other complex $(\bC[vu],
\bF[vu])$ in Theorem~\ref{thm:big-cube} so that the quasi-isomorphism
respects the $\Z/4$ grading. 
Let us then refine the
definition of $\bC[vu]$ by specifying a grading mod $4$:
\[
 \bC[vu] =  \bigoplus_{  v\ge v' \ge u} C_{v'}\shift{j(v')},
\]
where
\[
\begin{aligned}
    j(v') &= -\iota(v',w) - |v'-u|_{1} \\
    &= -\iota(v',u) - |v'-u|_{1} - n + b_{0}(K_{u}) - b_{0}(K_{w})
\end{aligned}
\]
where $n=|v-u|_{1}$ (the dimension of the cube).
With this definition, it is easily verified that the differential
$\bF[vu]$ has degree $-1$. So $(\bC[vu], \bF[vu])$ is another
$\Z/4$-graded complex. We then have the following refinement of the
theorem:

\begin{proposition}\label{prop:Z4-grading-cube}
    If $C_{w}$ and $\bC[vu]$ are given absolute $\Z/4$ gradings as
    above, then the quasi-isomorphism of Theorem~\ref{thm:big-cube}
    becomes a quasi-isomorphism
    \[
                  C_{w}   \to  \bC[vu] 
    \]
    of $\Z/4$ graded complexes.
\end{proposition}

\begin{proof}
    The quasi-isomorphism is exhibited in the proof of
    Theorem~\ref{thm:big-cube} as the composite of $n$ maps,
    each of which (as is easy to check) has a well-defined $\Z/4$
    degree. The composite map has a component
    \[
           C_{w} \to C_{v}\shift{j(v)} \subset \bC[vu] 
     \]
    which is the map induced by the cobordism $\Tigma_{wv}$. The Euler
    number of this surface is $-n$, so the map
    $C_{w}\to C_{v}$ induced by $\Tigma_{wv}$ has degree
    \[
          \iota(w,v)=    n - b_{0}(K_{w}) + b_{0}(K_{v})
     \]
    with respect to the original $\Z/4$ gradings. This last quantity
    coincides with $j(v)$; so the map has degree $0$ as a map
     \[
             C_{w} \to C_{v}\shift{j(v)} \subset \bC[vu].
     \]
\end{proof}

\section{\texorpdfstring{Unlinks and the $E_{2}$ term}{Unlinks and the
  E-2 term}}
\label{sec:unlinks-e-2}

\subsection{Statement of the result}

We now turn to classical knots and  links $K$, and invariants
$\Inat(K)$ and $\Isharp(K)$ introduced in
section~\ref{subsec:classical} 
above.  
We will focus on the \emph{unreduced} version, $\Isharp(K)$, and
return to the reduced version later. Recall that we have defined
\[
         \Isharp (K)  = I^{\omega}(S^{3}, K \amalg \Hopf)
\]
where $K$ is regarded as a link in $\R^{3}$ and $\Hopf$ is a standard Hopf
link near infinity, with $\omega$ an arc joining the components of $\Hopf$.
From this definition, it is
apparent that the results of section~\ref{sec:cubes} apply equally
well to the invariant  $\Isharp(K)$ as they do to $I^{\omega}(Y,K)$ in
general. Thus for example, if $K_{2}$, $K_{1}$ and $K_{0}$ are links
in $S^{3}$ which differ only inside a single ball, as in 
Figure~\ref{fig:Tetrahedra-skein}, then there is a skein exact sequence
\[
          \cdots \to \Isharp(K_{2}) \to \Isharp(K_{1}) \to
          \Isharp(K_{0}) \to \cdots
\]
in which the maps are induced by the elementary cobordisms
$\Tigma_{21}$ etcetera. More generally, we can consider again a
collection of links $K_{v}$ indexed by $v\in\{0,1,2\}^{N}$ which
differ by the same unoriented skein relations in a collection of $N$
disjoint balls in $\R^{3}$. From Corollary~\ref{cor:spectral-sequence}
we obtain:

\begin{corollary}\label{cor:spectral-sequence-hat}
For links $K_{v}$ as above, there is a spectral sequence whose $E_{1}$ term is
\[
          \bigoplus_{v \in \{0,1\}^{N}} 
                \Isharp(K_{v})
\]
and which abuts to the instanton Floer homology $\Isharp(K_{v})$,
for $v=(2,\dots,2)$.  The differential $d_{1}$ is the sum of the maps
induced by the cobordisms $\Tigma_{vu}$ with $v>u$ and $|v-u|=1$,
equipped with $I$-orientations satisfying the conditions of
Lemma~\ref{lem:orientation-consistency} and corrected by the signs
$(-1)^{\tilde{\delta}(v,u)}$ as given in
Corollary~\ref{cor:spectral-sequence-signs}.  \qed
\end{corollary}

Let $K$ be a link in $\R^{3}\subset S^{3}$ with a planar projection
giving a diagram $D$ in $\R^{2}$. Let $N$ be the number of crossings
in the diagram. As in \cite{Khovanov}, we can consider the $2^{N}$
possible smoothings of $D$, indexed by the points $v$ of the cube
$\{0,1\}^{N}$,  with the conventions of
\cite{Khovanov,Rasmussen-slice}, for example. This labeling of the
smoothings is consistent with the convention illustrated in
Figure~\ref{fig:Tetrahedra-skein}. This gives  $2^{N}$ different
unlinks $K_{v}$. For each $v\ge u$ in $\{0,1\}^{N}$,
we have our standard cobordism $\Tigma_{vu}$ from $K_{v}$ to $K_{u}$.

We can consistently orient all the links $K_{v}$, for
$v\in\{0,1\}^{N}$, and all the cobordisms $\Tigma_{vu}$, so that
$\partial \Tigma_{vu} = K_{u} - K_{v}$.  To do this, start with a
checkerboard coloring of the regions of the diagram $D$, and simply
orient each $K_{v}$ so that, away from the crossings and their
smoothings, the orientation of $K_{v}$ agrees with the boundary
orientation of the black regions of the checkerboard
coloring. We
can then give each $\Tigma_{vu}$ the $I$-orientation is obtains as an
oriented surface. The resulting
$I$-orientations respect composition: they satisfy the conditions of
Lemma~\ref{lem:orientation-consistency}.

We therefore apply Corollary~\ref{cor:spectral-sequence-hat} to this
situation. We learn that there is a spectral sequence abutting to
$\Isharp(K)$ whose $E_{1}$ term is
\[
   E_{1}= \bigoplus_{v\in \{0,1\}^{N}} \Isharp(K_{v}).
\]
In this sum, each $K_{v}$ is an unlink. The differential $d_{1}$ is
\begin{equation}
\label{eq:d1-sum}
     d_{1} = \sum_{v\ge u} (-1)^{\tilde{\delta}(v,u)} \Isharp(\Tigma_{vu}),
\end{equation}
where each cobordism $\Tigma_{vu}$ is obtained from a ``pair of
pants'' that either joins two components into one, or splits one
component into two.
We can consider the spectral
sequence of Corollary~\ref{cor:spectral-sequence-hat} in this
setting, about which we have the following result.

\begin{theorem}\label{thm:E2-is-Kh}
    In the above situation, the page $(E_{1},d_{1})$ of the spectral sequence
    furnished by Corollary~\ref{cor:spectral-sequence-hat}     
    is isomorphic (as an abelian group with differential) to the
    complex that computes the Khovanov cohomology of $\bar{K}$ (the
    mirror image of $K$) from the given diagram.
    Therefore, the $E_{2}$ term of the spectral sequence
    is
    isomorphic to the Khovanov cohomology of $\bar{K}$.  
    The spectral sequence abuts to the
    instanton homology $\Isharp(K)$.
\end{theorem}

\begin{remarks}
    The relation expressed by this theorem, between $\Isharp(K)$ and
    $\kh(\bar{K})$, pays no attention to the bigrading that is carried
    by $\kh(\bar{K})$. It is natural to ask, for example, whether at
    least the filtration of $\Isharp(K)$ that arises from the
    spectral sequence is a topological invariant of $K$. More
    generally, one can ask whether the intermediate pages of the
    spectral sequence, as filtered groups, are invariants of $K$ (see
    \cite{Baldwin} for the similar question concerning the spectral
    sequence of Ozsv\'ath and Szab\'o). A related question is whether
    the intermediate pages are functorial for knot
    cobordisms. Although the bigrading is absent, there is at least a
    $\Z/4$ grading throughout: carrying
    Proposition~\ref{prop:Z4-grading-cube} over to the present
    situation, we see that there is a spectral sequence of $\Z/4$
    graded groups abutting to $\Isharp(K)$ whose $E_{1}$ term is
    \begin{equation}
        \label{eq:E1-shift}
   E_{1}= \bigoplus_{v\in \{0,1\}^{N}} \Isharp(K_{v})\shift{k_{v}},
    \end{equation}
    with
    \begin{equation}
        \label{eq:k-shift}
        k_{v} = - b_{0}(K_{v}) + 2 b_{0}(K_{0}) - b_{0}(K) - N.
    \end{equation}
    In deriving this formula from the formula for $j(v')$, we have
    used the fact that the cobordisms between the different smoothings
    are all orientable, which implies that $|v-v'|_{1} = b_{0}(v) -
    b_{0}(v')$ mod $2$. The formula for $k_{v}$ can also be written
    (mod $4$) as
    \[
         k_{v} = -b_{0}(K_{v}) + b_{0}(K) - N_{-} + N_{+},
    \]
    where $N_{-}$ and $N_{+}$ are the number of positive and negative
    crossings in the diagram. From this version of the formula, it is
    straightforward to compare our $\Z/4$ grading to the bigradings in
    \cite{Khovanov}. The result is that the $\Z/4$-graded $E_{2}$ page
    of our spectral sequence is isomorphic to $\kh(\bar{K})$ with the
    $\Z/4$-grading defined by
    \[
                     q-h - b_{0}(K)
     \]
    where $q$ and $h$ are the $q$-grading and homological grading
    respectively. That is, the part of the $E_{2}$ term in
    $\Z/4$-grading $\alpha$ is
    \[
                \bigoplus_{j-i-b_{0}(K) = \alpha} \kh^{i,j}(\bar{K}).
    \]
\end{remarks}

Understanding the $E_{1}$ page and the differential $d_{1}$  means
computing $\Isharp(K_{v})$ for an unlink $K_{v}$ and computing the
maps given by pairs of pants. We take up these calculations in
the remaining parts of this section.

\subsection{Unlinks}

We write $U_{n}$ for the unlink in $\R^{3}$ with $n$ components, so
that $U_{0}$ is the empty link and $U_{1}$ is the unknot. We take
specific models for these. For example, we may take $U_{n}$ to be the
union of standard circles in the $(x,y)$ plane, each of diameter
$1/2$, and centered on the first $n$ integer lattice points along the
$x$ axis; and we can then orient the components of $U_{n}$ as the
boundaries of the standard disks that they bound.

Given any subset $\bi=\{i_{1},\dots,i_{m}\}\subset \{1,\dots,n\}$,
there is a corresponding $m$-component sublink $U_{\bi}$ of
$U_{n}$. We will identify $U_{\bi}$ with $U_{m}$ in a standard way,
via a self-evident isotopy (preserving the ordering of the components).

We have
already seen that $\Isharp(U_{0})$ is $\Z$
(Proposition~\ref{prop:isharpemptyset}). 
There are two possible identifications of
$\Isharp(U_{0})$ with $\Z$, differing in sign. We fix one of them,
once and for all by specifying a generator
\[
       \bu_{0} \in \Isharp(U_{0}),
\]
 so that
\[
              \Isharp(U_{0})=\Z.
\]
This $\Z$ occurs in grading $0$
mod $4$, by convention. 

\begin{lemma}\label{lem:unknot-1}
    For the unknot $U_{1}$, the instanton homology $\Isharp(U_{1})$ is free of rank
    $2$,
    \[
             \Isharp(U_{1}) \cong \Z \oplus \Z,
     \]
 with generators in degrees $0$ and $-2$ mod $4$.
 \end{lemma}

\begin{proof}
    Draw a diagram of the Hopf link $\Hopf$ with an extra crossing, so that
    by smoothing that crossing in two different ways one obtains the
    links $\Hopf$ (again) and $\Hopf \amalg U_{1}$ (see
    Figure~\ref{fig:U1-skein}).  
\begin{figure}
    \begin{center}
        \includegraphics[height=1.4in]{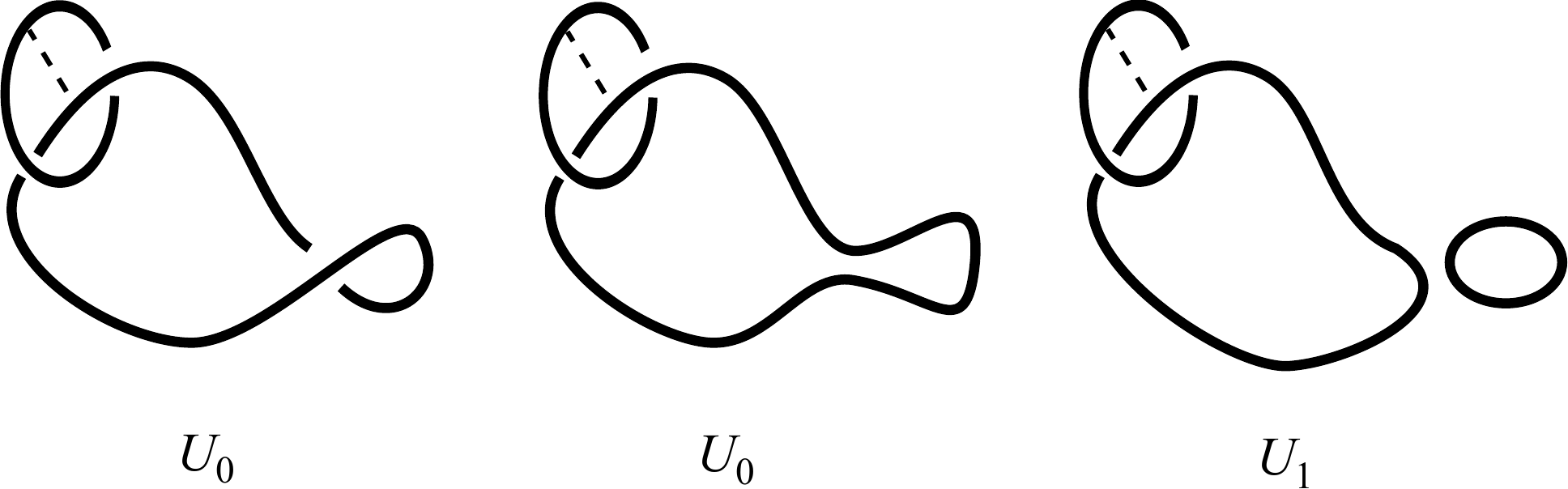}
    \end{center}
    \caption{\label{fig:U1-skein}
    A skein sequence relating $U_{0}$ (twice) and $U_{1}$.}
\end{figure}
  The skein sequence for this situation
    gives a long exact sequence
    \[
          \cdots \to \Isharp(U_{0}) \stackrel{a}{\to} \Isharp(U_{1})
          \stackrel{b}{\to} 
          \Isharp(U_{0}) \stackrel{c}{\to}
          \Isharp(U_{0}) \to \cdots,
    \]
   in which the maps $a$ and $b$ have degree $-2$ and $0$
   respectively, 
   while $c$ 
   has degree $1$. From our calculation of
   $\Isharp(U_{0})$ it follows that $c=0$ and that $\Isharp(U_{1})$ is
   free of rank $2$ with generators in degrees $0$ and $-2$.  The
   generator of degree $-2$ is the image of $a$, while the generator of
   degree $0$ is mapped by $b$ to a generator of $\Isharp(U_{0}) =
   \I^{\omega}(S^{3}, \Hopf)$.
\end{proof}

   We wish to have explicit generators of the rank-2 group
   $\Isharp(U_{1})$ defined without reference to the auxiliary Hopf
   link $H$. To this end, let $D$ be the standard disk that $U_{1}$
   bounds in the $(x,y)$ plane. Let $D^{+}$ be the oriented cobordism
   from the empty link $U_{0}$ to $U_{1}$ obtained by pushing the disk
   $D$ a little into the 4-dimensional cylinder $[-0,1]\times
   \R^{3}$. Similarly, let $D^{-}$ be the cobordism from $U_{1}$ to
   $U_{0}$ obtained from $D$ with its opposite orientation. These
   oriented cobordisms give preferred maps
   \[
   \begin{aligned}
       \Isharp(D^{+}) : \Isharp(U_{0}) &\to \Isharp(U_{1}) \\
       \Isharp(D^{-}) : \Isharp(U_{1}) &\to \Isharp(U_{0}) 
   \end{aligned}
    \]
     of degrees $0$ and $2$ respectively.

    \begin{lemma}
        There are preferred generators $\vp$ and $\vm$ for the
        rank-2 group $\Isharp(U_{1})$, in degrees $0$ and $2$ mod $4$
        respectively, characterized by the conditions
        \[
                   \Isharp(D^{+})(\bu_{0}) = \vp
        \]
         and
         \[
                  \Isharp(D^{-})(\vm) =  \bu_{0}
          \]
          respectively, where $\bu_{0}$ is the chosen generator for
          $\Isharp(U_{0})=\Z$.
    \end{lemma}

    \begin{proof}
        The proof of the previous lemma shows that a generator of the
        degree-0 part of $\Isharp(U_{1})$ is the image of the map
        $b$. So to show that $\vp$ as defined in the present lemma is
        a generator it suffices to show that the composite map
        $b\circ \Isharp(D^{+})$ is the identity map on the rank-1 group
        $\Isharp(U_{0}) = I^{\omega}(S^{3}, H)$. This in turn follows
        from the fact that the composite cobordism from the Hopf link
        $H$ to itself is a product.
        
        Similarly, to show that there is a generator $\vm$ of the
        degree-2 part with the property described, it suffices to show
        that the composite map $\Isharp(D^{-})\circ a$ is the identity
        map $\Isharp(U_{0})$. The composite cobordism is again a
        product, so the result follows.
    \end{proof}

\begin{corollary}\label{cor:unlink-tensor-n}
    Write  $V= \langle \vp, \vm \rangle \cong \Z^{2}$ 
    for the group $\Isharp(U_{1})$.
    We then have  isomorphisms of
    $\Z/4$-graded abelian groups,
    \[
        \Phi_{n}: V^{\otimes n} \to \Isharp(U_{n})       ,
     \]
      for all $n$,
     with the following  properties. First, 
     if $D^{+}_{n}$ denotes the cobordism from $U_{0}$ to $U_{n}$
     obtained from standard disks as in the previous lemma, then
     \[
                  \Isharp(D^{+}_{n})(\bu_{0}) = 
                           \Phi_{n}( \vp\otimes \dots \otimes \vp).
    \]
     Second, the isomorphism is natural for split cobordisms from
     $U_{n}$ to itself. Here, a ``split'' cobordism means a cobordism
     from $U_{n}$ to $U_{n}$ in $[0,1]\times \R^{3}$ which is the
     disjoint union of $n$ cobordisms from $U_{1}$ to $U_{1}$, each
     contained in a standard ball $[0,1]\times B^{3}$. 
\end{corollary}

\begin{proof}
    This follows from the general product formula for split links, 
    Corollary~\ref{cor:split-link}.
\end{proof}

\begin{remark}
    We will see later, in Proposition~\ref{prop:unlink-canonical},
    the extent to which this isomorphism is canonical.
\end{remark}

\subsection{An operator of degree 2}

Let $K$ be a link, let $p$ be a marked point on $K$, and let an
orientation be chosen for $K$ at $p$. We can then form a cobordism
$\Tigma$ from $K$ to $K$ by taking the cylinder $[-1,1] \times K$ and
forming a connect sum with a standard torus at the point $(0,p)$.  
This cobordism then
determines a map
\begin{equation}\label{eq:sigma-map}
           \sigma : \Isharp(K) \to \Isharp(K)
\end{equation}
of degree $2$ mod $4$.  

\begin{lemma}\label{lem:sigma-nilpotent}
    For any link $K$, and any base-point $p\in K$,  the map $\sigma$ is nilpotent.
\end{lemma}

\begin{proof}
    The map $\sigma$ behaves naturally with respect to cobordisms of
    links with base-points. So if $K_{2}$, $K_{1}$ and $K_{0}$ are
    three links related by the unoriented skein relation, then
    $\sigma$ commutes with the maps in the long exact sequence
    relating the groups $\Isharp(K_{i})$. From this it follows that if
    $\sigma$ is nilpotent on $\Isharp(K_{i})$ for two of the three links
    $K_{i}$ then it is nilpotent also on the third. By repeated use of
    the skein relation, we see that it is enough to check the case
    that $K=U_{1}$.  Finally, for $U_{1}$, we can use the exact
    sequence from the proof of Lemma~\ref{lem:unknot-1} to reduce to
    the case of $\Isharp(U_{0})$, or more precisely to the case of
    $I^{\omega}(S^{3}, \Hopf)$, with a marked point $p$ on one of the two
    components of the Hopf link $\Hopf$. This last case is trivial,
    however, because the group has a generator only in one of the
    degrees mod $4$, which forces $\sigma$ to be zero on $\Isharp(U_{0})$.
\end{proof}

\subsection{Pairs of pants}

Let $\pants$ be a pair-of-pants cobordism from $U_{1}$ to $U_{2}$. We
wish to calculate the corresponding map on instanton homology. 
Via the isomorphisms of Corollary~\ref{cor:unlink-tensor-n}, this map
becomes a map
\[
     \mcoprod:  V \to V\otimes V.
\]
The degree of this map is $-2$. 

    There is also a pair-of-pants cobordism $\copants$, from $U_{2}$
    to $U_{1}$, which induces a map
     \[
           \mprod : V\otimes V \to V
      \]
      of degree $0$.

\begin{lemma}
   In
    terms of the  generators $\vp$ and $\vm$ for $V$, 
  the map $\mcoprod$ is given by
  \[
  \begin{aligned}
      \vm &\mapsto  \vm\otimes \vm \\
       \vp & \mapsto  ( \vm \otimes \vp + \vp \otimes \vm )
  \end{aligned}
   \]
   and $\mprod$ is given by
   \[
   \begin{aligned}
       \vp \otimes \vp &\mapsto \vp \\
       \vp \otimes \vm &\mapsto \vm \\
       \vm \otimes \vp &\mapsto \vm \\
       \vm \otimes \vm &\mapsto  0 .    
   \end{aligned}
   \]
\end{lemma}

\begin{proof}
    We begin with  $\mcoprod(\vp)$. Because of the $\Z/4$ grading, we
    know that
   \[
     \mcoprod(\vp) = \lambda_{1}( \vm \otimes \vp) + \lambda_{2} (\vp
     \otimes \vm),
\]
  for some integers $\lambda_{1}$ and $\lambda_{2}$. Consider the
  composite cobordism $\pants_{0}=D^{+} \cup \pants$ from $U_{0}$ to $U_{2}$,
  formed from the pair of pants $\pants$ by attaching a disk to the
  incoming boundary component. The composite cobordism determines
  a map 
   \[
            \mcoprod_{0} : \Z \to V^{\otimes 2}
    \]
    with
       \[
     \mcoprod_{0}(\bu_{0}) = \lambda_{1}( \vm \otimes \vp) + \lambda_{2} (\vp
     \otimes \vm).
\]
     Next, form a cobordism $\pants_{1}=\pants_{0} \cup D^{-}$ from
     $U_{0}$ to $U_{1}$ by attaching a disk to the first of the two outgoing boundary
     components of $\pants_{0}$. As a cobordism from $U_{1}$ to
     $U_{0}$, the disk maps $\vm$ to $1$; so by the naturality
     expressed in Corollary~\ref{cor:split-link}, it maps
     $\vm\otimes \vp$ to $\vp$ when viewed as a cobordism from $U_{2}$
     to $U_{1}$. The composite cobordism $\pants_{1}$ therefore
     defines a map
     \[
           \mcoprod_{1}: \Z \to V
     \]
    with
       \[
     \mcoprod_{1}(\bu_{0}) = \lambda_{1} \vp.
\]
     But $\pants_{1}$ is simply a disk, with the same orientation as
     the original $\pants$; so $\mcoprod_{1}(\bu_{0}) = \vp$ by definition of $\vp$,
     and we conclude that $\lambda_{1}=1$. We also have
     $\lambda_{2}=1$ by the same argument, so we have computed $\mcoprod(\vp)$.

     As a preliminary step towards computing $\mcoprod(\vm)$, we
     compute the effect of the degree-2 operator $\sigma$ on
     $\Isharp(U_{1})$. The cobordism which defines $\sigma$ on $U_{1}$
     can be seen as a composite cobordism $U_{1} \to U_{1}$ obtained
     from two pairs of pants: first $\pants$ from $U_{1}$ to $U_{2}$,
     then $\copants$ from $U_{2}$ to $U_{1}$. These are oriented
     surfaces, so the cobordisms are canonically $I$-oriented, in a
     manner that is compatible with composition. By an argument dual
     to the one in the previous paragraph, we see that
     \[
               \mprod(\vp\otimes \vm) = \mprod(\vm\otimes \vp) = \vm.
     \]
     Looking at the composite $\mprod\comp \mcoprod$, we see that
     \[
               \sigma(\vp) =  2 \vm.
      \]
     We also know that $\sigma$ is nilpotent
     (Lemma~\ref{lem:sigma-nilpotent}), so we must have
     \[
          \sigma(\vm) = 0.
     \]
       
      We now appeal to the naturality of $\sigma$ with respect to
      cobordism of links with base-points. This tells us that
       \[
                 \mcoprod \comp \sigma = (1\otimes \sigma) \comp \mcoprod.
        \]
       In particular, 
       \[
                \mcoprod( \sigma(\vp)) = (1\otimes\sigma)(\mcoprod(\vp)).
       \]
        We have already calculated $\mcoprod(\vp)$ and $\sigma(\vp)$, so we have,
        \[
        \begin{aligned}
            \mcoprod(2 \vm) &= (1\otimes \sigma) (\vm \otimes \vp
            + \vp \otimes \vm) \\
            &=2( \vm\otimes \vm).
        \end{aligned}
        \]
        It follows that $\mcoprod(\vm) = \vm \otimes \vm$ as claimed. A
        dual argument determines the remaining terms of $\mprod$ in a
        similar manner.
\end{proof}

\subsection{Isotopies of the unlink}

\begin{lemma}
    Let $S \subset [0,1] \times \R^{3}$ be a closed, oriented surface,
    regarded as a cobordism from the empty link $U_{0}$ to
    itself. Then the induced map $\Isharp(S) :\Isharp(U_{0}) \to
    \Isharp(U_{0})$ 
     is
    multiplication by $2^{k}$ if $S$ consists of $k$ tori; and
    $\Isharp(S)$ is zero otherwise.
\end{lemma}

\begin{proof}
    The first point is that the map $\Isharp(S)$ in this situation 
    depends only on $S$ as an abstract surface, not on its embedding
    in $(0,1)\times \R^{3}$. This can be deduced from results
    of \cite{K-obstruction}, which show that the invariants of a closed
    pair $(X,\Sigma)$ defined using singular instantons depend on
    $\Sigma$ only through its homotopy class. To apply the results of
    \cite{K-obstruction} to the present situation, we proceed as
    follows. Since $\Isharp(U_{0})$ is $\Z$, the map $\Isharp(S)$ is
    determined by its trace; and twice the trace can be interpreted as
    the invariant of a closed pair $(X,\Sigma)$ obtained by gluing the
    incoming to the outgoing ends of the cobordism. Thus 
     $X$ is $S^{1}\times S^{3}$ and $\Sigma$ is the union of
     $S^{1}\times H$ and the surface $S$. The results of
     \cite{K-obstruction} can be applied directly to any homotopy of
     $S$ that remains in a ball disjoint from $S^{1}\times H$, and
     this is all that we need.

     Because of this observation, it now suffices to verify the
     statement in the case that $S$ is a standard connected, oriented
     surface of arbitrary genus.If we decompose a genus-1 surface $S$
     as an incoming disk $D^{+}$, a genus-1 cobordism from $U_{1}$ to itself,
     and an outgoing disk $D^{-}$, we find that $\Isharp(S)$ in this
     case is given by
     \[
                 \Isharp(S)(\bu_{0}) = \Isharp(D^{-})\comp \sigma
                 \comp \Isharp(D^{+})(\bu_{0}),
      \]
       which is $2\bu_{0}$ by our previous results. For the case of
       genus $g$, we look at
       \[
                \Isharp(S)(\bu_{0}) = \Isharp(D^{-})\comp \sigma^{g}
                 \comp \Isharp(D^{+})(\bu_{0}),
         \]
        which is zero for all $g$ other than $g=1$.
\end{proof}

\begin{lemma}
    Let $S$ be an oriented concordance from the standard unlink $U_{n}$ to
    itself, consisting of $n$ oriented annuli in $[0,1]\times
    \R^{3}$. 
    Let $\tau$ be the permutation of $\{1,\dots,n\}$
    corresponding to the permutation of the components of $U_{n}$
    arising from $S$. Then the standard isomorphism $\Phi_{n}$ of
    Corollary~\ref{cor:unlink-tensor-n} intertwines the map
     \[
              \Isharp(S) : \Isharp(U_{n}) \to \Isharp(U_{n})
     \]
     with the permutation map
     \[
                 \tau_{*} : V\otimes \dots\otimes V \to  V\otimes
                 \dots \otimes V.
      \]
      In particular, if the permutation $\tau$ is the identity, then
      $\Isharp(S)$ is the identity.
\end{lemma}

\begin{proof}
    We start with the case that the permutation $\tau$ is the
    identity. Let $\sigma_{i} : \Isharp(U_{n}) \to \Isharp(U_{n})$ be
    the map $\sigma$ applied with a chosen basepoint on the $i$'th
    component of the link. The isomorphism $\Phi_{n}$ intertwines
    $\sigma_{i}$ with
    \[
           1\otimes \dots\otimes 1\otimes \sigma \otimes
           1\otimes\dots\otimes 1
     \] 
     with $\sigma$ in the $i$'th spot, by
     Corollary~\ref{cor:unlink-tensor-n}. Furthermore, $\Isharp(S)$
     commutes with $\sigma_{i}$, because the corresponding
     cobordisms commute up to diffeomorphism relative to the
     boundary. Since $\sigma(\vp)=2\vm$, we therefore see that, to
     show $\Isharp(S)$ is the identity, we need only show that
     \[
                 \Isharp(S)(\vp \otimes \cdots \otimes \vp) 
                            =     \vp \otimes \cdots \otimes \vp.
     \]
     Let $\lambda_{m}$ be the coefficient of $\vp^{\otimes m} \otimes
     \vm^{\otimes(n-m)}$ in $\Isharp(S)(\vp^{\otimes n})$. (There is
     no loss of generality in putting the $\vp$ factors first here: it
     is only a notational convenience.) We must show that
     $\lambda_{m}=0$ for $m<n$ and $\lambda_{n}=1$. From our
     calculation of $\sigma$ etc., we see that we can get hold of
     $\lambda_{m}$ by the formula
      \[
                     \Isharp(D_{n}^{-}) \circ \sigma_{1}\circ\cdots\circ
                     \sigma_{m} \Isharp(D_{n}^{+})(\bu_{0}) = 2^{m} \lambda_{m} \bu_{0}.
      \]
       On the other hand, the composite map on the left is equal to
       $\Isharp(S)$, where $S$ is a closed surface consisting of $m$
       tori and $n-m$ spheres, viewed as a cobordism from $U_{0}$ to
       $U_{0}$. From the results of the previous lemma, we see that
       the left-hand side is $0$ if $m\ne n$ and is $2^{n}$ if
       $m=n$. This completes the proof in the case that the
       permutation $\tau$ is $1$.

       For the case that $S$ provides a non-trivial permutation $\tau$
       of the components, the map $\Isharp(S)$ intertwines
       $\sigma_{i}$ with $\sigma_{\tau(i)}$. It is again sufficient to
       show that $\Isharp(S)$ sends $\vp^{\otimes n}$ to itself, and
       essentially the same argument applies.
\end{proof}

As a special case of a cobordism from $U_{n}$ to itself, we can
consider the trace of an isotopy $f_{t} : U_{n} \to \R^{3}$
($t\in[0,1]$) which begins and ends with the standard inclusion. As an
application of the lemma, we therefore have:

\begin{proposition}\label{prop:unlink-canonical}
    Let $\cU_{n}$ be any oriented link in link-type of $U_{n}$, and
    let its components be enumerated. Then there is canonical
    isomorphism
    \[
             \Psi_{n} :  V\otimes \dots \otimes V \to \Isharp(\cU_{n})
    \]
    which can be described as $ \Isharp(S) \circ \Phi_{n}$, where
    $\Phi_{n}$ is the standard isomorphism of
    Corollary~\ref{cor:unlink-tensor-n} and $S$ is any cobordism from
    $U_{n}$ to $\cU_{n}$ arising from an isotopy from $U_{n}$ to
    $\cU_{n}$, respecting the orientations and the enumeration of
    the components. 
    
    If the enumeration of the components of $\cU_{n}$ is changed by a
    permutation $\tau$, then the isomorphism $\Psi_{n}$ is changed
    simply by composition with the corresponding permutation of the
    factors in the tensor product. \qed
\end{proposition}

The following proposition encapsulates the calculations of this
section so far.  

\begin{proposition}\label{prop:unlink-summary}
    Let $V$ be a $\Z/4$-graded free abelian group with generators
    $v_{-}$ and $v_{+}$ in degrees $-2$ and $0$. Then for each
    oriented $n$-component unlink
    $\cU_{n}$ with enumerated components $K_{1}, \dots, K_{n}$,
    there is a canonical isomorphism
    \[
                \Phi_{n}(\cU_{n}) : V\otimes \dots\otimes V \to \Isharp(\cU_{n})
    \]
    with the following properties.
    \begin{enumerate}
    \item Given an orientation-preserving 
         isotopy from $\cU_{n}$ to $\cU'_{n}$, respecting the
         enumeration of the components, the map $\Isharp(S)$ arising
         from the corresponding cobordism $S$ intertwines
         $\Phi_{n}(\cU_{n})$ with $\Phi_{n}(\cU'_{n})$.
    \item If the components of $\cU_{n}$ are enumerated differently,
        then $\Phi_{n}(\cU_{n})$ changes by composition with the
        permutation of the factors in the tensor product.
      \item If $\pants_{n}$ is the oriented cobordism from $\cU_{n}$ to
          some $\cU_{n+1}$  which attaches a pair of pants $\Pi$ to the last
          component, inside a ball disjoint from the other components, 
          then the isomorphisms $\Phi_{n}(\cU_{n})$ and
          $\Phi_{n}(\cU_{n+1})$ intertwine the corresponding map
          $\Isharp(\pants_{n}):\Isharp(\cU_{n}) \to \Isharp(\cU_{n+1})$ with the map
         \[
            1\otimes \dots \otimes 1 \otimes \mcoprod .
          \]
     \item Similarly, for an oriented cobordism $\cU_{n+1} \to \cU_{n}$ obtained
         using a pair of pants $\copants$ on the last two
         components, in a ball disjoint from the other components, we obtain the map
       \[
               1\otimes \dots \otimes 1 \otimes \mprod .
      \]
    \end{enumerate}
\qed
\end{proposition}
 
\subsection{Khovanov cohomology}

We now have all that we need to conclude the proof of
Theorem~\ref{thm:E2-is-Kh}. 
If we write $n(v)$ for the
number of components of $K_{v}$ and enumerate those components, then
we have a canonical identification
\[
    E_{1} = \bigoplus_{v\in \{0,1\}^{N}} V^{\otimes n(v)}.
\]
The differential $d_{1}$ is given by the sum \eqref{eq:d1-sum} 
The cobordism $\Tigma_{vu}$ is
a the union of some product cylinders and a single pair of pants,
either $\pants$ or $\copants$. Proposition~\ref{prop:unlink-summary}
therefore tells us that, after pre- and post-composing by
permutations of the components, the map $f_{vu}$ is given either by
\[
      (-1)^{\tilde{\delta}(v,u)} (1\otimes \dots\otimes 1\otimes \mcoprod) 
\]
or
\[
      (-1)^{\tilde{\delta}(v,u)} (1\otimes \dots\otimes 1\otimes \mprod).
\]
The complex $(E_{1}, d_{1})$ that one arrives at in this way is
exactly the complex that computes the Khovanov cohomology of $\bar{K}$. The
fact that the mirror $\bar{K}$ of the link $K$ appears in this statement is
accounted for  by the fact that, in Khovanov's definition, the
differential is the sum of contributions from the edges oriented so
that $|v|$ \emph{increases} along the edges, whereas in our setup the
differential $d_{1}$ \emph{decreases} $|v|$. It follows that the
$E_{2}$ page of the spectral sequence is isomorphic (as an abelian
group) to the Khovanov cohomology $\kh(\bar{K})$.

\subsection{The reduced homology theories}

Recall that for a link $K$ with a
marked point $x$ and normal vector $v$ at $x$,
we have defined
\[
      \Inat(K) = I^{\omega}(S^{3}, K \cup L  ),
\]
where $L$ is a meridional circle centered at $x$ and $\omega$ is an
arc in the direction of $v$. There is a skein exact sequence
(illustrated for the case of the unknot in Figure~\ref{fig:U1-skein}),
\[
       \cdots \to \Inat(K) \to \Isharp(K) \to \Inat(K) \to \Inat(K)
       \to \cdots. 
\]
Corollary~\ref{cor:spectral-sequence-hat} has a straightforward
adaptation to this reduced theory, which can again be deduced from the
more general result, Corollary~\ref{cor:spectral-sequence}.

The maps in the long exact sequence above have already been described
for the unknot $U_{1}$. Thus, the map
\[
         \Isharp(U_{1}) \to \Inat(U_{1})
\]
is the same as the map $\Isharp(U_{1}) \to \Isharp(U_{0}) = \Z$ given as
the quotient map
\[
             V\mapsto V/\langle v_{-} \rangle \cong \Z.
\]
For the unlink $U_{n}$ we similarly have
\[
         \Inat(U_{n}) = V\otimes \dots\otimes V \otimes V/\langle v_{-} \rangle,
\]
as a quotient of $V^{\otimes n}$. (The marked point is on the last
component here.) The maps $\mprod$ and $\mcoprod$ give rise to maps
\[
      \mprod_{r} :     V \otimes V/\langle v_{-} \rangle \to  V/\langle v_{-} \rangle
\]
and
\[
         \mcoprod_{r}: V/\langle v_{-} \rangle \to V \otimes V/\langle v_{-} \rangle ;
\]
and these are precisely the maps induced by the pair-of-pants
cobordisms $\copants$ and $\pants$.  

In the spectral sequence abutting to $\Inat(K)$, we can therefore
identify $E_{1}$ and $d_{1}$. The $E_{1}$ term is obtained from the
unreduced version by replacing (at each vertex of the cube) the factor $V$
corresponding to the marked component by a factor $V/\langle v_{-}
\rangle$. And the differential $d_{1}$ is obtained from the unreduced
case by replacing $\mprod$ or $\mcoprod$ by $\mprod_{r}$ or
$\mcoprod_{r}$ whenever the marked component is involved. The
resulting complex is precisely the complex that computes the reduced
Khovanov cohomology of the mirror of $K$. We therefore have:

\begin{theorem}
   For a knot or link $K$ in $S^{3}$, there is a spectral sequence of
    abelian groups whose $E_{2}$ term is the reduced Khovanov cohomology of
    $\bar{K}$ and which abuts to $\Inat(K)$. \qed
\end{theorem}

\bibliographystyle{abbrv}
\bibliography{khovanov-unknot}

\end{document}